\documentclass[11pt]{article}

\usepackage[utf8]{inputenc}
\usepackage[T1]{fontenc}
\usepackage[letterpaper,margin=1in]{geometry}

\usepackage{mathptmx}      % Times-like text/math
\usepackage{amsmath}
\usepackage{amssymb}
\usepackage{amsthm}
\usepackage{mathtools}
\usepackage{bm}
\usepackage{bbm}
\usepackage{mathrsfs}
\usepackage{centernot}
\usepackage{blkarray}

\DeclareMathAlphabet{\mathcal}{OMS}{cmsy}{m}{n}

\usepackage{graphicx}
\graphicspath{{FIGURES/}}

\usepackage{float}
\usepackage{placeins}
\usepackage{afterpage}
\usepackage{pdflscape}
\usepackage{rotating}

\usepackage{array}
\usepackage{booktabs}
\usepackage{tabularx}
\usepackage{longtable}
\usepackage{adjustbox}
\usepackage{threeparttable}
\usepackage{makecell}

\usepackage{soul}
\usepackage{cancel}
\usepackage{ulem}
\usepackage[font=bf,justification=centering]{caption}
\usepackage{subcaption}

\usepackage{enumitem}
\usepackage{setspace}
\usepackage[autostyle,english=american]{csquotes}
\MakeOuterQuote{"}

\usepackage{epigraph}

\usepackage{etoolbox}
\makeatletter
\newlength\epitextskip
\apptocmd{\@epitext}{\em}{}{}
\patchcmd{\epigraph}{\@epitext{#1}\\}{\@epitext{#1}\\[\epitextskip]}{}{}
\makeatother
\usepackage{minitoc}
\usepackage[toc,page,header]{appendix}
\usepackage{titletoc}

\usepackage{tikz}
\usepackage{tikz-cd}  
\usetikzlibrary{calc,tikzmark}

\usepackage[natbibapa]{apacite}

\numberwithin{equation}{section}

\newcommand{\indep}{\perp\!\!\!\!\perp}

\makeatletter
\newcommand{\vast}{\bBigg@{4}}
\newcommand{\Vast}{\bBigg@{5}}
\newcommand{\neutralize}[1]{\expandafter\let\csname c@#1\endcsname\count@}
\makeatother

\newtheorem{theorem}{Theorem}

\newtheorem{lemma}{Lemma}
\newtheorem{proposition}{Proposition}[section]
\newtheorem{assumption}{Assumption}
\newtheorem{corollary}{Corollary}[theorem]

\theoremstyle{definition}

\theoremstyle{remark}
\newtheorem{remark}{Remark}

\providecommand{\coloneq}{\mathrel{\mathop:}=}

\newcolumntype{L}[1]{>{\raggedright\arraybackslash}m{#1}}
\newcolumntype{C}[1]{>{\centering\arraybackslash}m{#1}}
\newcommand{\wideformula}[1]{%
\makebox[\linewidth][c]{\(\displaystyle #1\)}%
}
\newcommand{\tablesep}{\addlinespace[0.75em]}
\theoremstyle{definition}
\newtheorem{example}{Example}

\usepackage[
  colorlinks=true,
  urlcolor=blue,
  citecolor=blue,
  linkcolor=.
]{hyperref}

\begin{document}

\begin{titlepage}
\title{Tweedie Calculus\thanks{The author thanks Alberto Abadie, Isaiah Andrews, Victor Chernozhukov, Anna Mikusheva, and Whitney Newey for guidance and comments.}}
\author{Santiago Torres\thanks{Department of Economics, Massachusetts Institute of Technology. \href{mailto:storresp@mit.edu}{storresp@mit.edu}}}
\date{\today}

\maketitle
\begin{abstract}
\noindent
Tweedie's formula, which under Gaussian noise expresses the posterior mean of a latent variable directly from the observed-data density, is a cornerstone of empirical Bayes and measurement-error analysis. No general theory, however, explains when analogous identities hold, how they are structured, or how to derive them for non-Gaussian noise and for posterior functionals other than the mean. This paper develops
such a framework for additive-noise models. I characterize when conditional
expectations of an unobserved latent variable, given the observed signal, admit
direct expressions in terms of the observed density---identities I call
\emph{Tweedie representations}---and show that they are governed by a linear
map, the \emph{Tweedie functional}.  Under general conditions, I prove that this functional exists, is unique, and is continuous. I provide a constructive method for its computation based on Fourier analysis: the functional is obtained by extending the inverse Fourier transform of an explicit tempered distribution. The theory yields posterior-mean formulas for non-Gaussian noise and provides new representations for nonlinear posterior functionals. Applications include Laplace mechanisms in differential privacy and heteroskedastic Gaussian sequence models in compound decision problems.
\end{abstract}

\vspace{1em}
\par
\vspace{0.5em}

\setcounter{page}{0}
\thispagestyle{empty}
\end{titlepage}
\pagebreak \newpage

\doparttoc % Tell minitoc to generate a toc for the parts
\faketableofcontents % Run a fake tableofcontents command for the partocs

\section{Introduction}

Many statistical problems are naturally formulated in terms of latent effects observed through noise. In compound decision, measurement error, and empirical Bayes problems, estimands such as posterior means depend not only on the noise model but also on the unknown distribution of those latent effects. This naturally shifts attention to recovering that distribution from the data, a task known as deconvolution. Deconvolution, however, is notoriously difficult: it is an ill-posed inverse problem, and the object to be recovered is itself infinite-dimensional.

Tweedie's formula shows that one can estimate posterior means without recovering the latent distribution. Suppose that
$Y \sim \mathcal{N}(X, \sigma^2)$, where $X \sim P_X$ is an unknown latent signal. \cite{robbins1956Tweedie}
credits Maurice Tweedie with the identity
\begin{equation}\label{Eq: Tweedie}
  \mathbb{E}[X \mid Y = y] = y + \sigma^2 \frac{f_Y'(y)}{f_Y(y)},
\end{equation}
\noindent where \(f_Y\) is the marginal density of the observable \(Y\). The identity represents the posterior mean directly in terms of the observable density, which can be estimated without recovering the latent distribution \(P_X\).

Accordingly, Tweedie's formula provides a route to estimation and inference that avoids direct estimation of the latent signal distribution. Rather than estimate the latent law, one can express posterior functionals as transformations of the observed marginal density and then estimate that marginal density from the data.\footnote{In empirical Bayes terminology, the former approach is \(g\)-modeling and the latter is \(f\)-modeling \citep{efron2014two,efron2019bayes}.} This perspective is useful beyond computation. For the posterior mean, Tweedie’s formula gives a shrinkage interpretation: the posterior mean equals the observation plus a correction determined by the observed marginal density. In the Gaussian location model, this correction moves \(y\) toward regions of higher observed density, with magnitude governed by the noise variance and the score \(f_Y'(y)/f_Y(y)\).\footnote{This shrinkage view includes the James--Stein estimator \citep{james1961estimation} as an instance: in the Gaussian sequence model with a Gaussian latent distribution, the empirical Bayes plug-in version of Tweedie's formula yields James--Stein shrinkage up to degrees-of-freedom corrections \citep{efron2016empirical}.} In some models, Tweedie-type representations also imply shape restrictions, such as monotonicity, that can improve estimation \citep{houwelingen1983monotone}. They also identify the source of statistical difficulty. In the Gaussian posterior-mean case, estimation reduces to estimating pointwise derivatives of the observed marginal density, which has polynomial minimax rates under standard smoothness conditions \citep{stone1983optimal,pinkse2023estimates}. This contrasts with full nonparametric recovery of the latent law, which can have only logarithmic minimax rates under Gaussian noise \citep{fan1991optimal}.

Although formulas of this form are available for particular models and particular estimands, there is no general framework that explains when they exist, what structure they share, or how they can be derived. This paper provides such a framework. I consider general additive-noise models: triples of \(\mathbb R^d\)-valued random variables \((Y,X,V)\) satisfying \(Y=X+V\) with \(X\indep V\). For a broad class of functions \(g\) and noise laws \(V\), I study when posterior expectations\footnote{I use the term ``posterior'' rather than ``conditional'' following a Bayesian or empirical Bayes interpretation: \(X\) is the latent signal, \(Y\) is the observation, and \(\mathbb E[g(X)\mid Y=y]\) is the posterior expectation of \(g(X)\) after observing \(Y=y\).} admit a representation of the form
\begin{equation}\label{Eq: TweedieRep}
  \theta_V^g(y)=\mathbb{E}[g(X)\mid Y=y]
  = \frac{\mathcal{T}_{g,V,y}[f_Y]}{f_Y(y)},
\end{equation}
\noindent where \(\mathcal{T}_{g,V,y}\) is a linear functional on a function space containing all possible observed densities. It depends on the estimand \(g\), the noise law of \(V\), and the evaluation point \(y\), but not on the latent distribution \(P_X\). For example, in the one-dimensional Gaussian case with \(g(x)=x\), the functional in question is
\[
\mathcal{T}_{g,V,y}[f_Y] = y f_Y(y) + \sigma^2 f_Y'(y),
\]
which recovers \eqref{Eq: Tweedie}. I call \(\mathcal{T}_{g,V,y}\) the \emph{Tweedie functional}, and I call any expression of the form \eqref{Eq: TweedieRep} a \emph{Tweedie representation}. I characterize when such functionals exist and provide a general strategy to compute them in closed form. 

Several functionals fit into this notation. The examples below are common in the literature and will be used throughout the paper.

\begin{itemize}
       \item \textit{Posterior moments.} For a scalar latent variable, posterior moments correspond to polynomial choices of \(g\). The most used is the posterior mean, which is the Bayes rule for estimating the latent variable under squared-error loss. The conditional risk of this rule is the posterior variance. Since the posterior variance is determined by the first two posterior moments, it can also be expressed as a function of the marginal density of the data. More generally, posterior risks under polynomial loss functions can be written in terms of finitely many posterior moments, so they can likewise be expressed through the observed marginal density.

       \item \textit{Posterior cumulative distribution function.} For \(t\in\mathbb R\), the posterior cumulative distribution function corresponds to the indicator \(g_t(x)=\mathbbm{1}\{x\le t\}\). Thus \(F_{X\mid Y}(t\mid y)=\mathbb E[\mathbbm{1}\{X\le t\}\mid Y=y]\). This functional encompasses several quantities of interest. In large-scale studies, a common target is the posterior probability that an effect is positive, \(\Pr(X>0\mid Y=y)=1-F_{X\mid Y}(0\mid y)\). Such sign probabilities arise when assessing directional evidence across many units and have been studied in both theoretical and applied work; see \cite{es2005asymptotic,dattner2011deconvolution,efron2016empirical,greenshtein2018application,brennan2020estimating}. A related quantity is the local false sign rate \citep{stephens2017false}, \(\min\{\text{Pr}(X \leq 0 \mid Y=y),\text{Pr}(X \geq 0 \mid Y=y)\}\), the posterior probability of a sign error when reporting the more probable sign.

       \item \textit{Posterior moment generating function.} For \(t\in\mathbb R^d\), the posterior moment generating function corresponds to \(g_t(x)=\exp(t^\top x)\). When finite in a neighborhood of zero, the posterior moment generating function characterizes the posterior distribution and can be used to recover posterior moments and other distributional features. Moment generating functions also provide a standard tool for deriving tail bounds; see \cite{wainwright2019high}, Chapter 2.
    
\end{itemize}

The preceding examples motivate three questions about Tweedie representations that organize the paper:
\begin{enumerate}[label=(\roman*)]
  \item \textit{Existence and regularity of the Tweedie functional.}  Section \ref{S: Existence} asks when a posterior functional admits a Tweedie representation of the form \eqref{Eq: TweedieRep}. The main result, Theorem \ref{T: Representation}, gives sufficient conditions under which the Tweedie functional exists, is unique, and is continuous in a particular topology. The theorem applies to a broad class of additive-noise models, including cases where the law of \(V\) need not be symmetric or have full support. Although the result is not constructive, it gives a general structural characterization of such representations.

  \item \textit{Strategy for computing Tweedie representations.} Section \ref{S: Computation} builds on the existence result and develops a general method for deriving Tweedie representations through Fourier analysis. The main tool is the theory of tempered distributions, developed by \cite{schwartz1966theorie} and reviewed in Section \ref{S: Setup}. This language is needed because the relevant inverse Fourier transforms need not exist as ordinary functions. More importantly, it also provides a Fourier theory for linear functionals: with it, Fourier inversion applies not only to functions, but also to linear maps, delivering a calculus for functionals.

 The main result, Theorem \ref{T:master-fourier-identification}, shows that, under suitable conditions,  the Tweedie functional can be computed from the inverse Fourier transform of a known, computable linear functional. A major practical advantage of this viewpoint is that it reduces the derivation of Tweedie representations to the computation of inverse Fourier transforms. In this sense, the problem becomes one of calculus. As such, the relevant inverse transform can be computed directly or identified from standard transform tables, much as Laplace-transform tables are used in the study of differential equations. Useful references for this purpose include \cite{bateman_1954} and the Appendix of  \cite{kammler2007first}.

\item \textit{Applications.} Section \ref{S: Applications} illustrates the scope of the framework through several applications. First, it uses the Fourier-based strategy of Section \ref{S: Computation} to derive posterior-mean representations for a broad class of noise distributions. Table~\ref{tab:noise-families} summarizes these formulas. Except for the Gaussian and Laplace cases, these representations appear to be new. 

Second, Section~\ref{SS: Laplace} applies the framework to the conventional Laplace mechanism from differential privacy. The paper derives Tweedie representations for several posterior functionals under this noise structure. Table~\ref{tab:laplace-tweedie-functionals} summarizes selected cases.

Section \ref{S: Heteroskedastic} studies the heteroskedastic Gaussian sequence model. This model is especially important in economics and the social sciences, where researchers often observe noisy estimates of means or effects across many units or populations.\footnote{\cite{walters2024empirical} provides a review of common applications of this model in economics.} The Gaussian specification formalizes the normal approximation already implicit in many standard inference procedures, often justified by large-sample central limit arguments. In this model, \(Y\) is a noisy estimate of a latent parameter \(X\), and the noise covariance matrix \(\Sigma\) may vary across units. It may also be random and correlated with \(X\). The key observation is that, after conditioning on \(\Sigma=\Sigma_0\) and standardizing, the model reduces to an additive Gaussian model with fixed covariance. The general framework therefore applies conditionally on \(\Sigma\), yielding Tweedie representations in terms of the conditional density \(f_{Y\mid\Sigma}(\cdot\mid\Sigma_0)\). This conditional formulation gives explicit representations for posterior means and for other empirical Bayes functionals, including posterior covariances, moments, risk functions, distribution functions, and moment generating functions.

\end{enumerate}

\subsection{Relation to the literature.} 

The paper by \cite{raphan2011least} is closest to my approach. They show that, for several additive-noise models, the posterior mean can be written as the ratio of a linear operator applied to the observed marginal density and the density itself. They further note that, once such an operator is available, iterating it recovers posterior moments and, more generally, posterior expectations of polynomial functions. This paper formalizes and extends that observation. Rather than build outward from the posterior-mean case, I study when, for a general measurable function \(g\), the conditional expectation \(\mathbb E[g(X)\mid Y=y]\) admits an analogous representation. I also establish formal properties of the representing map, including existence, uniqueness, and continuity under the relevant topology. The framework also extends beyond the basic additive-noise setting, for example to heteroskedastic Gaussian sequence models, where the additive-noise structure holds only conditionally.

On the computational side, the paper's main innovation is to introduce tempered distributions as a tool for deriving Tweedie representations. This serves two purposes. First, it makes rigorous the Fourier calculations that underlie many Tweedie-type identities for posterior moments, including cases in which the relevant inverse Fourier transforms exist only as distributions.  Second, it yields a systematic procedure for deriving Tweedie representations for posterior functionals beyond moments, including nonsmooth and non-polynomial targets. Closest in this respect is \cite{calmon2025debiasing}, which uses tempered distributions to derive unbiased estimators under Laplace noise. This approach complements derivations based on Stein identities and exponential-family integration by parts. For example, in exponential families, \cite{efron2011tweedie} observes that posterior-moment representations can be obtained by differentiating the log ratio of the marginal density to the base-measure density defining the family. Related identities, including Stein’s lemma and Stein’s unbiased risk estimate, are often derived by integration by parts \citep{stein1981estimation}.

The paper also provides several new Tweedie representations. For instance, it derives a Tweedie representation for the posterior cumulative distribution function under Gaussian noise. It also derives analogous representations for posterior means under several non-Gaussian noises, including Cauchy, Gumbel, and Gamma noise. In this respect, the paper is related to the literature on deriving new Tweedie-type identities in Bayesian, empirical Bayes, and measurement-error problems. Historically, the Gaussian identity appears in print in \cite{dyson1926method}, where Dyson attributes it to Arthur Eddington. Later, \cite{robbins1956Tweedie} reported the formula and credited Maurice Kenneth Tweedie. Much of the subsequent literature develops model-specific analogues or extensions of this identity. \cite{west1982aspects} derives a scale-parameter analogue of Tweedie's formula for models with known location and unknown scale. \cite{pericchi1992exact} and \cite{pericchi1993posterior} derive exact relationships for posterior moments and cumulants in normal-location and exponential-family settings.  \cite{Polson1991} obtains an exact representation of the posterior mean for location models with normal scale-mixture priors, and \cite{polson2019bayesian} gives a Gaussian linear-regression analogue. \cite{robbins1982estimating} derives a chi-square analogue in the empirical Bayes problem of estimating many variances. \cite{shi2015nonlinear} and \cite{hillebrand2023unbiased} develop Laplace-noise-based Tweedie-type formulae for the posterior mean and for polynomial functionals, respectively. 

Finally, the paper relates to applications of Tweedie's formula. These methods are widely used in empirical Bayes problems \citep{efron2011tweedie,efron2014two}, including large-scale hypothesis testing \citep{efron2012large}. More recently, Tweedie formulas have been used in diffusion-based sampling and score-based generative modeling \citep{tang2026tweedie,montanari2023sampling,song2020score,sohl2015deep}. Related denoising problems also arise in differential privacy, where Gaussian or Laplace noise is often added to statistics to limit the information revealed about any single data point. For Laplace noise, \cite{hong2003simple} propose a correction based on the Tweedie representer for nonlinear method-of-moments models with measurement error, and \cite{shi2019estimation} apply Laplace-based Tweedie-type formulas to measurement-error and regression problems. Tweedie's formula has also been used for forecasting in dynamic panel models \citep{Liu2020}. Recent work by \cite{liang2025distributional,liang2025distributionalII} uses identities closely related to Tweedie's formula to construct shrinkage maps for recovering the latent signal distribution. More broadly, the framework is connected to the deconvolution literature \citep{fan1991optimal,butucea2009adaptive,pensky2017minimax,hall2007ridge} and to additive measurement-error models \citep{carroll2006measurement,meister2009,schennach2004estimation,schennach2004nonparametric,schennach_2013}.

\section{Setup and basic notions in Fourier Analysis} \label{S: Setup}

This section introduces the additive-noise model studied throughout the paper, states the standing assumptions under which the theory is developed, and records several immediate consequences. The framework is intentionally broad: it nests the classical Gaussian model underlying Tweedie's formula and accommodates more general deconvolution problems in arbitrary dimensions.

\subsection{Setup}

Let $\mathcal{P}(\mathbb{R}^d)$ denote the set of probability measures on
$(\mathbb{R}^d,\mathcal{B}(\mathbb{R}^d))$, where
$\mathcal{B}(\mathbb{R}^d)$ is the Borel $\sigma$-algebra on
$\mathbb{R}^d$. For a random vector \(Z\), let \(P_Z\) denote its law. Consider a triple of $d-$dimensional random vectors $(Y,X,V)$ defined on the same probability space $(\Omega,\mathscr{A}, \mathbb{P})$. Here $Y \sim P_Y$ is the observed vector, $X \sim P_X$ is the latent signal of
interest, and $V \sim P_V$ is the noise.

For any $P \in \mathcal{P}(\mathbb{R}^d)$, let
$\varphi_P : \mathbb{R}^d \to \mathbb{C}$ denote its characteristic function,
defined by
\[
\varphi_P(\omega)
\coloneq
\int_{\mathbb{R}^d} e^{i\omega^\top x}\,dP(x)
=
\mathbb{E}_{Z\sim P}\!\left[e^{i\omega^\top Z}\right],
\qquad \omega \in \mathbb{R}^d.
\]
When $Z \sim P_Z$, I use the shorthand notation $\varphi_{Z} \coloneq \varphi_{P_Z}$. I now state the main assumptions used throughout the paper.

\begin{assumption}[Additivity and independence] \label{A: independence}
The triple  $(Y,X,V)$ is related by addition $Y=X+V$, and \(X\) and \(V\) are independent.
\end{assumption}

Under this assumption, the observed characteristic function factors as
\(\varphi_Y(\omega)=\varphi_X(\omega)\varphi_V(\omega)\) for all
\(\omega\in\mathbb R^d\).\footnote{\cite{schennach2019convolution} emphasizes that such factorization can hold even without independence of \(X\) and \(V\). In this paper, however, full independence is needed to connect deconvolution to conditional expectations. It is also the standard assumption in additive-noise applications.} This factorization is what makes closed-form Tweedie
representations possible: it allows terms involving the unknown latent law \(P_X\), encoded by \(\varphi_X\), to be rewritten in terms of the observed marginal law and the known noise law, characterized by \(\varphi_Y\) and \(\varphi_V\).

\begin{assumption}[Identification] \label{A: Identification}
The noise characteristic function is nonzero almost everywhere:
\[
\varphi_V(\omega)\neq 0
\quad\text{for Lebesgue-a.e. }\omega\in\mathbb{R}^d.
\]
\end{assumption}

Assumption \ref{A: Identification} is the basic nondegeneracy condition behind
Fourier deconvolution. Informally, recovering $P_X$ from the relation $\varphi_Y(\omega)=\varphi_X(\omega)\varphi_V(\omega)$
requires division by $\varphi_V(\omega)$. Therefore, this assumption rules out,
except on a null set, frequency regions at which inversion breaks down.  It is satisfied by the standard Gaussian, Laplace, logistic, and many
other common noise laws. It also holds, for example, when \(V\) has compact
support.\footnote{A simple sufficient condition is the existence of an
exponential moment: \(\mathbb E \left[\exp(r\|V\|_2)\right]<\infty\) for some \(r>0\). Then
\(\varphi_V\) extends holomorphically to a complex neighborhood of
\(\mathbb R^d\). Since \(\mathfrak{Re}(\varphi_V)(0)=1\), its real part is a
nonzero real-analytic function. Accordingly, the zero set of
\(\mathfrak{Re}(\varphi_V)\) has Lebesgue measure zero, and so does the
zero set of \(\varphi_V\) \citep{mityagin2020zero}.}

\begin{assumption}[Noise regularity] \label{A: Density}
The noise distribution $P_V$ admits a Lebesgue density $f_V$.
\end{assumption}

ssumption \ref{A: Density}, together with the convolution structure imposed by Assumption \ref{A: independence}, has two immediate consequences. First, the observed law \(P_Y\) is absolutely continuous with respect to Lebesgue measure, so the observed density \(f_Y\) appearing in Tweedie representations is well defined. Second, Proposition \ref{P: convolution} shows that one version of its density is obtained by convolving the latent law \(P_X\) with the fixed version of the noise density \(f_V\) in Assumption~\ref{A: Density}. Throughout the paper, \(f_Y\) denotes this convolution version. This convolution formula will be essential for characterizing the Tweedie functional.

\begin{proposition}\label{P: convolution}
Under assumptions \ref{A: independence} and \ref{A: Density}, the law $P_Y$
is absolutely continuous with respect to Lebesgue measure. A version of its
density is
\[
f_Y(y)=\int_{\mathbb{R}^d} f_V(y-x)\,dP_X(x), \qquad y\in\mathbb{R}^d.
\]

\end{proposition}

\subsection{Some important notions in Fourier analysis}

This subsection fixes the Fourier analysis conventions used throughout the paper. It first defines the ordinary Fourier transform for integrable functions and then extends the transform to tempered distributions.

For \(1\le p\le\infty\) and \(\mathbb K\in\{\mathbb R,\mathbb C\}\), let
\(L^p(\mathbb R^d;\mathbb K)\) denote the usual Lebesgue space, with norm
\[
\|f\|_p =
\begin{cases}
\left(\displaystyle\int_{\mathbb R^d}|f(x)|^p\,dx\right)^{1/p},
& 1\le p<\infty,\\[1ex]
\operatorname*{ess\,sup} \limits_{x\in\mathbb R^d}|f(x)|,
& p=\infty.
\end{cases}
\]
When \(\mathbb K\) is omitted, it is understood to be \(\mathbb R\). The symbol
\(i\) denotes the imaginary unit, and \(\overline z\) denotes the complex
conjugate of \(z\in\mathbb C\). Let \(C_0(\mathbb R^d;\mathbb K)\) denote the
space of continuous \(\mathbb K\)-valued functions on \(\mathbb R^d\) that
vanish at infinity.

For \(\psi\in L^1(\mathbb R^d;\mathbb C)\), define its Fourier transform by
\[
\widetilde\psi(\omega)
\coloneq
\int_{\mathbb R^d} e^{i\omega^\top x}\psi(x)\,dx,
\qquad
\omega\in\mathbb R^d.
\]
For \(\psi\in L^1(\mathbb R^d;\mathbb C)\), let \(\psi^\sharp\) denote the inverse Fourier integral
\[
\psi^\sharp(x)
\coloneq
\frac{1}{(2\pi)^d}
\int_{\mathbb R^d} e^{-i\omega^\top x}\psi(\omega)\,d\omega,
\qquad
x\in\mathbb R^d,
\]
whenever this integral is well defined. Thus \(\psi\mapsto\psi^\sharp\) denotes the inverse Fourier operation under the convention above. If
\(\psi,\widetilde\psi\in L^1(\mathbb R^d;\mathbb C)\), Fourier inversion gives
\[
\psi(x)
=
\widetilde\psi^\sharp(x)
=
\frac{1}{(2\pi)^d}
\int_{\mathbb R^d} e^{-i\omega^\top x}\widetilde\psi(\omega)\,d\omega
\]
for Lebesgue-a.e. \(x\in\mathbb R^d\).

For the present paper, this classical framework is too narrow. Many objects that arise in the computation of Tweedie representations are not integrable, so their inverse Fourier transforms need not exist as functions. More importantly, Tweedie functionals are linear maps acting on the observed density. Consequently, it is advantageous to use a Fourier theory that can transform and invert linear functionals, rather than one limited to ordinary functions. I use the framework of tempered distributions, introduced by \cite{schwartz1966theorie}, for this purpose. This framework extends Fourier analysis from functions to continuous linear functionals and thus provides a natural language for deconvolution problems.

The starting point is the Schwartz space on \(\mathbb R^d\), denoted by
\(\mathscr S(\mathbb R^d)\). It is the space of smooth test functions against
which tempered distributions will be evaluated. Formally,
\[
\mathscr{S}(\mathbb{R}^d)
\coloneq
\left\{
\psi\in C^\infty(\mathbb{R}^d;\mathbb{C})
:
\sup_{x\in\mathbb{R}^d}
\bigl|x^{\alpha} \partial^{\beta} \psi(x)\bigr|<\infty
\text{ for all multi-indices }\alpha,\beta\in\mathbb{N}_0^d
\right\}.
\]
Here \(x^{\alpha} \coloneq x_1^{\alpha_1}\cdots x_d^{\alpha_d}\),
\(\partial^{\beta} \coloneq \partial^{|\beta|}/
\partial x_1^{\beta_1}\cdots \partial x_d^{\beta_d}\), and
\(|\beta|=\beta_1+\cdots+\beta_d\). Thus \(\mathscr S(\mathbb R^d)\) consists
of smooth functions whose derivatives decay faster than any inverse polynomial.
It contains, for example, compactly supported smooth functions and Gaussian
densities.

The Fourier transform is especially well behaved on \(\mathscr S(\mathbb R^d)\).
The map \(\psi\mapsto\widetilde\psi\) is a continuous automorphism of
\(\mathscr S(\mathbb R^d)\), with inverse \(\psi\mapsto\psi^\sharp\). Thus every
Schwartz function has a Fourier transform, and every Schwartz function
is the inverse Fourier transform of another Schwartz function. Continuity
is understood with respect to the standard Schwartz topology, under which
\(\mathscr S(\mathbb R^d)\) is a Fr\'echet space.\footnote{That is, a complete
metrizable locally convex vector space. For a textbook construction of the
topology, see \cite{treves2006topological}, Chapter 10, Example IV.}
Appendix~\ref{A:schwartz-tempered-topology} provides additional details on this construction.

A tempered distribution is a continuous linear functional on \(\mathscr S(\mathbb R^d)\). I write $\mathscr S'(\mathbb R^d)
\coloneq
\bigl(\mathscr S(\mathbb R^d)\bigr)'$ 
for the space of tempered distributions. Critically, many ordinary functions define tempered
distributions when their integration functionals are continuous on \(\mathscr S(\mathbb R^d)\). Specifically, if \(Q:\mathbb R^d\to\mathbb C\) is measurable and
\[
\mathcal I_Q[\psi]
\coloneq
\int_{\mathbb R^d}Q(x)\psi(x)\,dx,
\qquad
\psi\in\mathscr S(\mathbb R^d),
\]
defines a continuous linear functional on \(\mathscr S(\mathbb R^d)\), then
\(\mathcal I_Q\in\mathscr S'(\mathbb R^d)\).\footnote{Lemma~\ref{L:weighted-integrability-temperedness}
gives a general sufficient condition on $Q$ such that it defines a tempered distribution.}  The space \(\mathscr S'(\mathbb R^d)\) also contains singular distributions that are not induced by ordinary functions. For example, evaluation at a point \(a\in\mathbb R^d\),
\[
\operatorname{ev}_a[\psi]\coloneq \psi(a),
\qquad
\psi\in\mathscr S(\mathbb R^d),
\]
is a tempered distribution.

The Fourier transform extends from Schwartz functions to tempered distributions by duality. With the convention above, define operators
\(\mathcal F,\mathcal F^{-1}:\mathscr S'(\mathbb R^d)\to \mathscr S'(\mathbb R^d)\) by
\[
\mathcal F\{T\}[\psi]
\coloneq
T[\widetilde\psi],
\qquad
\mathcal F^{-1}\{T\}[\psi]
\coloneq
T[\psi^\sharp],
\qquad
\psi\in\mathscr S(\mathbb R^d).
\]
These definitions are well posed because the maps
\(\psi\mapsto\widetilde\psi\) and \(\psi\mapsto\psi^\sharp\) map \(\mathscr S(\mathbb R^d)\) continuously to itself. Under the weak-\(\star\) topology
on \(\mathscr S'(\mathbb R^d)\), the operators \(\mathcal F\) and
\(\mathcal F^{-1}\) are continuous automorphisms; see
Appendix~\ref{A:schwartz-tempered-topology}.

\section{Existence of the Tweedie functional} \label{S: Existence}

This section develops the existence theory for Tweedie representations. In the additive-noise model, the problem reduces to the convolution operator induced by the noise. I characterize when this operator is invertible on its range and use the resulting inverse to equip the class of admissible observed densities with a natural topology.

To formalize the representation problem, let \(\lambda:\mathbb{R}^d\to\mathbb{C}\) be measurable and define
\[
\Lambda_\lambda(P_X)
\coloneq
\int_{\mathbb{R}^d} \lambda(x)\,dP_X(x),
\]
whenever the integral is well defined. The question is whether this latent functional depends on \(P_X\) only through the observed marginal density \(f_Y\). Given a triple \((Y,X,V)\) satisfying Assumptions~\ref{A: independence}--\ref{A: Density}, I identify a class of integrands \(\lambda\) for which there exists a unique linear functional \(\mathcal{T}_{\lambda,V}\) satisfying $\Lambda_\lambda(P_X)=\mathcal{T}_{\lambda,V}[f_Y]$. For these integrands, the latent integral is therefore representable in terms of the observed density alone.

To state this result, let $\mathcal{M}(\mathbb{R}^d)$ denote the space of finite complex Borel
measures on $\mathbb{R}^d$. For $\mu\in\mathcal{M}(\mathbb{R}^d)$, define its
total variation norm by
\[
\|\mu\|_{TV}
\coloneq
\sup\left\{
\left|\int_{\mathbb{R}^d} h(x)\,d\mu(x)\right|
:
h\in C_0(\mathbb{R}^d;\mathbb{C}),\ \|h\|_\infty\le 1
\right\}.
\]
It is well known that $\bigl(\mathcal{M}(\mathbb{R}^d),\|\cdot\|_{TV}\bigr)$ is a complex Banach space \citep[Theorem 6.19]{rudin1987real}. Since $\mathcal{P}(\mathbb{R}^d)\subseteq \mathcal{M}(\mathbb{R}^d)$, the results below apply in particular to probability measures. I nevertheless work on $\mathcal{M}(\mathbb{R}^d)$ because its vector space structure will be useful throughout.

Given the noise law $P_V$ with density $f_V$, define the convolution operator
$\mathcal{K}_V:\mathcal{M}(\mathbb{R}^d)\to L^1(\mathbb{R}^d;\mathbb{C})$ by
\[
\mathcal{K}_V[\mu](y)
\coloneq
(f_V * \mu)(y)
=
\int_{\mathbb{R}^d} f_V(y-x)\,d\mu(x),
\qquad y\in\mathbb{R}^d.
\]
The operator $\mathcal{K}_V$ is linear and continuous from
$\bigl(\mathcal{M}(\mathbb{R}^d),\|\cdot\|_{TV}\bigr)$ into
$\bigl(L^1(\mathbb{R}^d;\mathbb{C}),\|\cdot\|_1\bigr)$. Indeed, it is immediate from the definition that $\|\mathcal{K}_V[\mu]\|_1
\le
\|\mu\|_{TV}$.

The range of $\mathcal{K}_V$ consists of all $L^1$ functions obtainable by convolving the noise density $f_V$ against a finite complex Borel measure. I call this range the class of $V$-admissible mixtures and write
\[
\mathcal{A}_V(\mathbb{R}^d)
\coloneq
\{f=\mathcal{K}_V[\mu]=f_V * \mu : \mu\in\mathcal{M}(\mathbb{R}^d)\}
\subseteq L^1(\mathbb{R}^d;\mathbb{C}).
\]
By Proposition \ref{P: convolution}, whenever the observed density $f_Y$ exists, it belongs to $\mathcal{A}_V(\mathbb{R}^d)$.

I next establish that $\mathcal{K}_V$ is a bijection between latent measures and $V$-admissible mixtures. Surjectivity is immediate from the definition of $\mathcal{A}_V(\mathbb{R}^d)$. The remaining issue is injectivity. Assumption \ref{A: Identification} provides it by ruling out distinct latent measures that generate the same mixture.

\begin{proposition}\label{P: bijection}
Under Assumption \ref{A: Identification}, the operator $\mathcal{K}_V$ is a bijection between $\mathcal{M}(\mathbb{R}^d)$ and $\mathcal{A}_V(\mathbb{R}^d)$.     
\end{proposition}

Proposition~\ref{P: bijection} implies the inverse
$\mathcal{K}_V^{-1} : \mathcal{A}_V(\mathbb{R}^d) \to \mathcal{M}(\mathbb{R}^d)$ is well defined. I use this inverse to transfer the total variation norm from
$\mathcal{M}(\mathbb{R}^d)$ to $\mathcal{A}_V(\mathbb{R}^d)$. Specifically, if
$f \in \mathcal{A}_V(\mathbb{R}^d)$ and $f = f_V \ast \mu$ for some
$\mu \in \mathcal{M}(\mathbb{R}^d)$, define
\[
\|f\|_{\mathcal{A}_V}
\coloneq
\|\mathcal{K}_V^{-1}[f]\|_{TV}
=
\|\mu\|_{TV}.
\]

This definition is unambiguous because Proposition~\ref{P: bijection} gives a unique representing measure $\mu$. Furthermore, with this norm, $\mathcal{K}_V :
\bigl(\mathcal{M}(\mathbb{R}^d),\|\cdot\|_{TV}\bigr)
\to
\bigl(\mathcal{A}_V(\mathbb{R}^d),\|\cdot\|_{\mathcal{A}_V}\bigr)$ is an isometric isomorphism. Hence $\mathcal{A}_V(\mathbb{R}^d)$ inherits the Banach space structure of $\mathcal{M}(\mathbb{R}^d)$. Let $\mathcal{A}_V'(\mathbb{R}^d)
\coloneq
\bigl(\mathcal{A}_V(\mathbb{R}^d)\bigr)'$ denote the Banach dual of $\mathcal{A}_V(\mathbb{R}^d)$, that is, the space of continuous linear functionals on $\mathcal{A}_V(\mathbb{R}^d)$. The next theorem shows that the density-based representation maps sought here are precisely elements of this dual space.

\begin{theorem}[Existence and uniqueness of the representing functional] \label{T: Representation}
Let $(Y,X,V)$ satisfy Assumptions \ref{A: Identification} and \ref{A: Density}. For
every $\lambda\in C_0(\mathbb{R}^d;\mathbb{C})$, there exists a unique functional
$\mathcal{T}_{\lambda,V}\in \mathcal{A}'_V(\mathbb{R}^d)$ such that
\[
\mathcal{T}_{\lambda,V}\!\left[\mathcal{K}_V[\mu]\right]
=
\int_{\mathbb{R}^d}\lambda(x)\,d\mu(x)
\qquad
\text{for all }\mu\in\mathcal{M}(\mathbb{R}^d).
\]
In particular, under Assumption~\ref{A: independence}, $f_Y = \mathcal{K}_V[P_X]$, and so
\[
  \mathcal{T}_{\lambda,V}[f_Y] = \int_{\mathbb{R}^d} \lambda(x)\, dP_X(x).
\]
\end{theorem}

Theorem \ref{T: Representation} establishes existence and uniqueness of the representation map. Specifically, for any $\lambda\in C_0(\mathbb{R}^d)$, the theorem gives a unique continuous linear functional $\mathcal{T}_{\lambda,V}\in \mathcal{A}_V'(\mathbb{R}^d)$ such that $\mathcal{T}_{\lambda,V}[\mathcal{K}_V[\mu]]
=
\int_{\mathbb{R}^d}\lambda(x)\,d\mu(x)$ for every $\mu\in\mathcal{M}(\mathbb{R}^d)$. Applying this identity to the observed density $f_Y=\mathcal{K}_V[P_X]$ shows that the latent integral $\Lambda_\lambda(P_X)$ can be recovered as the value of a continuous linear functional of $f_Y$.

An immediate consequence of Theorem \ref{T: Representation} is the existence and uniqueness of the Tweedie representation. The key point is that, after applying Bayes' formula to $\theta_V^g$, the numerator is an integral of a known function with respect to the latent law $P_X$. Indeed, under Assumptions \ref{A: independence} and \ref{A: Density}, for any $y$ with $f_Y(y)>0$,
\begin{equation}\label{Eq: ratio}
\theta_V^g(y)
=
\frac{1}{f_Y(y)}
\int_{\mathbb{R}^d} \lambda_{g,V,y}(x)\,dP_X(x)
=
\frac{\Lambda_{\lambda_{g,V,y}}(P_X)}{f_Y(y)},
\quad
\lambda_{g,V,y}(x)\coloneq g(x)f_V(y-x).
\end{equation}

Thus, whenever $\lambda_{g,V,y}\in C_0(\mathbb{R}^d)$, the numerator in \eqref{Eq: ratio} falls under Theorem \ref{T: Representation}. This gives the following corollary.

\begin{corollary}[Existence and uniqueness of the Tweedie functional]
\label{C: TweedieExistence}
Fix \(y\in\mathbb R^d\) and let \(g:\mathbb R^d\to\mathbb R\) be measurable.
Under Assumptions~\ref{A: independence}, \ref{A: Identification}, and~\ref{A: Density}, if $ x \mapsto \lambda_{g,V,y}(x)=g(x)f_V(y-x) \in C_0(\mathbb{R}^d;\mathbb{C})$, then there exists a unique
\(\mathcal T_{g,V,y}\in\mathcal A_V'(\mathbb R^d)\) such that
\[
\mathcal T_{g,V,y}\!\left[\mathcal K_V[\mu]\right]
=
\int_{\mathbb R^d} g(x)f_V(y-x)\,d\mu(x),
\qquad
\forall \mu\in\mathcal M(\mathbb R^d).
\]
Consequently, if \(X\sim P_X\), \(f_Y=\mathcal K_V[P_X]\in\mathcal A_V(\mathbb R^d)\),
and \(f_Y(y)>0\), then the posterior expectation of \(g(X)\) given \(Y=y\) is finite and satisfies
\[
\theta_V^g(y)
=
\mathbb E[g(X)\mid Y=y]
=
\frac{\mathcal T_{g,V,y}[f_Y]}{f_Y(y)}.
\]
\end{corollary}

\begin{remark}
The representation is defined only at points $y$ such that $f_Y(y)>0$.
This restriction is harmless for most  statements about $Y$, since the representation is defined $P_Y$-almost surely:
\[
P_Y(\{y:f_Y(y)=0\})
=
\int_{\{y:f_Y(y)=0\}} f_Y(y)\,dy
=
0.
\]
\end{remark}

\begin{remark}
A limitation of the preceding corollary is that it applies only when the integrand \(\lambda_{g,V,y}\) is continuous and vanishes at infinity. This excludes important cases, such as indicator functions \(g\), which are needed to represent posterior cumulative distribution functions. Section~\ref{SS: Approximation} later removes this restriction by showing that the corresponding Tweedie functionals arise as limits of Tweedie functionals associated with continuous approximations to \(g\).
\end{remark}

\begin{remark}
Corollary~\ref{C: TweedieExistence} provides a pointwise identity for each
fixed \(y\in\mathbb{R}^d\). In many applications, however, the object of
interest is the conditional mean as a function of the observation, namely
\(\mathbb{E}[g(X)\mid Y]\) viewed as an integrable
\(\sigma(Y)\)-measurable random variable. Proposition~\ref{P:operator-Tweedie}
in Appendix~\ref{A:TweedieOperator} gives an operator-level version of the
pointwise representation. Concretely, it establishes under further regularity that the pointwise maps \(\mathcal{T}_{g,V,y}\) are obtained
by evaluating a single linear operator \(\mathscr{T}_{g,V}\):
\[
    \mathscr{T}_{g,V}[f_Y](y)
    =
    \mathcal{T}_{g,V,y}[f_Y]
    \qquad \text{for every } y \text{ with } f_Y(y)>0.
\]
Consequently, the conditional mean admits the global representation
\[
    \mathbb{E}[g(X)\mid Y]
    =
    \frac{\mathscr{T}_{g,V}[f_Y](Y)}{f_Y(Y)}
    \qquad \mathbb{P}\text{-almost surely}.
\]
I call \(\mathscr{T}_{g,V}\) the \emph{Tweedie operator}. This formulation
separates the dependence on the observed density from the dependence on the
evaluation point: the density \(f_Y\) is first mapped to the function
\(\mathscr{T}_{g,V}[f_Y]\), and the pointwise numerator is then obtained by
evaluation. Equivalently, it factors the Tweedie functional as
\(\mathcal{T}_{g,V,y}=\operatorname{ev}_y \circ \mathscr{T}_{g,V}\)
on the common domain where both sides are defined.  Because the image of \(\mathscr{T}_{g,V}\) can be characterized, this
operator-level representation provides a way to study global regularity
properties of the posterior mean. In particular, Corollary~\ref{C:PosteriorMeanSmoothness} uses this structure to establish
differentiability of \(\mathbb{E}[g(X)\mid Y]\). A related operator-level
formulation for posterior means appears in \cite{raphan2011least}, Eq.~3.5.
\end{remark}

\section{Computing Tweedie representations} \label{S: Computation}
For each \(y\in\mathbb R^d\), Section~\ref{S: Existence} constructed a continuous linear functional \(\mathcal T_{g,V,y}\in\mathcal A'_V(\mathbb R^d)\) such that, under the additive-noise model, whenever \(f_Y(y)>0\), 
\[ \theta_V^g(y) = \frac{1}{f_Y(y)}\,\mathcal T_{g,V,y}[f_Y], \quad \mathcal T_{g,V,y}[f_Y] = \int_{\mathbb R^d}\lambda_{g,V,y}(x)\,dP_X(x), \quad \mathcal T_{g,V,y}\in\mathcal A'_V(\mathbb R^d). \]
To carry out this strategy, however, I need to impose additional regularity on the noise
density. This regularity is inherited by \(V\)-mixtures: if the noise density
has \(k\) continuous derivatives that vanish at infinity, then every
mixture generated by this noise has the same degree of smoothness. This
assumption seems natural in the present problem. Indeed, in the Gaussian case, Tweedie's formula already requires differentiability
of the observed marginal density. 

The smoothness assumption lets us work in an ambient Banach space whose topology
is weaker than the original \(\mathcal A_V\)-topology used for existence.
Crucially, Schwartz functions are dense in this weaker topology. Since the 
Schwartz space is the natural setting for Fourier inversion of tempered
distributions, we can perform the deconvolution and inverse Fourier calculation
there, and then extend the resulting functional continuously to all
\(V\)-mixtures. This is how the abstract functional from
Section~\ref{S: Existence} becomes computable.

I now make this construction explicit. For \(k\in\mathbb N_0\), define
\[
C_0^k(\mathbb R^d;\mathbb C)
\coloneq
\left\{
\psi\in C^k(\mathbb R^d;\mathbb C)
:
\partial^\alpha \psi\in C_0(\mathbb R^d;\mathbb C)
\text{ for all } |\alpha|\le k
\right\},
\]
Thus \(C_0^k(\mathbb R^d;\mathbb C)\) is the space of \(k\)-times continuously differentiable functions whose derivatives up to order \(k\) vanish at infinity. Now define
\[
\Xi^k(\mathbb R^d;\mathbb C)
\coloneq
C_0^k(\mathbb R^d;\mathbb C)\cap L^1(\mathbb R^d;\mathbb C),
\qquad
\|\psi\|_{\Xi^k}
\coloneq
\max_{|\alpha|\le k}\|\partial^\alpha \psi\|_\infty+\|\psi\|_1.
\]

\begin{assumption}[Noise smoothness of order \(k\)]
\label{A:SmoothBoundedNoise}
For some \(k\in\mathbb N_0\), the noise density \(f_V\) belongs to
\(C_0^k(\mathbb R^d;\mathbb R)\).
\end{assumption}

Assumption~\ref{A:SmoothBoundedNoise} says that the noise density has
continuous derivatives up to order \(k\), all of which vanish at infinity.
Since \(f_V\) is a density, the assumption naturally implies that 
\(f_V\in\Xi^k(\mathbb R^d;\mathbb C)\). This assumption  encompasses many common
noise laws, for suitable choices of \(k \geq 0\), such as Gaussian, logistic, Gumbel,
Laplace, and Cauchy distributions.

The next lemma records the basic facts about this space used to establish the
main theorem. The space \(\Xi^k\) is a Banach space; the space of
\(V\)-mixtures embeds continuously into \(\Xi^k\); and Schwartz functions are
dense in \(\Xi^k\).

\begin{lemma}[Mixture smoothness and approximation]
\label{L: SmoothModels}
Assume \(f_V\) satisfies Assumption~\ref{A:SmoothBoundedNoise} with smoothness
order \(k\). Then:

\begin{itemize}

    \item[a)] \textbf{Completeness.}
    \((\Xi^k(\mathbb R^d;\mathbb C),\|\cdot\|_{\Xi^k})\) is a Banach space.

    \item[b)] \textbf{\(V\)-mixtures inherit noise regularity.}
    The inclusion $\mathcal A_V(\mathbb R^d)\hookrightarrow \Xi^k(\mathbb R^d;\mathbb C)$ is continuous.

    \item[c)] \textbf{Schwartz approximation.}
    The inclusion $\mathscr S(\mathbb R^d)\hookrightarrow \Xi^k(\mathbb R^d;\mathbb C)$ is continuous, and \(\mathscr S(\mathbb R^d)\) is dense in
    \((\Xi^k(\mathbb R^d;\mathbb C),\|\cdot\|_{\Xi^k})\).

\end{itemize}
\end{lemma}

\subsection{A general strategy for computing Tweedie representations}

I now state the main result of this section, and the main computational result
of the paper.

\begin{theorem}[Computation of Tweedie functionals]
\label{T:master-fourier-identification}
Fix \(y\in\mathbb R^d\) and let \(g:\mathbb R^d\to\mathbb R\) be measurable. Suppose that \(V\) satisfies Assumptions~\ref{A: Identification}
and~\ref{A:SmoothBoundedNoise} for some \(k\ge 0\). Let $x \mapsto g(x)f_V(y-x) \eqqcolon\lambda_{g,V,y}(x)\in  C_0(\mathbb{R}^d;\mathbb{C})$. Suppose there exists a measurable function $\mathcal Q_{g,V,y}:\mathbb{R}^d \to \mathbb{C}$ such that

\begin{enumerate}
\item[(i)] \textbf{Temperedness.}
There exists \(N\in\mathbb N_0\) such that
\[
\int_{\mathbb R^d}
\frac{|\mathcal Q_{g,V,y}(\omega)|}{(1+\|\omega\|_2)^N}\,d\omega
<\infty.
\]

\item[(ii)] \textbf{\(\Xi^k\)-continuity.}
There exists \(C>0\) such that, for every \(\psi\in\mathscr S(\mathbb R^d)\),
\[
\bigl|\mathcal F^{-1}\{\mathcal I_{\mathcal Q_{g,V,y}}\}[\psi]\bigr|
\le C\|\psi\|_{\Xi^k}.
\]

\item[(iii)] \textbf{Fourier-domain representer identity}

\[
\mathcal F^{-1}
\left\{
\mathcal I_{ \omega  \; \mapsto  \; \mathcal{Q}_{g,V,y}(\omega)\varphi_V(-\omega)}
\right\}
=
\mathcal I_{\lambda_{g,V,y}}
\qquad
\text{in }\mathscr S'(\mathbb R^d)
\]

\end{enumerate}

Then \(\mathcal F^{-1}\{\mathcal I_{\mathcal Q_{g,V,y}}\}\) extends uniquely
to a continuous linear functional on \(\Xi^k(\mathbb R^d;\mathbb C)\), and its
restriction to \(\mathcal A_V(\mathbb R^d)\) is the Tweedie functional
\(\mathcal T_{g,V,y}\). Equivalently, for every
\(f\in \mathcal A_V(\mathbb R^d)\) and every sequence
\((f_n)_{n=1}^\infty\subset \mathscr S(\mathbb R^d)\) with
\(\|f_n-f\|_{\Xi^k}\to 0\),
\[
\mathcal T_{g,V,y}[f]
=
\lim_{n\to\infty}
\mathcal F^{-1}\{\mathcal I_{\mathcal Q_{g,V,y}}\}[f_n],
\]
and the limit is independent of the approximating sequence.
\end{theorem}

Theorem~\ref{T:master-fourier-identification} turns the abstract existence of
the Tweedie functional into a closed-form formula. The construction proceeds by identifying a complex-valued function \(\mathcal Q_{g,V,y}\) with three properties. Condition~(i) is an integrability condition. It ensures that \(\mathcal Q_{g,V,y}\) defines a tempered distribution by integration. Its inverse Fourier transform is therefore well defined as a tempered distribution, and hence as a linear functional on the Schwartz space. Condition~$(ii)$ is where the ambient space \(\Xi^k(\mathbb R^d;\mathbb C)\) enters. It requires this Schwartz-level functional to be continuous with respect to the \(\Xi^k\)-norm, and therefore to extend continuously to \(\Xi^k(\mathbb R^d;\mathbb C)\). 

Conditions~$(i)$ and~$(ii)$ together ensure that \(\mathcal Q_{g,V,y}\) induces a well-defined continuous functional on \(\Xi^k(\mathbb R^d;\mathbb C)\). They do not, by themselves, identify this functional with the Tweedie functional. Identification comes from Condition~$(iii)$, the Fourier-domain representer identity: after accounting for the noise, the functional generated by \(\mathcal Q_{g,V,y}\) must agree with the functional induced by integration against \(\lambda_{g,V,y}\). Any function \(\mathcal Q_{g,V,y}\) satisfying this identity is called a \textit{Fourier-domain representer} for \(\lambda_{g,V,y}\). When Conditions~$(i)$--$(iii)$ all hold, the theorem shows that the Tweedie functional can be computed by evaluating the inverse Fourier transform of \(\mathcal I_{\mathcal Q_{g,V,y}}\) on Schwartz approximations to any \(V\)-admissible mixture and passing to the \(\Xi^k\)-limit.

\begin{remark}
Suppose further that 
\(\lambda_{g,V,y}\in L^1(\mathbb R^d;\mathbb C)\). Under
Assumption~\ref{A: Identification}, one candidate Fourier-domain representer is
\[
\mathcal Q_{g,V,y}(\omega)
=
\frac{\widetilde{\lambda}_{g,V,y}(\omega)}
{\varphi_V(-\omega)}
\mathbbm 1\{\varphi_V(-\omega)\neq 0\}.
\]

Indeed, for every \(\psi\in\mathscr S(\mathbb R^d)\), Fourier duality gives
\[
\begin{aligned}
\mathcal F^{-1}
\left\{
\mathcal I_{\omega\mapsto \mathcal Q_{g,V,y}(\omega)\varphi_V(-\omega)}
\right\}[\psi]=
\int_{\mathbb R^d}
\widetilde{\lambda}_{g,V,y}(\omega)\psi^\sharp(\omega)\,d\omega =
\int_{\mathbb R^d}
\lambda_{g,V,y}(\omega)\psi(\omega)\,d\omega =
\mathcal I_{\lambda_{g,V,y}}[\psi].
\end{aligned}
\]
Thus Condition~\textup{(iii)} of Theorem~\ref{T:master-fourier-identification} holds. 
\end{remark}

The preceding remark shows that, when \(\lambda_{g,V,y}\) is integrable, a Fourier-domain representer \textit{is always available} and can be computed directly. Moreover, this integrability requirement is often mild. Under Gaussian noise, for example, the exponential decay of the noise density accommodates functions \(g\) with polynomial growth of arbitrary degree. In such cases, because a Fourier-domain representer is already known, applying Theorem~\ref{T:master-fourier-identification} only requires verifying the two analytic conditions, Conditions~\textup{(i)} and~\textup{(ii)}, for the resulting \(\mathcal Q_{g,V,y}\). When \(\lambda_{g,V,y}\) is not integrable, a Fourier-domain representer may still be obtained by direct inspection of the Fourier-domain representer identity.\footnote{See for instance the Cauchy posterior mean example in Appendix \ref{A:TableEntries}.}

Pragmatically, the theorem thus gives a three-step strategy for computing Tweedie functionals:

\begin{enumerate}[label=\textbf{Step \arabic*.}, leftmargin=*, align=left]
    \item Find a Fourier-domain representer
    \(\mathcal Q_{g,V,y}(\omega)\) and verify it defines a tempered distribution by integration.

    \item Compute the inverse Fourier transform
    \(\mathcal F^{-1}\{\mathcal I_{\mathcal Q_{g,V,y}}\}\) on Schwartz space
    \(\mathscr S(\mathbb R^d)\).

 \item Prove \(\Xi^k\)-continuity and compute
the Tweedie functional evaluated at a $V-$mixture by taking the \(\Xi^k\)-limit along
a Schwartz approximating sequence.
\end{enumerate}

I now apply this strategy to recover the classical Tweedie formula.

\begin{example}[Original Tweedie's formula.]
Suppose \(d=1\), \(Y=X+V\), \(X\indep V\), and
\(V\sim\mathcal N(0,\sigma^2)\). Let \(g(x)=x\), fix \(y\in\mathbb R\), and
take \(k=1\). Since the Gaussian density belongs to \(C_0^\infty(\mathbb R)\),
the smoothness assumption holds. First, let's compute a Fourier-domain representer. Here
\(x \mapsto\lambda_{g,V,y}(x)=x f_V(y-x)\in\Xi^0(\mathbb R;\mathbb C)\), and
\[
\widetilde{\lambda}_{g,V,y}(\omega)
=
\bigl(y+i\sigma^2\omega\bigr)e^{i\omega y}
e^{-\sigma^2\omega^2/2}.
\]
Since \(\varphi_V(-\omega)=e^{-\sigma^2\omega^2/2}\), the remark gives a Fourier-domain representer via
\[
\mathcal Q_{g,V,y}(\omega)=\dfrac{\widetilde{\lambda}_{g,V,y}(\omega)}{\varphi_V(-\omega)}
=
\bigl(y+i\sigma^2\omega\bigr)e^{i\omega y}.
\]
Second, since \(\mathcal Q_{g,V,y}\) grows at most linearly, it defines a
tempered distribution by integration, and the standard transform identities
\citep[Appendix 2]{kammler2007first} yield
\[
\mathcal F^{-1}\{\mathcal I_{\mathcal Q_{g,V,y}}\}[\psi]
=
y\psi(y)+\sigma^2\psi'(y),
\qquad \psi\in\mathscr S(\mathbb R).
\]
Third,
\[
\left|y\psi(y)+\sigma^2\psi'(y)\right|
\le
\bigl(|y|+\sigma^2\bigr)\|\psi\|_{\Xi^1},
\]
so the \(\Xi^1\)-continuity condition holds. Therefore
Theorem~\ref{T:master-fourier-identification} implies that, for any
\((f_n)_{n\ge1}\subset\mathscr S(\mathbb R)\) with
\(\|f_n-f_Y\|_{\Xi^1}\to0\),
\[
\mathcal T_{g,V,y}[f_Y]
=
\lim_{n\to\infty}
\{y f_n(y)+\sigma^2 f_n'(y)\}
=
y f_Y(y)+\sigma^2 f_Y'(y),
\]
since \(f_n\to f_Y\) in \(\Xi^1\), implies both \(f_n(y)\to f_Y(y)\) and \(f_n'(y)\to f_Y'(y)\). Dividing by
\(f_Y(y)\), whenever \(f_Y(y)>0\), gives
\[
\mathbb E[X\mid Y=y]
=
y+\sigma^2\frac{f_Y'(y)}{f_Y(y)}.
\]    
\end{example}

\subsection{Tweedie formulas via approximation by continuous functions} \label{SS: Approximation}

Theorem~\ref{T:master-fourier-identification} gives a systematic way to
compute Tweedie representations when the Fourier-domain representer
\(\mathcal Q_{g,V,y}\) defines a tempered distribution. Some 
functionals fall outside this setting. For example, the function \(g\)
defining the estimand may be discontinuous. In such cases, we can still apply the strategy by
approximating the original problem with instances for which the tools apply. The next
result gives one such approximation argument.

\begin{proposition}[Tweedie functionals via approximation]
\label{P:tweedie-approximation}
Let $(Y,X,V)$ satisfy Assumptions~\ref{A: Identification} and~\ref{A: Density}. Fix \(y\in\mathbb R^d\) and a measurable function
\(g:\mathbb R^d\to\mathbb R\). Let \((g_n)_{n\ge1}\) be a sequence of
measurable functions. Define $\lambda_{n,g,V,y}(x)=g_n(x)f_V(y-x)$ and $\lambda_{g,V,y}(x)=g(x)f_V(y-x)$. Assume that $\lambda_{n,g,V,y} \in C_0(\mathbb{R}^d;\mathbb C)$ for every $n \geq 1$, that $\lambda_{n,g,V,y}(x) \to \lambda_{g,V,y}(x)$ for all $x\in\mathbb R^d$, and that $\sup_{n\ge1}\|\lambda_{n,g,V,y}\|_\infty<\infty$. Then, for each \(n\ge1\), there exists a unique functional $\mathcal T_{n,g,V,y}\in \mathcal A_V'(\mathbb R^d)$ such that
\[
\mathcal T_{n,g,V,y}[f_V*\mu]
=
\int_{\mathbb R^d}\lambda_{n,g,V,y}(x)\,d\mu(x),
\qquad
\mu\in\mathcal M(\mathbb R^d).
\]

Moreover, there exists a unique continuous linear functional $\mathcal T_{g,V,y}\in \mathcal A_V'(\mathbb R^d)$ such that
\[
\mathcal T_{g,V,y}[f_V*\mu]
=
\int_{\mathbb R^d}\lambda_{g,V,y}(x)\,d\mu(x),
\qquad
\mu\in\mathcal M(\mathbb R^d).
\]
In fact, for every \(f\in\mathcal A_V(\mathbb R^d)\), this map is given as a pointwise limit
\[
\mathcal T_{g,V,y}[f]
=
\lim_{n\to\infty}
\mathcal T_{n,g,V,y}[f].
\]

In particular, if \((Y,X,V)\) also satisfies Assumption~\ref{A: independence}, then
\(f_Y=f_V*P_X\), and whenever \(f_Y(y)>0\),
\[
\mathbb E[g(X)\mid Y=y]
=
\frac{1}{f_Y(y)}
\lim_{n\to\infty}
\mathcal T_{n,g,V,y}[f_Y].
\]
\end{proposition}

The next example illustrates how this approximation argument produces a
Tweedie representation beyond the direct scope of
Theorem~\ref{T:master-fourier-identification}

\begin{example}[Posterior distribution function under Gaussian noise.]
 Suppose \(d=1\), \(Y=X+V\), \(X\indep V\), and
\(V\sim \mathcal N(0,\sigma^2)\).  We seek a Tweedie-type representation for
the posterior distribution function, $\mathbb P[X\le x\mid Y=y]$, which corresponds to the measurable function $g_x(u)=\mathbbm 1\{u\le x\}$, $x\in\mathbb R$.  For \(n\ge1\),  define the
smooth approximation
\[
g_{n,x}(u)
=
\Phi\left(\frac{x+n^{-1/2}-u}{n^{-1}}\right),
\qquad u\in\mathbb R.
\]

It is straightforward to check that \((g_{n,x})_{n\ge 1}\) satisfies the
conditions of Proposition~\ref{P:tweedie-approximation}. Moreover, each \(g_{n,x}\) gives

\begin{align*} \mathcal T_{n,g_x,V,y}[f] &= f(y)\Phi\left(\frac{x+n^{-1/2}-y}{\sqrt{\sigma^2+n^{-2}}}\right) +\mathsf S_n[f], \quad f \in \mathcal{A}_V(\mathbb{R}),\\ 
\mathsf S_n[f] &\coloneq \sum_{k=1}^{\infty} \frac{(-1)^k}{k!} \left( \frac{\sigma^2}{\sqrt{\sigma^2+n^{-2}}} \right)^k \Phi^{(k)} \left( \frac{x+n^{-1/2}-y}{\sqrt{\sigma^2+n^{-2}}} \right) f^{(k)}(y). 
\end{align*}

where $\Phi$ is the cumulative distribution function of a standard normal variable. For each fixed \(n\), this series converges absolutely.\footnote{Details of
this calculation are given in Lemma~\ref{L:SeriesRepresentation}.} Therefore,
\begin{align*}
\mathbb P[X\le x\mid Y=y]
&=\frac{1}{f_Y(y)}
\lim_{n\to\infty}
\mathcal T_{n,g_x,V,y}[f_Y]\\
&=
\Phi\left(\frac{x-y}{\sigma}\right)
+
\frac{1}{f_Y(y)}
\lim_{n\to\infty}\mathsf S_n[f_Y]
.
\end{align*}   
\end{example}

To the best of my knowledge, this representation is new. It also has a simple
interpretation. The first term,
\(\Phi((x-y)/\sigma)\), is the posterior distribution function obtained under a
flat, or locally uninformative, prior on \(X\). The remaining terms correct this
flat-prior approximation by using information in the marginal density \(f_Y\)
and its derivatives.

The formula also shows why posterior distribution functions are harder to
recover than posterior means. In the Gaussian case, the posterior mean depends
only on the first derivative of \(f_Y\), whereas the posterior cdf involves
derivatives of \(f_Y\) of all orders. In special cases, moreover, the expression
simplifies. For example, if \(X\sim\mathcal N(m,\tau^2)\), then $X\mid Y=y
\sim
\mathcal{N}\left(
\frac{\sigma^2m+\tau^2y}{\sigma^2+\tau^2},
\frac{\sigma^2\tau^2}{\sigma^2+\tau^2}
\right)$. In this case, the derivatives of \(f_Y\) are controlled, and the limit in
\(n\) can be interchanged with the series. Mehler's formula \citep{mehler1866ueber} then confirms the closed-form series expansion
\[
\frac{1}{f_Y(y)}
\sum_{k=0}^{\infty}
\frac{(-1)^k \sigma^k}{k!}
\Phi^{(k)}
\left(
\frac{x-y}{\sigma}
\right)
f_Y^{(k)}(y)
=
\Phi\left(
\frac{
x-\frac{\sigma^2m+\tau^2y}{\sigma^2+\tau^2}
}{
\sqrt{\frac{\sigma^2\tau^2}{\sigma^2+\tau^2}}
}
\right)=\mathbb P[X\le x\mid Y=y].
\]

\section{Applications of Tweedie Calculus} \label{S: Applications}

This section illustrates the range of Tweedie calculus through four applications.
First, I derive posterior mean formulas for several non-Gaussian additive-noise
models. Second, I identify cases in which Tweedie functionals can be written
as expectations of explicit functions of the observed variable. Third, I study
posterior functionals under the conventional Laplace mechanism used in
differential privacy. Fourth, I extend the calculus to the heteroskedastic Gaussian
sequence model by conditioning on the potentially random noise covariance structure. The detailed 
calculations supporting these applications are relegated to
Appendix~\ref{A: DetailedComp}.

\subsection{Posterior means beyond Gaussian  noise}

The first application concerns posterior means under additive noise laws other
than the Gaussian. As noted before, the three-step strategy from Section~\ref{S: Computation}
can be applied case by case. Alternatively, the same tools can also identify entire families of
Tweedie functionals from the structure of their Fourier-domain representers.
The next proposition gives one such characterization. It generalizes the
original Gaussian calculation by showing that certain factorizations of
\(\mathcal Q_{g,V,y}\) yield weighted-derivative representations.

\begin{proposition}[Tweedie representation via factorization of measures]
\label{P:measure-factorization-functional}
Fix \(y\in\mathbb R^d\) and let \(g:\mathbb R^d\to\mathbb R\) be measurable.
Suppose that \((Y,X,V)\) satisfies Assumptions~\ref{A: independence},
\ref{A: Identification}, \ref{A: Density}, and
\ref{A:SmoothBoundedNoise} with smoothness order \(k\in\mathbb N_0\). Assume $\lambda_{g,V,y} \in C_0(\mathbb{R}^d,\mathbb{C})$. Suppose there exists a family of finite complex Borel measures
\(\{\mu_\alpha:|\alpha|\le k\}\subset \mathcal M(\mathbb R^d)\) such that the function 
\[
\mathcal Q_{g,V,y}(\omega)
=
\sum_{|\alpha|\le k}(i\omega)^\alpha\varphi_{\mu_\alpha}(\omega), \qquad \varphi_\mu(\omega)
\coloneq
\int_{\mathbb{R}^d} e^{i\omega^\top x}\,d\mu(x),
\qquad\text{for Lebesgue-a.e. }\omega\in\mathbb R^d.
\]
\noindent is a Fourier-domain representer for $\lambda_{g,V,y}$. Then, for every
\(f\in\mathcal A_V(\mathbb R^d)\),
\[
\mathcal T_{g,V,y}[f]
=
\sum_{|\alpha|\le k}
\int_{\mathbb R^d}\partial^\alpha f(z)\,d\mu_\alpha(z).
\]
In particular, whenever \(f_Y(y)>0\),
\[
\mathbb E[g(X)\mid Y=y]
=
\frac{1}{f_Y(y)}
\sum_{|\alpha|\le k}
\int_{\mathbb R^d}\partial^\alpha f_Y(z)\,d\mu_\alpha(z).
\]
\end{proposition}

Proposition~\ref{P:measure-factorization-functional} covers the case in which
the Fourier-domain representer \(\mathcal Q_{g,V,y}\) factors into powers of
\(\omega\) whose coefficients are Fourier--Stieltjes transforms of finitely
many measures. In this case, the Tweedie functional is a finite linear
combination of weighted derivatives of \(f_Y\).  The original 
Tweedie formula is one instance of this factorization. We now turn to a
non-Gaussian example and show that the same method applies just as naturally
under Laplace noise.

\begin{example}[Posterior mean under Laplace noise.]
Suppose \(d=1\), \(Y=X+V\), $X \indep V$, and
\(V\) has Laplace density \(f_V(u)=(2b)^{-1}e^{-|u|/b}\), with \(b>0\).
Let \(g(x)=x\). Since \(f_V\in C_0(\mathbb R)\), take \(k=0\). For fixed
\(y\in\mathbb R\),
\[
\widetilde{\lambda}_{g,V,y}(\omega)
=
e^{i\omega y}
\left(
\frac{y}{1+b^2\omega^2}
+
\frac{2ib^2\omega}{(1+b^2\omega^2)^2}
\right).
\]
Since \(\varphi_V(-\omega)=(1+b^2\omega^2)^{-1}\), the quotient construction
gives a Fourier-domain representer for $\lambda_{g,V,y}$:
\[
\mathcal Q_{g,V,y}(\omega)=\dfrac{\widetilde{\lambda}_{g,V,y}(\omega)}{\varphi_V(-\omega)}
=
\left(
y+\frac{2ib^2\omega}{1+b^2\omega^2}
\right)e^{i\omega y}.
\]
Now \(e^{i\omega y}=\varphi_{\delta_y}(\omega)\). If \(\nu_y\) denotes the
finite signed measure with density $d\nu_y(z)=\operatorname{sgn}(z-y)e^{-|z-y|/b}\,dz$, then

\[
\varphi_{\nu_y}(\omega)
=
\frac{2ib^2\omega}{1+b^2\omega^2}e^{i\omega y} \qquad \text{ and } \qquad \mathcal Q_{g,V,y}(\omega)
=y\,\varphi_{\delta_y}(\omega)+\varphi_{\nu_y}(\omega)=
\varphi_{y\delta_y+\nu_y}(\omega).
\]

By Proposition~\ref{P:measure-factorization-functional}, with
\(\mu_0=y\delta_y+\nu_y\), whenever \(f_Y(y)>0\),
\[
\mathbb E[X\mid Y=y]
=
y+
\frac{1}{f_Y(y)}
\int_{\mathbb R}
\operatorname{sgn}(z-y)e^{-|z-y|/b}f_Y(z)\,dz.
\]    
\end{example}

Other families are developed using the same techniques in
Appendix~\ref{A:Calculations}. I conclude this subsection by illustrating the flexibility of the method
across several noise distributions. Table~\ref{tab:noise-families} reports
posterior mean representations for these noise laws, with the corresponding
derivations collected in Appendix~\ref{A:TableEntries}.

Several comments are in order. First, despite their analytic differences, the resulting decision rules share a common shrinkage form. Each posterior mean equals the observation, after correcting for the noise location, plus a normalized adjustment that depends only on the observed marginal density $f_Y$. Thus the unknown prior distribution enters only through $f_Y$, and the adjustment can be interpreted as a shrinkage or correction term. The form of this term, however, depends sharply on the noise law.

Second, the table shows that statistical difficulty is distribution-specific. For Gaussian noise, the correction is local: it depends only on $f_Y(y)$ and $f_Y'(y)$. This locality appears to be exceptional. The other entries involve nonlocal functionals of $f_Y$. Notably, these pose a lower estimation difficulty.

\begin{landscape}
\begin{table}[ht]
\centering
\scriptsize                     % was \scriptsize
\setlength{\tabcolsep}{3.5pt}  % was 3.5pt
\renewcommand{\arraystretch}{1.0} % was 1.08

\begin{threeparttable}
\caption{Posterior mean Tweedie representations for several additive-noise models.}
\label{tab:noise-families}

\begin{tabular}{@{}lcc>{$\displaystyle}l<{$}@{}}
\toprule
Noise law
& $f_V(v)$
& $\varphi_V(\omega)$
& \multicolumn{1}{c@{}}{Tweedie representation of $\mathbb E[X\mid Y=y]$} \\
\midrule

\multicolumn{4}{@{}l}{\emph{Two-sided laws on $\mathbb R$}}\\[1pt]

\begin{tabular}[t]{@{}l@{}}
$\mathrm{Normal}(\mu,\sigma^2)$\\[-1pt]
$\mu\in\mathbb R,\ \sigma>0$
\end{tabular}
& $\displaystyle \frac{1}{\sqrt{2\pi}\sigma}\exp\!\left(-\frac{(v-\mu)^2}{2\sigma^2}\right)$
& $\displaystyle e^{i\mu\omega-\sigma^2\omega^2/2}$ 
& y-\mu+\sigma^2\frac{f_Y'(y)}{f_Y(y)}
\\
\addlinespace[\smallskipamount]

\begin{tabular}[t]{@{}l@{}}
$\mathrm{Generalized\ Laplace}(\mu,b,\lambda)$\\[-1pt]
$\mu\in\mathbb R,\ b>0,\ \lambda>\frac12$
\end{tabular}
& $\displaystyle \frac{1}{\sqrt{\pi}\,\Gamma(\lambda)\,b}\Bigl(\frac{|v-\mu|}{2b}\Bigr)^{\lambda-\frac12}
K_{\lambda-\frac12}\!\Bigl(\frac{|v-\mu|}{b}\Bigr)$
& $\displaystyle e^{i\mu\omega}(1+b^2\omega^2)^{-\lambda}$
& \begin{aligned}[t]
y-\mu
&+\frac{\lambda}{f_Y(y)}
\left[
\int_{\mathbb R}\operatorname{sgn}(z-y)e^{-|z-y|/b}f_Y(z)\,dz
\right]
\end{aligned}
\\
\addlinespace[\smallskipamount]

\begin{tabular}[t]{@{}l@{}}
$\mathrm{Laplace}(\mu,b)$\\[-1pt]
$\mu\in\mathbb R,\ b>0$
\end{tabular}
& $\displaystyle \frac{1}{2b}e^{-|v-\mu|/b}$
& $\displaystyle \frac{e^{i\mu\omega}}{1+b^2\omega^2}$
& \begin{aligned}[t]
y-\mu
&+\frac{1}{f_Y(y)}
\left[
\int_{\mathbb R}\operatorname{sgn}(z-y)e^{-|z-y|/b}f_Y(z)\,dz
\right]
\end{aligned}
\\
\addlinespace[\smallskipamount]

\begin{tabular}[t]{@{}l@{}}
$\mathrm{Asymmetric\ Laplace}(\mu,b_-,b_+)$\\[-1pt]
$\mu\in\mathbb R,\ b_->0,\ b_+>0$
\end{tabular}
& $\displaystyle \frac{1}{b_-+b_+} \left(\exp\!\Bigl(\frac{v-\mu}{b_-}\Bigr)\mathbf{1}_{\{v<\mu\}}
+\exp\!\Bigl(-\frac{v-\mu}{b_+}\Bigr)\mathbf{1}_{\{v\ge \mu\}}\right)$
& $\displaystyle \frac{e^{i\mu\omega}}{(1+i b_-\omega)(1-i b_+\omega)}$
& y-\mu+
\frac{1}{f_Y(y)}
\left[
\int_y^\infty e^{-(z-y)/b_-}f_Y(z)\,dz
-
\int_{-\infty}^y e^{(z-y)/b_+}f_Y(z)\,dz
\right]
\\
\addlinespace[\smallskipamount]

\begin{tabular}[t]{@{}l@{}}
$\mathrm{Logistic}(\mu,s)$\\[-1pt]
$\mu\in\mathbb R,\ s>0$
\end{tabular}
& $\displaystyle \frac{e^{-(v-\mu)/s}}{s\bigl(1+e^{-(v-\mu)/s}\bigr)^2}$
& $\displaystyle e^{i\mu\omega}\frac{\pi s\omega}{\sinh(\pi s\omega)}$
& \begin{aligned}[t]
y-\mu
&+\frac{1}{f_Y(y)}
\left[
\int_0^\infty
\frac{f_Y(y+t)-f_Y(y-t)}{e^{t/s}-1}\,dt
\right]
\end{aligned}
\\
\addlinespace[\smallskipamount]

\begin{tabular}[t]{@{}l@{}}
$\mathrm{Gumbel}(\mu,\beta)$\\[-1pt]
$\mu\in\mathbb R,\ \beta>0$
\end{tabular}
& $\displaystyle \frac{1}{\beta}\exp\!\left[-\frac{v-\mu}{\beta}-e^{-(v-\mu)/\beta}\right]$
& $\displaystyle e^{i\mu\omega}\Gamma(1-i\beta\omega)$
& \begin{aligned}[t]
y-\mu-\beta\gamma_E
&+\frac{1}{f_Y(y)}
\left[
\int_0^\infty
\frac{f_Y(y)-f_Y(y-u)}{e^{u/\beta}-1}\,du
\right]
\end{aligned}
\\
\addlinespace[\smallskipamount]

\begin{tabular}[t]{@{}l@{}}
$\mathrm{Cauchy}(\mu,\gamma)$\\[-1pt]
$\mu\in\mathbb R,\ \gamma>0$
\end{tabular}
& $\displaystyle \frac{1}{\pi}\frac{\gamma}{(v-\mu)^2+\gamma^2}$
& $\displaystyle e^{i\mu\omega-\gamma|\omega|}$
& y-\mu-\gamma\,\frac{\mathcal H[f_Y](y)}{f_Y(y)}
\\
\addlinespace[\smallskipamount]

\begin{tabular}[t]{@{}l@{}}
$\mathrm{Hyperbolic\ Secant}(\mu,s)$\\[-1pt]
$\mu\in\mathbb R,\ s>0$
\end{tabular}
& $\displaystyle \frac{1}{2s}\operatorname{sech}\!\left(\frac{\pi(v-\mu)}{2s}\right)$
& $\displaystyle e^{i\mu\omega}\operatorname{sech}(s\omega)$
& \begin{aligned}[t]
y-\mu
&+\frac{1}{f_Y(y)}
\left[
\int_0^\infty
\frac{f_Y(y+t)-f_Y(y-t)}
{2\sinh\!\left(\frac{\pi t}{2s}\right)}\,dt
\right]
\end{aligned}
\\
\addlinespace[\smallskipamount]

\addlinespace[2pt]
\multicolumn{4}{@{}l}{\emph{One-sided laws on $(0,\infty)$}}\\[1pt]

\begin{tabular}[t]{@{}l@{}}
$\mathrm{Gamma}(\alpha,\theta)$\\[-1pt]
$\alpha>1,\ \theta>0$
\end{tabular}
& $\displaystyle \frac{1}{\Gamma(\alpha)\theta^\alpha}v^{\alpha-1}e^{-v/\theta}\mathbf{1}_{\{v>0\}}$
& $\displaystyle (1-i\theta\omega)^{-\alpha}$
& \begin{aligned}[t]
y
&-\frac{\alpha}{f_Y(y)}
\left[
\int_{-\infty}^y e^{(z-y)/\theta}f_Y(z)\,dz
\right]
\end{aligned}
\\

\addlinespace[\smallskipamount]

\begin{tabular}[t]{@{}l@{}}
$\mathrm{Noncentral}\ \chi^2_\nu(\delta)$\\[-1pt]
$\nu>2,\ \delta\ge 0$
\end{tabular}
&
$\begin{array}[t]{@{}ll@{}}
\displaystyle
\frac{1}{2}e^{-(v+\delta)/2}
\left(\frac{v}{\delta}\right)^{\nu/4-1/2}
I_{\nu/2-1}\!\bigl(\sqrt{\delta v}\bigr)
\mathbf 1_{\{v>0\}},
& \;\delta>0, \\[7pt]
\displaystyle
\frac{1}{2^{\nu/2}\Gamma(\nu/2)}
v^{\nu/2-1}e^{-v/2}\mathbf 1_{\{v>0\}},
& \;\delta=0
\end{array}$
&
$\displaystyle
(1-2i\omega)^{-\nu/2}
\exp\!\Bigl(\frac{i\delta\omega}{1-2i\omega}\Bigr)$
&
y-
\frac{1}{f_Y(y)}
\left[
\frac{\nu}{2}\int_{-\infty}^y e^{(z-y)/2}f_Y(z)\,dz
+
\frac{\delta}{4}\int_{-\infty}^y (y-z)e^{(z-y)/2}f_Y(z)\,dz
\right]
\\[2pt]

\addlinespace[\smallskipamount]

\begin{tabular}[t]{@{}l@{}}
$\mathrm{Inverse\ Gaussian}(\mu,\lambda)$\\[-1pt]
$\mu>0,\ \lambda>0$
\end{tabular}
& $\displaystyle \left(\frac{\lambda}{2\pi v^3}\right)^{1/2}
\exp\!\left(-\frac{\lambda(v-\mu)^2}{2\mu^2v}\right)\mathbf{1}_{\{v>0\}}$
& $\displaystyle \exp\!\left[\frac{\lambda}{\mu}\left(1-\sqrt{1-\frac{2i\mu^2\omega}{\lambda}}\right)\right]$
& y-
\frac{1}{f_Y(y)}
\left[
\sqrt{\frac{\lambda}{2\pi}}
\int_{-\infty}^y
(y-z)^{-1/2}
\exp\!\left(-\frac{\lambda(y-z)}{2\mu^2}\right)
f_Y(z)\,dz
\right]
\\[2pt]

\bottomrule
\end{tabular}

\begin{tablenotes}[flushleft]
\footnotesize
\item \textit{Notes:} Parameter restrictions are displayed below each distribution and indicate the conditions under which the displayed Tweedie representation is derived in the appendix. Here $\Gamma$ denotes the Euler gamma function, $I_\nu$ and $K_\nu$ the modified Bessel functions of the first and second kind, respectively, $\mathcal H$ the Hilbert transform, $\gamma_E$ the Euler--Mascheroni constant, $\operatorname{sgn}(x)$ the sign function, and $\operatorname{sech}(x)=2/(e^x+e^{-x})$.
\end{tablenotes}

\end{threeparttable}
\end{table}
\end{landscape}

Finally, the table illustrates the value of recasting the problem as Fourier
calculus. The Cauchy entry is especially instructive. Its posterior mean
involves the Hilbert transform of the observed density, a representation that
would be hard to guess without the language of tempered distributions. In the present
approach, it follows from the standard identity identifying the Hilbert
transform with the inverse Fourier transform of the integration map defined by 
\(\omega\mapsto -i\operatorname{sgn}(\omega)\)
\citep[Chapter 5, Equation 5.1.12]{grafakos2008classical}.

\subsection{Tweedie functionals given as an  expectation}
Appendix \ref{A:integrable-fourier-representers} considers Tweedie representations in the case where the Fourier-domain representer is integrable, $\mathcal Q_{g,V,y}\in L^1(\mathbb R^d;\mathbb C)$. In this regime, the Tweedie functional can be written as an expectation of the inverse Fourier transform of $\mathcal Q_{g,V,y}$.

This representation connects the framework to the classical literature on unbiased estimation of functions of latent variables under additive noise. In the Gaussian case, the problem goes back to \cite{kolmogorov1950unbiased}. The present result obtains the representation in closed form from the Fourier-domain representer. Other frameworks, notably \cite{stefanski1989unbiased} and \cite{voinov2012unbiased}, derive related formulas through power-series expansions of analytic functions.

\subsection{Functionals under the conventional Laplace mechanism} \label{SS: Laplace}

Differential privacy is a standard framework for protecting individual privacy in statistical and machine-learning applications. A canonical way to achieve differential privacy is to add noise to the output of a query before release. For vector-valued queries, such as Census counts across geographic or demographic cells, the conventional Laplace mechanism adds independent one-dimensional Laplace noise to each coordinate. This mechanism is defined in \cite[Definition~3.3]{dwork2014algorithmic} and gives standard differential privacy guarantees \citep[Theorem~3.6]{dwork2014algorithmic}.. 

This subsection studies Tweedie representations for posterior functionals after
such a release. Motivated by the Laplace mechanism, consider the model
\[
Y=X+V,
\qquad
X\indep V,
\qquad
V=(V_1,\ldots,V_d),
\qquad
V_j\overset{\mathrm{i.i.d.}}{\sim}\operatorname{Lap}(0,b).
\]

Table~\ref{tab:laplace-tweedie-functionals} collects several Tweedie
representations for the product-Laplace model in dimension one.
Appendix~\ref{A:LaplaceMechanism} provides the derivations and extends the
posterior mean, posterior covariance matrix, and posterior moment generating
function formulas to arbitrary dimensions.

\begin{landscape}
\begin{table}[ht]
\centering
\scriptsize
\setlength{\tabcolsep}{4pt}
\renewcommand{\arraystretch}{1.65}

\begin{threeparttable}
\caption{Examples of Tweedie representations under Laplace noise, \(d=1\).}
\label{tab:laplace-tweedie-functionals}

\begin{tabular}{@{}L{0.18\linewidth}C{0.18\linewidth}C{0.60\linewidth}@{}}
\toprule
Name of functional
&
\(g(x)\)
&
\(\mathbb E[g(X)\mid Y=y]\)
\\
\midrule

\multicolumn{3}{@{}l}{\textbf{Panel A. Tweedie representations for conditional expectations}}\\

Posterior mean
&
\(\displaystyle x\)
&
\wideformula{
y+
\frac{1}{f_Y(y)}
\int_{\mathbb R}
\operatorname{sgn}(u)e^{-|u|/b}f_Y(y+u)\,du
}
\\
\tablesep
\tablesep

Posterior second moment
&
\(\displaystyle x^2\)
&
\wideformula{
y^2
+
\frac{2y}{f_Y(y)}
\int_{\mathbb R}
\operatorname{sgn}(u)e^{-|u|/b}f_Y(y+u)\,du
+
\frac{1}{f_Y(y)}
\int_{\mathbb R}
(2|u|-b)e^{-|u|/b}f_Y(y+u)\,du
}
\\
\tablesep
\tablesep

Posterior moment generating function
&
\(\displaystyle e^{tx},\quad |t|<1/b\)
&
\wideformula{
e^{ty}
+\frac{e^{ty}}{f_Y(y)}
\left[
\int_{\mathbb R}
e^{tu}
\left\{
t\operatorname{sgn}(u)-\frac{bt^2}{2}
\right\}
e^{-|u|/b}f_Y(y+u)\,du
\right]
}
\\
\tablesep
\tablesep

Posterior distribution function at \(a\)
&
\(\displaystyle \mathbf 1\{x\le a\},\quad a\in\mathbb R\)
&
\wideformula{
\begin{cases}
\displaystyle
\frac{
e^{-(y-a)/b}
\left\{
f_Y(a)-bD_+f_Y(a)
\right\}
}{
2f_Y(y)
},
& a\le y,
\\[1.1em]
\displaystyle
1-
\frac{
e^{-(a-y)/b}
\left\{
f_Y(a)+bD_+f_Y(a)
\right\}
}{
2f_Y(y)
},
& a>y.
\end{cases}
}
\\
\tablesep
\tablesep

Posterior squared risk around \(a\)
&
\(\displaystyle (x-a)^2,\quad a\in\mathbb R\)
&
\wideformula{
(y-a)^2
+
\frac{2(y-a)}{f_Y(y)}
\int_{\mathbb R}
\operatorname{sgn}(u)e^{-|u|/b}f_Y(y+u)\,du
+
\frac{1}{f_Y(y)}
\int_{\mathbb R}
(2|u|-b)e^{-|u|/b}f_Y(y+u)\,du
}
\\
\tablesep
\tablesep

Posterior hinge loss around \(a\)
&
\(\displaystyle (x-a)_+,\quad a\in\mathbb R\)
&
\wideformula{
(y-a)_+
+
\frac{1}{f_Y(y)}
\int_{0}^{\infty}
e^{-|u|/b}f_Y(y+u)\,du
-
\frac{1}{f_Y(y)}
\int_{\min(0,a-y)}^{\max(0,a-y)}
e^{-|u|/b}f_Y(y+u)\,du
-
\frac{b e^{-|y-a|/b}f_Y(a)}{2f_Y(y)}
}
\\
\tablesep
\tablesep

Posterior pinball loss around \(a\)
&
\(\displaystyle \rho_\tau(x-a), \; a\in\mathbb R,\ \tau\in(0,1)\)
&
\(\displaystyle
\begin{aligned}
&
\tau(y-a)_+
+
(1-\tau)(a-y)_+
+
\frac{\tau}{f_Y(y)}
\int_{0}^{\infty}
e^{-|u|/b}f_Y(y+u)\,du
+
\frac{1-\tau}{f_Y(y)}
\int_{-\infty}^{0}
e^{-|u|/b}f_Y(y+u)\,du
\\
&\quad
-
\frac{1}{f_Y(y)}
\int_{\min(0,a-y)}^{\max(0,a-y)}
e^{-|u|/b}f_Y(y+u)\,du
-
\frac{b e^{-|y-a|/b}f_Y(a)}{2f_Y(y)}
\end{aligned}
\)
\\
\tablesep
\tablesep

\multicolumn{3}{@{}l}{\textbf{Panel B. Tweedie representations for transformations}}\\

Posterior variance
&
\(\displaystyle \text{--}\)
&
\wideformula{
\frac{1}{f_Y(y)}
\int_{\mathbb R}
(2|u|-b)e^{-|u|/b}f_Y(y+u)\,du
-
\left[
\frac{1}{f_Y(y)}
\int_{\mathbb R}
\operatorname{sgn}(u)e^{-|u|/b}f_Y(y+u)\,du
\right]^2
}
\\
\tablesep

\bottomrule
\end{tabular}

\begin{tablenotes}
\footnotesize
\item \textit{Notes:} Throughout, \(Y=X+V\), \(X\indep V\), \(V\sim\operatorname{Laplace}(0,b)\), \(b>0\), and \(f_Y(y)>0\). The operator \(D_+\) is the right derivative of the marginal density with respect to its scalar argument: \(D_+f_Y(a)=\lim_{h\downarrow0}\{f_Y(a+h)-f_Y(a)\}/h\). Here \((z)_+=\max\{z,0\}\), and \(\rho_\tau(z)=\tau z_+ +(1-\tau)(-z)_+\) for \(\tau\in(0,1)\).
\end{tablenotes}

\end{threeparttable}
\end{table}
\end{landscape}

\subsection{Tweedie representations in the heteroskedastic Gaussian sequence model} \label{S: Heteroskedastic}

The final application considers the heteroskedastic Gaussian sequence model.
Its distinctive feature is that the noise covariance varies across units and
may be statistically related to the latent parameter. This setting arises
naturally in empirical Bayes problems with many noisy estimates, especially in
economics and related social sciences.\footnote{Examples include value-added
modeling \citep{einav2025producing}, the evaluation of place-based policies
\citep{chetty2018impacts}, the measurement of discrimination
\citep{kline2022systemic}, policy targeting \citep{moon2025optimal}, and
research production \citep{abadie2026estimatingvalueevidencebaseddecision}.}
A common interpretation in these settings is that \(Y\) is an estimator of a
latent parameter \(X\), constructed from a sample mean or, more generally, from
an asymptotically linear statistic. The Gaussian specification formalizes the normal approximation implicit in standard large-sample inference for such estimators. Heteroskedasticity arises because sample sizes, sampling designs, or underlying population variability differ across units, making the covariance matrix unit-specific. It is therefore
natural to model \((X,\Sigma)\) as drawn from a joint superpopulation
distribution \(P_{X,\Sigma}\).

Formally, let \(\mathbb S_{++}^d\) denote the set of symmetric positive
definite \(d\times d\) matrices. The model is
\[
Y=X+\Sigma^{1/2}V,
\]
where \(V\sim\mathcal N(0,I_d)\), \((X,\Sigma)\sim P_{X,\Sigma}\), and
\((X,\Sigma)\indep V\). It is common practice to consider $\Sigma$ to be known and equal to a sample-based estimate.

At first glance, heteroskedasticity seems to put the model outside the scope of
the preceding sections, since there is no single fixed noise distribution to
deconvolve. The key observation is that, after conditioning on
\(\Sigma=\Sigma_0\) and standardizing by \(\Sigma_0^{-1/2}\), the model becomes
an additive Gaussian model with identity noise covariance.\footnote{Conditional
statements are understood after fixing a version of \(P_{X\mid\Sigma}\). Since regular conditional laws are
unique up to null sets, the conclusions are unaffected by this choice outside a \(P_\Sigma\)-null set.} The results from
Sections~\ref{S: Existence} and~\ref{S: Computation} can therefore be applied
conditionally, yielding Tweedie representations in terms of the conditional
density \(f_{Y\mid\Sigma}(\cdot\mid\Sigma_0)\).

Critically, this conditional approach does not require an explicit model for
the dependence between the latent parameter \(X\) and the noise covariance
\(\Sigma\).   This distinguishes the approach from the existing empirical Bayes
literature on heteroskedastic normal means, which imposes parametric structure,
or other restrictions, on the joint distribution \(P_{X,\Sigma}\); see  \citep{chamberlain1984,jiang2010empirical,efron2016empirical,Weinstein2018,ignatiadis2019covariate,kline2024discrimination,chen2026empirical}.

For a measurable function \(g:\mathbb R^d\to\mathbb R\), the objects of
interest in this section are conditional posterior functionals of the form
\[
\mathbb E\!\left[g(X)\mid Y=y,\Sigma=\Sigma_0\right].
\]
If \(\Sigma\) were degenerate at \(\Sigma_0\in\mathbb S_{++}^d\), this would
reduce to the usual posterior quantity under Gaussian noise with covariance
\(\Sigma_0\). In the general case, however, \(\Sigma\) is random, so the model
is not initially of the additive-noise form studied in the previous sections.
Conditioning on \(\Sigma=\Sigma_0\) removes this difficulty. To see this,
define
\[
Z\coloneq \Sigma^{-1/2}Y,
\qquad
W\coloneq \Sigma^{-1/2}X.
\]

Then, conditional on \(\Sigma=\Sigma_0\), we have $Z=W+V$, $W\indep V$, $V\sim\mathcal N(0,I_d)$. In particular,
\[
\varphi_{Z\mid \Sigma}(\omega\mid \Sigma_0)
=
\varphi_{W\mid \Sigma}(\omega\mid \Sigma_0)\varphi_V(\omega).
\]

Thus, conditional on \(\Sigma=\Sigma_0\), the standardized model satisfies the
same structural conditions as the additive Gaussian model studied in
Sections~\ref{S: Existence} and~\ref{S: Computation}. The results developed
there therefore apply conditional on \(\Sigma\), and hence pointwise in
\(\Sigma_0\). The next proposition records this implication formally.

\begin{proposition}[Conditional Tweedie representation in the heteroskedastic Gaussian sequence model]
\label{P: conditional-heteroskedastic}
Fix \(y\in\mathbb R^d\), and for \(\Sigma_0\in\mathbb S_{++}^d\), set $z_0\coloneq \Sigma_0^{-1/2}y$.
Let \(g:\mathbb R^d\to\mathbb R\) be measurable. Then the following statements hold for \(P_\Sigma\)-almost every \(\Sigma_0\).

\begin{enumerate}[label=(\alph*)]
\item  \(f_{Z\mid \Sigma}(z_0\mid \Sigma_0)>0\), and whenever the
right-hand side below is well defined,
\[
\mathbb E[g(X)\mid Y=y,\Sigma=\Sigma_0]
=
\frac{
\displaystyle
\int_{\mathbb R^d}
g(\Sigma_0^{1/2}w)\phi(z_0-w)\,
dP_{W\mid\Sigma}(w\mid\Sigma_0)
}{
f_{Z\mid\Sigma}(z_0\mid\Sigma_0)
}.
\]

\item Define $\lambda_{g,y,\Sigma_0}(w)
\coloneq
g(\Sigma_0^{1/2}w)\phi(z_0-w)$. If \(\lambda_{g,y,\Sigma_0}\in C_0(\mathbb R^d;\mathbb C)\), then there exists a unique continuous linear functional $\mathcal T_{g,y,\Sigma_0}\in \mathcal A_V'(\mathbb R^d)$, such that
\[
\mathcal T_{g,y,\Sigma_0}[f]
=
\int_{\mathbb R^d}
\lambda_{g,y,\Sigma_0}(w)\,d\mu(w)
\qquad
\text{for every } f=\phi*\mu \in \mathcal{A}_V(\mathbb R^d).
\]
In particular,
\[
\mathcal T_{g,y,\Sigma_0}
\!\left[f_{Z\mid\Sigma}(\cdot\mid\Sigma_0)\right]
=
\int_{\mathbb R^d}
\lambda_{g,y,\Sigma_0}(w)\,
dP_{W\mid\Sigma}(w\mid\Sigma_0),
\]
and therefore
\[
\mathbb E[g(X)\mid Y=y,\Sigma=\Sigma_0]
=
\frac{
\mathcal T_{g,y,\Sigma_0}
\!\left[f_{Z\mid\Sigma}(\cdot\mid\Sigma_0)\right]
}{
f_{Z\mid\Sigma}(z_0\mid\Sigma_0)
}.
\]

\item Suppose, in addition to the condition in part \textup{(b)}, that there exists a Fourier-domain representer
\(\mathcal Q_{g,y,\Sigma_0}\) for \(\lambda_{g,y,\Sigma_0}\) such that
conditions \textup{(i)} and \textup{(ii)} of
Theorem~\ref{T:master-fourier-identification} hold for some \(N\ge0\) and
\(k\ge0\). Then, for every \(f\in\mathcal A_V(\mathbb R^d)\) and every
sequence \((f_n)_{n\ge1}\subset\mathscr S(\mathbb R^d)\) such that
\(\|f_n-f\|_{\Xi^k}\to0\),
\[
\mathcal T_{g,y,\Sigma_0}[f]
=
\lim_{n\to\infty}
\mathcal F^{-1}
\{\mathcal I_{\mathcal Q_{g,y,\Sigma_0}}\}[f_n],
\]
and the limit is independent of the approximating sequence.
\end{enumerate}
\end{proposition}

The proposition shows that, after conditioning on \(\Sigma=\Sigma_0\), the heteroskedastic problem can be treated using the same computational strategy developed in Section~\ref{S: Computation}. The existence and computation results from that section therefore apply pointwise in \(\Sigma_0\), with the marginal density \(f_Y\) replaced by the conditional density \(f_{Z\mid\Sigma}(\cdot\mid\Sigma_0)\), or equivalently by \(f_{Y\mid\Sigma}(\cdot\mid\Sigma_0)\) after the change of variables.

\begin{landscape}
\begin{table}[ht]
\centering
\scriptsize
\setlength{\tabcolsep}{4pt}
\renewcommand{\arraystretch}{1.50}

\begin{threeparttable}
\caption{ Examples of  conditional Tweedie representations in the heteroskedastic Gaussian sequence model, \(d=1\).}
\label{tab:conditional-tweedie-functionals}

\begin{tabular}{@{}
>{\raggedright\arraybackslash}p{0.225\linewidth}
>{\centering\arraybackslash}p{0.135\linewidth}
>{\centering\arraybackslash}p{0.615\linewidth}
@{}}
\toprule
Name of functional
&
\(g(x)\)
&
\(\mathbb E[g(X)\mid Y=y,\Sigma=\sigma^2]\)
\\
\midrule

\multicolumn{3}{@{}l}{\textbf{Panel A. Tweedie representations for conditional expectations}}\\

Posterior mean
&
\(\displaystyle x\)
&
\wideformula{
y+\sigma^2
\frac{\partial_y f_{Y\mid\Sigma}(y\mid\sigma^2)}
{f_{Y\mid\Sigma}(y\mid\sigma^2)}
}
\\
\tablesep
\tablesep
Posterior second moment
&
\(\displaystyle x^2\)
&
\wideformula{
y^2+\sigma^2
+2y\sigma^2
\frac{\partial_y f_{Y\mid\Sigma}(y\mid\sigma^2)}
{f_{Y\mid\Sigma}(y\mid\sigma^2)}
+\sigma^4
\frac{\partial_y^2 f_{Y\mid\Sigma}(y\mid\sigma^2)}
{f_{Y\mid\Sigma}(y\mid\sigma^2)}
}
\\
\tablesep
\tablesep
Posterior moment generating function
&
\(\displaystyle e^{tx},\quad t\in\mathbb R\)
&
\wideformula{
\exp\!\left(ty+\frac12\sigma^2t^2\right)
\frac{f_{Y\mid\Sigma}(y+\sigma^2t\mid\sigma^2)}
{f_{Y\mid\Sigma}(y\mid\sigma^2)}
}
\\
\tablesep
\tablesep
Posterior \(k\)-th raw moment
&
\(\displaystyle x^k,\quad k\in\mathbb N\)
&
\wideformula{
\sum_{r=0}^{k}
\binom{k}{r}\sigma^{2r}
\left[
\sum_{j=0}^{\lfloor (k-r)/2\rfloor}
\frac{(k-r)!}{2^j j!(k-r-2j)!}
\sigma^{2j}y^{k-r-2j}
\right]
\frac{\partial_y^r f_{Y\mid\Sigma}(y\mid\sigma^2)}
{f_{Y\mid\Sigma}(y\mid\sigma^2)}
}
\\
\tablesep
\tablesep
Posterior \(2m\)-risk around \(a\)
&
\(\displaystyle |x-a|^{2m},\quad a\in\mathbb R,\ m\in\mathbb N\)
&
\wideformula{
\sum_{r=0}^{2m}
\binom{2m}{r}\sigma^{2r}
\left[
\sum_{j=0}^{\lfloor (2m-r)/2\rfloor}
\frac{(2m-r)!}{2^j j!(2m-r-2j)!}
\sigma^{2j}(y-a)^{2m-r-2j}
\right]
\frac{\partial_y^r f_{Y\mid\Sigma}(y\mid\sigma^2)}
{f_{Y\mid\Sigma}(y\mid\sigma^2)}
}
\\
\tablesep
\tablesep
Posterior distribution function at \(a\)
&
\(\displaystyle \mathbf 1\{x\le a\},\quad a\in\mathbb R\)
&
\wideformula{
\Phi\!\left(\frac{a-y}{\sigma}\right)
+
\frac{1}{f_{Y\mid\Sigma}(y\mid\sigma^2)}
\lim_{n\to\infty}
\sum_{\ell=1}^{\infty}
\frac{(-1)^\ell}{\ell!}
\left(
\frac{\sigma^2}{\sqrt{\sigma^2+n^{-2}}}
\right)^\ell
\Phi^{(\ell)}
\!\left(
\frac{a+n^{-1/2}-y}{\sqrt{\sigma^2+n^{-2}}}
\right)
\partial_y^\ell f_{Y\mid\Sigma}(y\mid\sigma^2)
}
\\
\tablesep
\tablesep

Posterior hinge loss around \(a\)
&
\(\displaystyle (x-a)_+,\quad a\in\mathbb R\)
&
\(\displaystyle
\begin{aligned}
&
\sigma\phi\!\left(\frac{a-y}{\sigma}\right)
+
(y-a)
\left\{
1-\Phi\!\left(\frac{a-y}{\sigma}\right)
\right\}
+
\sigma^2
\left\{
1-\Phi\!\left(\frac{a-y}{\sigma}\right)
\right\}
\frac{\partial_y f_{Y\mid\Sigma}(y\mid\sigma^2)}
{f_{Y\mid\Sigma}(y\mid\sigma^2)}
\\
&\quad+
\frac{1}{f_{Y\mid\Sigma}(y\mid\sigma^2)}
\lim_{n\to\infty}
\sum_{\ell=2}^{\infty}
\frac{(-1)^\ell}{\ell!}
\left(
\frac{\sigma^2}{\sqrt{\sigma^2+n^{-2}}}
\right)^\ell
\sqrt{\sigma^2+n^{-2}}
\Phi^{(\ell-1)}
\!\left(
\frac{a+n^{-1/2}-y}{\sqrt{\sigma^2+n^{-2}}}
\right)
\partial_y^\ell f_{Y\mid\Sigma}(y\mid\sigma^2)
\end{aligned}
\)
\\
\tablesep
\tablesep

Posterior absolute risk around \(a\)
&
\(\displaystyle |x-a|,\quad a\in\mathbb R\)
&
\(\displaystyle
\begin{aligned}
&
2\sigma\phi\!\left(\frac{a-y}{\sigma}\right)
+
(a-y)
\left\{
2\Phi\!\left(\frac{a-y}{\sigma}\right)-1
\right\}
+
\sigma^2
\left\{
1-2\Phi\!\left(\frac{a-y}{\sigma}\right)
\right\}
\frac{\partial_y f_{Y\mid\Sigma}(y\mid\sigma^2)}
{f_{Y\mid\Sigma}(y\mid\sigma^2)}
\\
&\quad+
\frac{2}{f_{Y\mid\Sigma}(y\mid\sigma^2)}
\lim_{n\to\infty}
\sum_{\ell=2}^{\infty}
\frac{(-1)^\ell}{\ell!}
\left(
\frac{\sigma^2}{\sqrt{\sigma^2+n^{-2}}}
\right)^\ell
\sqrt{\sigma^2+n^{-2}}
\Phi^{(\ell-1)}
\!\left(
\frac{a+n^{-1/2}-y}{\sqrt{\sigma^2+n^{-2}}}
\right)
\partial_y^\ell f_{Y\mid\Sigma}(y\mid\sigma^2)
\end{aligned}
\)
\\
\multicolumn{3}{@{}l}{\textbf{Panel B. Tweedie representations for transformations}}\\

Posterior variance
&
\(\displaystyle \text{--}\)
&
\wideformula{
\sigma^2+\sigma^4
\left\{
\frac{\partial_y^2 f_{Y\mid\Sigma}(y\mid\sigma^2)}
{f_{Y\mid\Sigma}(y\mid\sigma^2)}
-
\left(
\frac{\partial_y f_{Y\mid\Sigma}(y\mid\sigma^2)}
{f_{Y\mid\Sigma}(y\mid\sigma^2)}
\right)^2
\right\}
}
\\
\tablesep
\tablesep
Posterior \(k\)-th centered moment

&
\(\displaystyle \text{--}\)
&
\wideformula{
\sum_{r=0}^{k}
\binom{k}{r}\sigma^{2r}
\left[
\sum_{j=0}^{\lfloor (k-r)/2\rfloor}
\frac{(k-r)!}{2^j j!(k-r-2j)!}
\sigma^{2j}
\left(
-\sigma^2
\frac{\partial_y f_{Y\mid\Sigma}(y\mid\sigma^2)}
{f_{Y\mid\Sigma}(y\mid\sigma^2)}
\right)^{k-r-2j}
\right]
\frac{\partial_y^r f_{Y\mid\Sigma}(y\mid\sigma^2)}
{f_{Y\mid\Sigma}(y\mid\sigma^2)}
}\\
\bottomrule
\end{tabular}

\begin{tablenotes}
\footnotesize
\item \textit{Notes:} Throughout, \(\sigma>0\), \(f_{Y\mid\Sigma}(y\mid\sigma^2)>0\), and all derivatives are with respect to \(y\). Also, \(\Phi\) denotes the standard normal distribution function and \(\Phi^{(\ell)}\) its \(\ell\)-th derivative. Here \((z)_+=\max\{z,0\}\).
\end{tablenotes}

\end{threeparttable}
\end{table}
\end{landscape}

I now apply this conditional Tweedie calculus to a range of empirical Bayes
estimands under this model. Table~\ref{tab:conditional-tweedie-functionals} summarizes the
one-dimensional formulas. Appendix~\ref{A:Examples} provides the derivations
and extends the formulas for the posterior mean, posterior covariance, and
posterior moment generating function to arbitrary dimensions.

%----------------------------------------------- Section

\section{Conclusion} \label{S: conclusion}

Under Gaussian noise, Tweedie's formula expresses the posterior mean as a functional of the observed marginal density. This paper develops a
general framework for analogous identities governing conditional expectations
in additive-noise models. The main contribution is twofold. First, I
characterize when posterior functionals admit a representation through a linear
map acting on the observed density---a structure I call the Tweedie
functional---and establish existence, uniqueness, and continuity under general
conditions. Second, I provide a constructive approach based on Fourier analysis
for deriving such representations, showing that the relevant operator can be
obtained by extending the inverse Fourier transform of an explicitly
characterized tempered distribution to the space of admissible observed
densities.

This perspective turns the search for Tweedie-type formulas into a tractable problem in the calculus of tempered distributions. It recovers the classical Gaussian identity, yields new posterior-mean representations for a broad family of noise laws, and, in some cases, delivers expressions that can be unbiasedly estimated from the data. The applications also cover posterior functionals under the Laplace mechanism used in differential privacy. Together, these results show that the same Fourier-distributional framework can handle different noise structures and different posterior functionals.

The heteroskedastic application shows that the method applies to extensions of the standard additive model. By conditioning on the realized covariance matrix and applying the appropriate change of variables, the model regains the structure needed for the theory to apply. This yields Tweedie representations for posterior means and other functionals directly in terms of the observable conditional density, without imposing additional structure on the model.

\newpage

\newpage

\bibliographystyle{apacite}
\bibliography{tweedie}

\newpage
\appendix

\renewcommand\thefigure{\thesection.\arabic{figure}}
\renewcommand{\thetable}{A-\arabic{table}}
\renewcommand{\thefigure}{A-\arabic{figure}}
\setcounter{table}{0}
\setcounter{figure}{0}

\begin{center}
    {\LARGE\bfseries Appendix for:} \\
    \vspace{0.4cm}
    {\LARGE\bfseries \textquote{Tweedie Calculus}}
\end{center}

\renewcommand{\partname}{}
\renewcommand{\thepart}{}

% Theorems/lemmas/etc. reset for the single appendix section.
\makeatletter
\@addtoreset{theorem}{section}
\@addtoreset{lemma}{section}
\@addtoreset{corollary}{section}
\@addtoreset{assumption}{section}
\makeatother

\renewcommand{\thetheorem}{\thesection-\arabic{theorem}}
\renewcommand{\thelemma}{\thesection-\arabic{lemma}}
\renewcommand{\thecorollary}{\thesection-\arabic{corollary}}
\renewcommand{\theassumption}{\thesection-\arabic{assumption}}

\thispagestyle{empty}

\part{} % Start the appendix part

\startcontents[appendix]
\printcontents[appendix]{}{2}{\noindent\textbf{\Large Contents}\vspace{0.5cm}\par}

\newpage

\section{Appendix}
\label{A:single-appendix}

\subsection{Schwartz functions and tempered distributions}
\label{A:schwartz-tempered-topology}

This appendix expands on the theory of tempered distributions used throughout the paper. We first define the topology on the Schwartz space, then describe the induced weak-\(\star\) topology on tempered distributions and the action of the Fourier transform by duality.

\subsubsection{The Schwartz topology}

Let \(\mathscr S(\mathbb R^d)\) denote the Schwartz space of complex-valued smooth functions:
\[
\mathscr S(\mathbb R^d)
=
\left\{
\psi\in C^\infty(\mathbb R^d;\mathbb C):
\sup_{x\in\mathbb R^d}
|x^\alpha \partial^\beta \psi(x)|<\infty
\text{ for all multi-indices } \alpha,\beta
\right\}.
\]
Following Chapter 8.1 of \cite{folland1999real}, define, for \(N\in\mathbb N_0\) and \(\alpha\in\mathbb N_0^d\),
\[
p_{N,\alpha}(\psi)
\coloneq
\sup_{x\in\mathbb R^d}
(1+\|x\|_2)^N
\left|\partial^\alpha\psi(x)\right|.
\]
The standard Schwartz topology is the locally convex topology generated by the seminorms
\[
\left\{
p_{N,\alpha}:
N\in\mathbb N_0,\ \alpha\in\mathbb N_0^d
\right\}.
\]
With this topology, \(\mathscr S(\mathbb R^d)\) is a Fr\'echet space \citep{folland1999real}, Proposition 8.2. Thus
\[
\psi_n\to\psi
\quad\text{in } \mathscr S(\mathbb R^d)
\]
if and only if
\[
p_{N,\alpha}(\psi_n-\psi)\to0
\qquad
\text{for every } N\in\mathbb N_0
\text{ and } \alpha\in\mathbb N_0^d .
\]

It will be useful to work with the equivalent increasing family
\[
q_m(\psi)
\coloneq
\max_{\substack{0\le N\le m\\ |\alpha|\le m}}
p_{N,\alpha}(\psi),
\qquad
m\in\mathbb N_0 .
\]
The seminorms \((q_m)_{m\ge0}\) generate the same topology as
\((p_{N,\alpha})_{N,\alpha}\). Hence
\[
\psi_n\to\psi
\quad\text{in } \mathscr S(\mathbb R^d)
\quad\Longleftrightarrow\quad
q_m(\psi_n-\psi)\to0
\text{ for every } m\in\mathbb N_0 .
\]

\subsubsection{Fourier transform on the Schwartz space}

We use the Fourier transform convention
\[
\widetilde\psi(\omega)
=
\int_{\mathbb R^d} e^{i\omega^\top x}\psi(x)\,dx,
\qquad
\psi\in\mathscr S(\mathbb R^d).
\]
Under this convention, differentiation and multiplication satisfy, for
\(j=1,\ldots,d\),
\[
\partial_{\omega_j}\widetilde\psi(\omega)
=
\widetilde{\,i x_j\psi\,}(\omega),
\qquad
\widetilde{\partial_{x_j}\psi}(\omega)
=
-i\omega_j\widetilde\psi(\omega).
\]
Consequently, each Schwartz seminorm of \(\widetilde\psi\) is controlled by finitely many Schwartz seminorms of \(\psi\). Equivalently, for every \(m\in\mathbb N_0\), there exist constants \(C_m>0\) and \(r_m\in\mathbb N_0\) such that
\[
q_m(\widetilde\psi)
\le
C_m q_{r_m}(\psi),
\qquad
\psi\in\mathscr S(\mathbb R^d).
\]
Thus the Fourier transform is continuous on \(\mathscr S(\mathbb R^d)\). The same statement holds for the inverse Fourier transform, so the Fourier transform is a continuous linear automorphism of \(\mathscr S(\mathbb R^d)\).

\subsubsection{Tempered distributions and weak-\texorpdfstring{\(\star\)}{star} convergence}

The space of tempered distributions is the continuous dual of the Schwartz space:
\[
\mathscr S'(\mathbb R^d)
\coloneq
\bigl(\mathscr S(\mathbb R^d)\bigr)'.
\]
Thus \(T\in\mathscr S'(\mathbb R^d)\) is a continuous linear functional
\[
T:\mathscr S(\mathbb R^d)\to\mathbb C .
\]
Since the seminorms \((q_m)_{m\ge0}\) are increasing and generate the Schwartz topology, \(T\in\mathscr S'(\mathbb R^d)\) if and only if there exist \(C>0\) and \(m\in\mathbb N_0\) such that
\[
|T[\psi]|
\le
C q_m(\psi),
\qquad
\psi\in\mathscr S(\mathbb R^d).
\]
Equivalently, there exist \(C>0\) and a finite set
\(F\subset \mathbb N_0\times\mathbb N_0^d\) such that
\[
|T[\psi]|
\le
C
\sum_{(N,\alpha)\in F}
p_{N,\alpha}(\psi),
\qquad
\psi\in\mathscr S(\mathbb R^d).
\]

The weak-\(\star\) topology on \(\mathscr S'(\mathbb R^d)\) is the topology of pointwise convergence on \(\mathscr S(\mathbb R^d)\). Hence a net
\((T_i)_{i\in I}\subset\mathscr S'(\mathbb R^d)\) converges to
\(T\in\mathscr S'(\mathbb R^d)\) if and only if
\[
T_i[\psi]\to T[\psi]
\qquad
\text{for every } \psi\in\mathscr S(\mathbb R^d).
\]

\subsubsection{Fourier transform on tempered distributions}

The Fourier transform on \(\mathscr S(\mathbb R^d)\) induces a Fourier transform on \(\mathscr S'(\mathbb R^d)\) by duality. With the convention above, define the inverse Fourier transform of
\(T\in\mathscr S'(\mathbb R^d)\) by
\[
\mathcal F^{-1}\{T\}[\psi]
=
T[\psi^\sharp],
\qquad
\psi^\sharp(\omega)
=
(2\pi)^{-d}\widetilde\psi(-\omega).
\]
The map \(\psi\mapsto\psi^\sharp\) is continuous from
\(\mathscr S(\mathbb R^d)\) into itself. Therefore
\(\mathcal F^{-1}\{T\}\in\mathscr S'(\mathbb R^d)\).

The map \(T\mapsto\mathcal F^{-1}\{T\}\) is weak-\(\star\) continuous. Indeed, if
\(T_i\to T\) weak-\(\star\), then, for every
\(\psi\in\mathscr S(\mathbb R^d)\),
\[
\mathcal F^{-1}\{T_i\}[\psi]
=
T_i[\psi^\sharp]
\to
T[\psi^\sharp]
=
\mathcal F^{-1}\{T\}[\psi].
\]
The same argument applies to the Fourier transform. Thus the Fourier transform and its inverse are weak-\(\star\) continuous automorphisms of \(\mathscr S'(\mathbb R^d)\).

\subsubsection{A sufficient condition for temperedness}

The next lemma gives the temperedness condition used in the main text.

\begin{lemma}[Weighted integrability implies temperedness]
\label{L:weighted-integrability-temperedness}
Let \(Q:\mathbb R^d\to\mathbb C\) be measurable. Suppose that, for some \(N\in\mathbb N_0\),
\[
A_N
\coloneq
\int_{\mathbb R^d}
\frac{|Q(\omega)|}{(1+\|\omega\|_2)^N}\,d\omega
<\infty .
\]
Then the map
\[
\mathcal I_Q[\psi]
\coloneq
\int_{\mathbb R^d} Q(\omega)\psi(\omega)\,d\omega,
\qquad
\psi\in\mathscr S(\mathbb R^d),
\]
is well-defined and belongs to \(\mathscr S'(\mathbb R^d)\).
\end{lemma}

\begin{proof}
Fix \(\psi\in\mathscr S(\mathbb R^d)\). By definition of \(p_{N,0}\),
\[
|\psi(\omega)|
\le
p_{N,0}(\psi)(1+\|\omega\|_2)^{-N},
\qquad
\omega\in\mathbb R^d .
\]
Therefore
\[
|Q(\omega)\psi(\omega)|
\le
p_{N,0}(\psi)
\frac{|Q(\omega)|}{(1+\|\omega\|_2)^N}.
\]
The right-hand side is integrable by assumption, so
\(\mathcal I_Q[\psi]\) is well-defined. Moreover,
\[
|\mathcal I_Q[\psi]|
\le
A_N p_{N,0}(\psi).
\]
This seminorm bound proves that \(\mathcal I_Q\) is continuous on
\(\mathscr S(\mathbb R^d)\). Hence
\(\mathcal I_Q\in\mathscr S'(\mathbb R^d)\).
\end{proof}

\subsection{Proofs omitted from the main text}

This appendix collects proofs that are used in the main text.

\subsubsection{Proof of Proposition~\ref{P: convolution}}

\begin{proof}
By independence,
\[
P_Y = P_X * P_V .
\]
Let \(A\subseteq\mathbb R^d\) be Borel. Since \(P_V(dv)=f_V(v)\,dv\), Tonelli's theorem gives
\[
P_Y(A)
=
\iint \mathbf 1_A(x+v) f_V(v)\,dv\,dP_X(x).
\]
For each fixed \(x\), the change of variables \(y=x+v\) gives
\[
\int \mathbf 1_A(x+v) f_V(v)\,dv
=
\int_A f_V(y-x)\,dy .
\]
Hence
\[
P_Y(A)
=
\int_A
\left\{
\int_{\mathbb R^d} f_V(y-x)\,dP_X(x)
\right\}
dy .
\]
Therefore \(P_Y\) is absolutely continuous with density
\[
f_Y(y)
=
\int_{\mathbb R^d} f_V(y-x)\,dP_X(x).
\]
\end{proof}

\subsubsection{Proof of Proposition~\ref{P: bijection}}

\begin{proof}
Surjectivity follows from the definition of
\(\mathcal A_V(\mathbb R^d)\):
\[
\mathcal A_V(\mathbb R^d)
=
\mathcal K_V\bigl(\mathcal M(\mathbb R^d)\bigr).
\]

It remains to prove injectivity. Since \(\mathcal K_V\) is linear, it is enough to show that its kernel is trivial. Let
\(\mu\in\mathcal M(\mathbb R^d)\) satisfy
\[
\mathcal K_V[\mu]
=
f_V * \mu
=
0
\qquad
\text{in } L^1(\mathbb R^d;\mathbb C).
\]
Let
\[
\varphi_\mu(\omega)
\coloneq
\int_{\mathbb R^d} e^{i\omega^\top x}\,d\mu(x),
\qquad
\omega\in\mathbb R^d,
\]
be the Fourier--Stieltjes transform of \(\mu\). Since \(\mu\) is finite,
\(\varphi_\mu\) is bounded and continuous, with
\[
|\varphi_\mu(\omega)|
\le
\|\mu\|_{TV},
\qquad
\omega\in\mathbb R^d .
\]

Taking Fourier transforms yields
\[
\widetilde{\mathcal K_V[\mu]}(\omega)
=
\widetilde{(f_V * \mu)}(\omega)
=
\widetilde f_V(\omega)\varphi_\mu(\omega)
=
\varphi_V(\omega)\varphi_\mu(\omega),
\qquad
\omega\in\mathbb R^d .
\]
Here \(\widetilde f_V=\varphi_V\) because \(f_V\) is the density of \(P_V\). Since \(\mathcal K_V[\mu]=0\) in \(L^1(\mathbb R^d;\mathbb C)\), its Fourier transform vanishes identically. Thus
\[
\varphi_V(\omega)\varphi_\mu(\omega)
=
0,
\qquad
\omega\in\mathbb R^d .
\]

Define
\[
D
\coloneq
\{\omega\in\mathbb R^d:\varphi_V(\omega)\neq 0\}.
\]
By Assumption~\ref{A: Identification}, \(D^c\) has Lebesgue measure zero. Hence \(D\) is dense in \(\mathbb R^d\), because no nonempty open ball can be contained in \(D^c\). On \(D\),
\[
\varphi_\mu(\omega)=0.
\]
The continuity of \(\varphi_\mu\) extends this equality to all of
\(\mathbb R^d\):
\[
\varphi_\mu(\omega)=0,
\qquad
\omega\in\mathbb R^d .
\]
By uniqueness of Fourier--Stieltjes transforms for finite complex measures,
\(\mu=0\). Therefore \(\mathcal K_V\) is injective.
\end{proof}

\subsubsection{Proof of Theorem~\ref{T: Representation}}

\begin{proof}
Define the linear functional
\[
I_\lambda:\mathcal M(\mathbb R^d)\to\mathbb C,
\qquad
I_\lambda[\mu]
\coloneq
\int_{\mathbb R^d}\lambda(x)\,d\mu(x).
\]
Since \(\lambda\in C_0(\mathbb R^d;\mathbb C)\), it is bounded. Hence
\[
|I_\lambda[\mu]|
\le
\|\lambda\|_\infty \|\mu\|_{TV},
\qquad
\mu\in\mathcal M(\mathbb R^d),
\]
so \(I_\lambda\) is continuous with respect to the total variation norm. By Proposition~\ref{P: bijection}, the map
\[
\mathcal K_V:
\bigl(\mathcal M(\mathbb R^d),\|\cdot\|_{TV}\bigr)
\to
\bigl(\mathcal A_V(\mathbb R^d),\|\cdot\|_{\mathcal A_V}\bigr)
\]
is an isometric isomorphism. We may therefore define
\[
\mathcal T_{\lambda,V}
\coloneq
I_\lambda\circ \mathcal K_V^{-1}.
\]
Equivalently, if \(f\in\mathcal A_V(\mathbb R^d)\) and
\(\mu_f=\mathcal K_V^{-1}[f]\), then
\[
\mathcal T_{\lambda,V}[f]
=
\int_{\mathbb R^d}\lambda(x)\,d\mu_f(x).
\]
This definition is unambiguous because the representing measure \(\mu_f\) is unique.

The same bound proves continuity. Indeed, for
\(f\in\mathcal A_V(\mathbb R^d)\), let
\(\mu_f=\mathcal K_V^{-1}[f]\). Then
\[
|\mathcal T_{\lambda,V}[f]|
=
|I_\lambda[\mu_f]|
\le
\|\lambda\|_\infty \|\mu_f\|_{TV}
=
\|\lambda\|_\infty \|f\|_{\mathcal A_V}.
\]
Thus
\[
\mathcal T_{\lambda,V}\in\mathcal A'_V(\mathbb R^d).
\]

Now let \(\mu\in\mathcal M(\mathbb R^d)\). Since
\(\mathcal K_V^{-1}[\mathcal K_V[\mu]]=\mu\),
\[
\mathcal T_{\lambda,V}\!\left[\mathcal K_V[\mu]\right]
=
I_\lambda[\mu]
=
\int_{\mathbb R^d}\lambda(x)\,d\mu(x).
\]
This proves the representation formula. Under Assumptions~\ref{A: independence} and~\ref{A: Density},
Proposition~\ref{P: convolution} gives
\[
f_Y
=
f_V * P_X
=
\mathcal K_V[P_X].
\]
Applying the representation formula with \(\mu=P_X\) yields
\[
\mathcal T_{\lambda,V}[f_Y]
=
\mathcal T_{\lambda,V}\!\left[\mathcal K_V[P_X]\right]
=
\int_{\mathbb R^d}\lambda(x)\,dP_X(x)
=
\Lambda_\lambda(P_X).
\]

It remains to prove uniqueness. Let
\(S\in\mathcal A'_V(\mathbb R^d)\) satisfy
\[
S\!\left[\mathcal K_V[\mu]\right]
=
\int_{\mathbb R^d}\lambda(x)\,d\mu(x),
\qquad
\mu\in\mathcal M(\mathbb R^d).
\]
Since \(\mathcal K_V\) is onto \(\mathcal A_V(\mathbb R^d)\), every
\(f\in\mathcal A_V(\mathbb R^d)\) can be written as
\(f=\mathcal K_V[\mu]\) for some \(\mu\in\mathcal M(\mathbb R^d)\). Therefore
\[
S[f]
=
S\!\left[\mathcal K_V[\mu]\right]
=
\int_{\mathbb R^d}\lambda(x)\,d\mu(x)
=
\mathcal T_{\lambda,V}\!\left[\mathcal K_V[\mu]\right]
=
\mathcal T_{\lambda,V}[f].
\]
Hence \(S=\mathcal T_{\lambda,V}\).
\end{proof}

\subsubsection{Proof of Lemma~\ref{L: SmoothModels}}

\begin{proof}
We prove the three assertions separately. 

\smallskip
\noindent\emph{Part (a): completeness of \(\Xi^k\).}
Let \((\psi_n)_{n\ge1}\) be a Cauchy sequence in \(\Xi^k\). Since
\[
\|\psi\|_\infty\le \|\psi\|_{\Xi^k},
\qquad
\|\psi\|_1\le \|\psi\|_{\Xi^k},
\]
the sequence \((\psi_n)\) is Cauchy both in
\((C_0(\mathbb R^d;\mathbb C),\|\cdot\|_\infty)\) and in
\((L^1(\mathbb R^d;\mathbb C),\|\cdot\|_1)\). These spaces are Banach, so there exist
\[
g_0\in C_0(\mathbb R^d;\mathbb C),
\qquad
h\in L^1(\mathbb R^d;\mathbb C),
\]
such that
\[
\|\psi_n-g_0\|_\infty\to0,
\qquad
\|\psi_n-h\|_1\to0.
\]
Uniform convergence gives \(\psi_n(x)\to g_0(x)\) for every \(x\). By
\(L^1\)-convergence, a subsequence converges to \(h\) almost everywhere
\cite{folland1999real}, Theorem~2.30. Hence \(g_0=h\) almost everywhere. We
therefore identify \(h\) with the continuous representative \(g_0\), and obtain
\[
g_0\in L^1(\mathbb R^d;\mathbb C),
\qquad
\|\psi_n-g_0\|_1\to0.
\]

For each multi-index \(\alpha\) with \(|\alpha|\le k\), the sequence
\((\partial^\alpha\psi_n)\) is Cauchy in
\(C_0(\mathbb R^d;\mathbb C)\), because
\[
\|\partial^\alpha\psi_n-\partial^\alpha\psi_m\|_\infty
\le
\|\psi_n-\psi_m\|_{\Xi^k}.
\]
Let \(g_\alpha\in C_0(\mathbb R^d;\mathbb C)\) denote its uniform limit:
\[
\|\partial^\alpha\psi_n-g_\alpha\|_\infty\to0.
\]

It remains to identify \(g_\alpha\) with \(\partial^\alpha g_0\). If \(k=0\),
there is nothing to prove. Suppose \(k\ge1\). Fix a multi-index
\(\alpha\) with \(|\alpha|\le k-1\), a coordinate \(j\in\{1,\ldots,d\}\),
a point \(x\in\mathbb R^d\), and \(t\in\mathbb R\). Here \(e_j\) denotes the
\(j\)-th coordinate vector. For every \(n\), the fundamental theorem of calculus
gives
\[
\partial^\alpha\psi_n(x+t e_j)-\partial^\alpha\psi_n(x)
=
\int_0^t
\partial^{\alpha+e_j}\psi_n(x+s e_j)\,ds .
\]
Letting \(n\to\infty\) and using uniform convergence of the derivatives yields
\[
g_\alpha(x+t e_j)-g_\alpha(x)
=
\int_0^t
g_{\alpha+e_j}(x+s e_j)\,ds .
\]
Since \(g_{\alpha+e_j}\) is continuous, this identity implies
\[
\partial_j g_\alpha(x)=g_{\alpha+e_j}(x).
\]
Applying this argument for all \(|\alpha|\le k-1\) and all \(j\), induction on
\(|\alpha|\) gives
\[
g_0\in C^k(\mathbb R^d;\mathbb C),
\qquad
\partial^\alpha g_0=g_\alpha
\quad\text{for all }|\alpha|\le k.
\]
Since \(g_\alpha\in C_0(\mathbb R^d;\mathbb C)\) for every \(|\alpha|\le k\) and
\(g_0\in L^1(\mathbb R^d;\mathbb C)\), we have \(g_0\in\Xi^k\). Finally,
\[
\|\psi_n-g_0\|_{\Xi^k}
=
\max_{|\alpha|\le k}
\|\partial^\alpha\psi_n-g_\alpha\|_\infty
+
\|\psi_n-g_0\|_1
\to0.
\]
Thus \(\Xi^k\) is complete.

\smallskip
\noindent\emph{Part (b): continuous inclusion of \(\mathcal A_V\) into \(\Xi^k\).}
Let \(f\in\mathcal A_V(\mathbb R^d)\). By definition of
\(\mathcal A_V(\mathbb R^d)\), and using the unique representing measure, write
\[
f=f_V*\mu,
\qquad
\mu\in\mathcal M(\mathbb R^d).
\]

First, \(f\in L^1(\mathbb R^d;\mathbb C)\). Indeed, by Tonelli's theorem applied
to the total variation measure \(|\mu|\),
\[
\begin{aligned}
\|f\|_1
&=
\int_{\mathbb R^d}
\left|
\int_{\mathbb R^d} f_V(x-z)\,d\mu(z)
\right|dx  \\
&\le
\int_{\mathbb R^d}
\int_{\mathbb R^d}
|f_V(x-z)|\,d|\mu|(z)\,dx  \\
&=
\int_{\mathbb R^d}
\left(
\int_{\mathbb R^d}|f_V(x-z)|\,dx
\right)d|\mu|(z)  \\
&=
\|f_V\|_1\|\mu\|_{TV}.
\end{aligned}
\]

Next fix \(\alpha\) with \(|\alpha|\le k\), and define
\[
g_\alpha(x)
\coloneq
\int_{\mathbb R^d}
\partial^\alpha f_V(x-z)\,d\mu(z).
\]
This integral is well-defined because
\(\partial^\alpha f_V\in C_0(\mathbb R^d;\mathbb C)\subset
L^\infty(\mathbb R^d;\mathbb C)\), and
\[
|g_\alpha(x)|
\le
\|\partial^\alpha f_V\|_\infty\|\mu\|_{TV}.
\]

We first show that \(g_\alpha\in C_0(\mathbb R^d;\mathbb C)\). If \(x_n\to x\),
then
\[
\partial^\alpha f_V(x_n-z)\to \partial^\alpha f_V(x-z)
\qquad
\text{for every }z\in\mathbb R^d,
\]
and
\[
|\partial^\alpha f_V(x_n-z)|
\le
\|\partial^\alpha f_V\|_\infty .
\]
The bound is \(|\mu|\)-integrable because \(|\mu|(\mathbb R^d)<\infty\). Dominated
convergence gives \(g_\alpha(x_n)\to g_\alpha(x)\), so \(g_\alpha\) is continuous.

To prove that \(g_\alpha\) vanishes at infinity, set
\[
M_\alpha\coloneq \|\partial^\alpha f_V\|_\infty .
\]
Fix \(\varepsilon>0\). Since \(|\mu|\) is finite, choose a compact set
\(K\subset\mathbb R^d\) such that
\[
|\mu|(K^c)
<
\frac{\varepsilon}{2(1+M_\alpha)}.
\]
Since \(\partial^\alpha f_V\in C_0(\mathbb R^d;\mathbb C)\), choose \(R>0\) such that
\[
\|y\|_2>R
\quad\Longrightarrow\quad
|\partial^\alpha f_V(y)|
<
\frac{\varepsilon}{2(1+|\mu|(K))}.
\]
Because \(K\) is compact, there exists \(M>0\) such that
\[
\|x\|_2>M,\ z\in K
\quad\Longrightarrow\quad
\|x-z\|_2>R.
\]
For \(\|x\|_2>M\),
\[
\begin{aligned}
|g_\alpha(x)|
&\le
\int_K |\partial^\alpha f_V(x-z)|\,d|\mu|(z)
+
\int_{K^c}|\partial^\alpha f_V(x-z)|\,d|\mu|(z) \\
&\le
\frac{\varepsilon}{2(1+|\mu|(K))}\,|\mu|(K)
+
M_\alpha|\mu|(K^c) \\
&<
\varepsilon .
\end{aligned}
\]
Thus \(g_\alpha\in C_0(\mathbb R^d;\mathbb C)\).

It remains to identify \(g_\alpha\) as \(\partial^\alpha f\). We argue by
induction on \(|\alpha|\). For \(\alpha=0\), the identity \(g_0=f\) is the
definition of \(f\). Suppose that, for some \(\alpha\) with \(|\alpha|\le k-1\),
\[
\partial^\alpha f(x)
=
\int_{\mathbb R^d}
\partial^\alpha f_V(x-z)\,d\mu(z).
\]
Fix \(j\in\{1,\ldots,d\}\). For \(t\neq0\),
\[
\frac{\partial^\alpha f(x+t e_j)-\partial^\alpha f(x)}{t}
=
\int_{\mathbb R^d}
\frac{
\partial^\alpha f_V(x+t e_j-z)-\partial^\alpha f_V(x-z)
}{t}
\,d\mu(z).
\]
For each \(z\), the integrand converges to
\(\partial^{\alpha+e_j}f_V(x-z)\) as \(t\to0\). Moreover, the fundamental theorem
of calculus gives
\[
\frac{
\partial^\alpha f_V(x+t e_j-z)-\partial^\alpha f_V(x-z)
}{t}
=
\int_0^1
\partial^{\alpha+e_j}f_V(x+s t e_j-z)\,ds,
\]
and therefore
\[
\left|
\frac{
\partial^\alpha f_V(x+t e_j-z)-\partial^\alpha f_V(x-z)
}{t}
\right|
\le
\|\partial^{\alpha+e_j}f_V\|_\infty .
\]
The right-hand side is \(|\mu|\)-integrable. Dominated convergence yields
\[
\partial^{\alpha+e_j}f(x)
=
\int_{\mathbb R^d}
\partial^{\alpha+e_j}f_V(x-z)\,d\mu(z)
=
g_{\alpha+e_j}(x).
\]
By induction,
\[
\partial^\alpha f=g_\alpha\in C_0(\mathbb R^d;\mathbb C)
\qquad
\text{for all }|\alpha|\le k.
\]
Together with the \(L^1\) bound above, this proves
\[
f\in\Xi^k(\mathbb R^d;\mathbb C).
\]

The same estimates give the continuity of the inclusion. If
\(\mu=\mathcal K_V^{-1}[f]\), then
\[
\|f\|_1
\le
\|f_V\|_1\|\mu\|_{TV},
\]
and, for every \(|\alpha|\le k\),
\[
\|\partial^\alpha f\|_\infty
\le
\|\partial^\alpha f_V\|_\infty\|\mu\|_{TV}.
\]
Hence
\[
\begin{aligned}
\|f\|_{\Xi^k}
&=
\max_{|\alpha|\le k}\|\partial^\alpha f\|_\infty
+
\|f\|_1 \\
&\le
\left(
\max_{|\alpha|\le k}\|\partial^\alpha f_V\|_\infty
+
\|f_V\|_1
\right)\|\mu\|_{TV}.
\end{aligned}
\]
Since \(\|\mu\|_{TV}=\|f\|_{\mathcal A_V}\), we obtain
\[
\|f\|_{\Xi^k}
\le
C_{V,k}\|f\|_{\mathcal A_V},
\qquad
C_{V,k}
\coloneq
\max_{|\alpha|\le k}\|\partial^\alpha f_V\|_\infty
+
\|f_V\|_1 .
\]
Thus the embedding
\[
\mathcal A_V(\mathbb R^d)\hookrightarrow \Xi^k(\mathbb R^d;\mathbb C)
\]
is continuous.

\smallskip
\noindent\emph{Part (c): continuous and dense inclusion of the Schwartz space.}
We first prove continuity of the inclusion
\[
\mathscr S(\mathbb R^d)\hookrightarrow \Xi^k(\mathbb R^d;\mathbb C).
\]
For \(\varphi\in\mathscr S(\mathbb R^d)\), define the Schwartz seminorm
\[
Q_{d,k}(\varphi)
\coloneq
\max_{|\alpha|\le k}
\sup_{x\in\mathbb R^d}
(1+\|x\|_2)^{d+1}|\partial^\alpha\varphi(x)|.
\]
Then
\[
\max_{|\alpha|\le k}\|\partial^\alpha\varphi\|_\infty
\le
Q_{d,k}(\varphi),
\]
and
\[
\begin{aligned}
\|\varphi\|_1
&=
\int_{\mathbb R^d}
(1+\|x\|_2)^{-d-1}
(1+\|x\|_2)^{d+1}
|\varphi(x)|\,dx \\
&\le
\left(
\int_{\mathbb R^d}
(1+\|x\|_2)^{-d-1}\,dx
\right)
Q_{d,k}(\varphi).
\end{aligned}
\]
The integral is finite. Hence
\[
\|\varphi\|_{\Xi^k}
\le
C_{d,k}Q_{d,k}(\varphi),
\]
which proves continuity of the inclusion.

It remains to prove density. Let
\(\psi\in\Xi^k(\mathbb R^d;\mathbb C)\). Choose
\(\chi\in C_c^\infty(\mathbb R^d)\) such that \(0\le\chi\le1\),
\(\chi\equiv1\) on \(B(0,1)\), and \(\chi\equiv0\) outside \(B(0,2)\). For
\(R\ge1\), define
\[
\chi_R(x)\coloneq\chi(x/R),
\qquad
\psi_R(x)\coloneq\chi_R(x)\psi(x).
\]
Then \(\psi_R\in C_c^k(\mathbb R^d;\mathbb C)\).

We first show that \(\psi_R\to\psi\) in \(\Xi^k\). Since
\(\chi_R(x)\to1\) for every \(x\) and
\[
|(\chi_R(x)-1)\psi(x)|\le |\psi(x)|,
\]
dominated convergence gives
\[
\|\psi_R-\psi\|_1\to0.
\]
Now fix \(\alpha\) with \(|\alpha|\le k\). By Leibniz's rule,
\[
\partial^\alpha(\psi_R-\psi)
=
(\chi_R-1)\partial^\alpha\psi
+
\sum_{0<\beta\le\alpha}
{\alpha\choose\beta}
(\partial^\beta\chi_R)\partial^{\alpha-\beta}\psi .
\]
Because \(\chi_R=1\) on \(B(0,R)\) and
\(\partial^\alpha\psi\in C_0(\mathbb R^d;\mathbb C)\),
\[
\|(\chi_R-1)\partial^\alpha\psi\|_\infty\to0.
\]
For \(0<\beta\le\alpha\),
\[
\partial^\beta\chi_R(x)
=
R^{-|\beta|}(\partial^\beta\chi)(x/R),
\]
and
\[
\operatorname{supp}(\partial^\beta\chi_R)
\subset
\{x\in\mathbb R^d:R\le \|x\|_2\le 2R\}.
\]
Thus
\[
\|(\partial^\beta\chi_R)\partial^{\alpha-\beta}\psi\|_\infty
\le
R^{-|\beta|}
\|\partial^\beta\chi\|_\infty
\sup_{\|x\|_2\ge R}
|\partial^{\alpha-\beta}\psi(x)|.
\]
Since \(\partial^{\alpha-\beta}\psi\in C_0(\mathbb R^d;\mathbb C)\), the right-hand
side tends to zero. Therefore
\[
\max_{|\alpha|\le k}
\|\partial^\alpha(\psi_R-\psi)\|_\infty
\to0,
\]
and hence
\[
\|\psi_R-\psi\|_{\Xi^k}\to0.
\]

Now choose a standard compactly supported mollifier
\(\rho\in C_c^\infty(\mathbb R^d)\) with \(\int\rho=1\), and set
\[
\rho_\varepsilon(x)
\coloneq
\varepsilon^{-d}\rho(x/\varepsilon),
\qquad
\psi_{R,\varepsilon}
\coloneq
\rho_\varepsilon*\psi_R .
\]
For fixed \(R\), we have
\[
\psi_{R,\varepsilon}\in C_c^\infty(\mathbb R^d)\subset\mathscr S(\mathbb R^d).
\]
Standard properties of mollifiers give, for every \(|\alpha|\le k\),
\[
\partial^\alpha\psi_{R,\varepsilon}
=
\rho_\varepsilon*\partial^\alpha\psi_R,
\]
and
\[
\|\psi_{R,\varepsilon}-\psi_R\|_1\to0,
\qquad
\|\partial^\alpha\psi_{R,\varepsilon}
      -\partial^\alpha\psi_R\|_\infty\to0
\quad (|\alpha|\le k),
\]
as \(\varepsilon\downarrow0\); see
\cite{folland1999real}, Propositions~8.10 and~8.14. Hence
\[
\|\psi_{R,\varepsilon}-\psi_R\|_{\Xi^k}\to0
\qquad
(\varepsilon\downarrow0).
\]

For each \(n\), choose \(R_n\) such that
\[
\|\psi_{R_n}-\psi\|_{\Xi^k}<\frac1{2n},
\]
and then choose \(\varepsilon_n>0\) such that
\[
\|\psi_{R_n,\varepsilon_n}-\psi_{R_n}\|_{\Xi^k}<\frac1{2n}.
\]
The functions
\[
\varphi_n\coloneq\psi_{R_n,\varepsilon_n}
\]
belong to \(\mathscr S(\mathbb R^d)\), and
\[
\|\varphi_n-\psi\|_{\Xi^k}\le \frac1n.
\]
Thus \(\mathscr S(\mathbb R^d)\) is dense in
\(\Xi^k(\mathbb R^d;\mathbb C)\).
\end{proof}
\subsubsection{Proof of Theorem~\ref{T:master-fourier-identification}}
\label{A:proof-master-fourier-identification}

The proof uses one auxiliary fact about translations on \(\Xi^k\) discussed in the next Lemma:

\begin{lemma}[Translations on \(\Xi^k\) and Bochner representation of convolution]
\label{L:translation-xik}
For \(x\in\mathbb R^d\), define
\[
(\tau_x f)(t)
\coloneq
f(t-x),
\qquad
t\in\mathbb R^d.
\]
Then the following statements hold.

\begin{enumerate}
\item[(a)]
The family \((\tau_x)_{x\in\mathbb R^d}\) is a strongly continuous group of
linear isometries on \(\Xi^k(\mathbb R^d;\mathbb C)\). In particular, for each
\(f\in\Xi^k(\mathbb R^d;\mathbb C)\), the orbit map
\[
x\longmapsto \tau_x f
\]
is bounded and continuous from \(\mathbb R^d\) into
\(\Xi^k(\mathbb R^d;\mathbb C)\).

\item[(b)]
Under Assumption~\ref{A:SmoothBoundedNoise}, for every
\(\mu\in\mathcal M(\mathbb R^d)\), the map
\(x\mapsto \tau_x f_V\) is Bochner integrable with respect to \(\mu\), and
\[
f_V*\mu
=
\int_{\mathbb R^d}\tau_x f_V\,d\mu(x)
\qquad
\text{in } \Xi^k(\mathbb R^d;\mathbb C).
\]
\end{enumerate}
\end{lemma}

\begin{proof}
We first prove part~(a). Each \(\tau_x\) is linear. Moreover, for every
multi-index \(\alpha\) with \(|\alpha|\le k\),
\[
\partial^\alpha(\tau_x\psi)
=
\tau_x(\partial^\alpha\psi).
\]
Thus translation preserves both the derivative sup norms and the \(L^1\)-norm:
\[
\|\partial^\alpha(\tau_x\psi)\|_\infty
=
\|\partial^\alpha\psi\|_\infty,
\qquad
\|\tau_x\psi\|_1
=
\|\psi\|_1.
\]
Therefore
\[
\|\tau_x\psi\|_{\Xi^k}
=
\|\psi\|_{\Xi^k},
\qquad
\psi\in\Xi^k(\mathbb R^d;\mathbb C).
\]
Each \(\tau_x\) is an isometry. Since
\[
\tau_0=\mathrm{Id},
\qquad
\tau_x\tau_y=\tau_{x+y},
\]
the family \((\tau_x)_{x\in\mathbb R^d}\) is a group.

It remains to prove strong continuity. Fix
\(f\in\Xi^k(\mathbb R^d;\mathbb C)\). At the origin,
\[
\|\tau_x f-f\|_{\Xi^k}
=
\max_{|\alpha|\le k}
\|\tau_x(\partial^\alpha f)-\partial^\alpha f\|_\infty
+
\|\tau_x f-f\|_1.
\]
The \(L^1\)-term tends to zero as \(x\to0\) by continuity of translations in
\(L^1\); see \cite{folland1999real}, Proposition~8.5. For each
\(|\alpha|\le k\), \(\partial^\alpha f\in C_0(\mathbb R^d;\mathbb C)\), hence
\(\partial^\alpha f\) is uniformly continuous. Therefore
\[
\|\tau_x(\partial^\alpha f)-\partial^\alpha f\|_\infty
\to0
\qquad
(x\to0).
\]
Thus
\[
\|\tau_x f-f\|_{\Xi^k}\to0
\qquad
(x\to0).
\]
For any \(x_0\in\mathbb R^d\),
\[
\|\tau_x f-\tau_{x_0}f\|_{\Xi^k}
=
\|\tau_{x-x_0}f-f\|_{\Xi^k}
\to0
\qquad
(x\to x_0),
\]
because \(\tau_{x_0}\) is an isometry. This proves part~(a).

We now prove part~(b). Assumption~\ref{A:SmoothBoundedNoise} implies that
\(\partial^\alpha f_V\in C_0(\mathbb R^d;\mathbb C)\) for all
\(|\alpha|\le k\). Since \(f_V\) is a density, \(f_V\in L^1(\mathbb R^d)\).
Hence
\[
f_V\in\Xi^k(\mathbb R^d;\mathbb C).
\]
Define
\[
\Phi:\mathbb R^d\to\Xi^k(\mathbb R^d;\mathbb C),
\qquad
\Phi(x)\coloneq\tau_x f_V.
\]
By part~(a), \(\Phi\) is bounded and continuous. Since \(\mathbb R^d\) is
separable, \(\Phi\) has separable range and is strongly measurable with respect
to \(|\mu|\). Also,
\[
\int_{\mathbb R^d}\|\Phi(x)\|_{\Xi^k}\,d|\mu|(x)
=
\|f_V\|_{\Xi^k}\,|\mu|(\mathbb R^d)
<\infty.
\]
Thus \(\Phi\) is Bochner integrable with respect to \(|\mu|\).

Since \(\mu\) is a finite complex measure, there exists a measurable function
\(\sigma_\mu\), with \(|\sigma_\mu|=1\) \(|\mu|\)-a.e., such that
\[
d\mu=\sigma_\mu\,d|\mu|;
\]
see \cite{rudin1987real}, Theorem~6.12. We define
\[
\int_{\mathbb R^d}\Phi(x)\,d\mu(x)
\coloneq
\int_{\mathbb R^d}\sigma_\mu(x)\Phi(x)\,d|\mu|(x).
\]
Set
\[
F
\coloneq
\int_{\mathbb R^d}\tau_x f_V\,d\mu(x)
\in
\Xi^k(\mathbb R^d;\mathbb C).
\]

We identify \(F\). For \(|\alpha|\le k\), the map
\[
J_\alpha:\Xi^k(\mathbb R^d;\mathbb C)\to C_0(\mathbb R^d;\mathbb C),
\qquad
J_\alpha(h)\coloneq \partial^\alpha h,
\]
is bounded and linear. Bounded linear maps commute with Bochner integration;
see \cite{Hsing2015}, Theorem~3.1.7. Hence
\[
\partial^\alpha F
=
\int_{\mathbb R^d}\tau_x(\partial^\alpha f_V)\,d\mu(x)
\qquad
\text{in } C_0(\mathbb R^d;\mathbb C).
\]
Evaluating at \(t\in\mathbb R^d\) gives
\[
\partial^\alpha F(t)
=
\int_{\mathbb R^d}
\partial^\alpha f_V(t-x)\,d\mu(x).
\]
Taking \(\alpha=0\), we obtain
\[
F(t)
=
\int_{\mathbb R^d}f_V(t-x)\,d\mu(x)
=
(f_V*\mu)(t).
\]
Therefore
\[
f_V*\mu
=
\int_{\mathbb R^d}\tau_x f_V\,d\mu(x)
\qquad
\text{in } \Xi^k(\mathbb R^d;\mathbb C).
\]
This proves part~(b).
\end{proof}

\begin{proof}[Proof of Theorem~\ref{T:master-fourier-identification}]
Set
\[
\lambda\coloneq \lambda_{g,V,y},
\qquad
Q\coloneq \mathcal Q_{g,V,y},
\qquad
L_Q\coloneq \mathcal F^{-1}\{\mathcal I_Q\}.
\]
By assumption~\textup{(i)} and
Lemma~\ref{L:weighted-integrability-temperedness},
\[
\mathcal I_Q\in\mathscr S'(\mathbb R^d).
\]
Hence
\[
L_Q\in\mathscr S'(\mathbb R^d).
\]
Assumption~\textup{(ii)} gives
\begin{equation}
\label{eq:master-fourier-bound}
|L_Q[\psi]|
\le
C\|\psi\|_{\Xi^k},
\qquad
\psi\in\mathscr S(\mathbb R^d).
\end{equation}
Thus \(L_Q\) is continuous on \(\mathscr S(\mathbb R^d)\) when
\(\mathscr S(\mathbb R^d)\) is equipped with the \(\Xi^k\)-norm.

By Lemma~\ref{L: SmoothModels}(c), \(\mathscr S(\mathbb R^d)\) is dense in
\(\Xi^k(\mathbb R^d;\mathbb C)\). Therefore \eqref{eq:master-fourier-bound}
extends \(L_Q\) uniquely to a continuous linear functional
\[
\overline L_Q\in
\bigl(\Xi^k(\mathbb R^d;\mathbb C)\bigr)'.
\]
By Lemma~\ref{L: SmoothModels}(b),
\[
\mathcal A_V(\mathbb R^d)
\subset
\Xi^k(\mathbb R^d;\mathbb C),
\]
so the restriction
\[
\overline L_Q|_{\mathcal A_V(\mathbb R^d)}
\]
is well-defined. We prove that this restriction is
\(\mathcal T_{g,V,y}\).

For \(\psi\in\mathscr S(\mathbb R^d)\), the definition of the inverse Fourier
transform on tempered distributions gives
\begin{equation}
\label{eq:master-fourier-action}
\begin{aligned}
L_Q[\psi]
&=
\mathcal I_Q[\psi^\sharp]  \\
&=
\int_{\mathbb R^d}Q(\omega)\psi^\sharp(\omega)\,d\omega  \\
&=
\frac{1}{(2\pi)^d}
\int_{\mathbb R^d}Q(\omega)\widetilde\psi(-\omega)\,d\omega,
\end{aligned}
\end{equation}
where
\[
\psi^\sharp(\omega)
=
(2\pi)^{-d}\widetilde\psi(-\omega).
\]

Define
\[
\Psi(x)
\coloneq
\overline L_Q[\tau_x f_V],
\qquad
x\in\mathbb R^d.
\]
By Lemma~\ref{L:translation-xik}(a), the map
\(x\mapsto\tau_x f_V\) is bounded and continuous into
\(\Xi^k(\mathbb R^d;\mathbb C)\). Since \(\overline L_Q\) is continuous,
\(\Psi\) is bounded and continuous.

We next identify \(\Psi\). Let \(\eta\in\mathscr S(\mathbb R^d)\), and let
\(\mu_\eta\in\mathcal M(\mathbb R^d)\) be the finite complex measure
\[
d\mu_\eta(x)=\eta(x)\,dx.
\]
By Lemma~\ref{L:translation-xik}(b),
\[
f_V*\eta
=
\int_{\mathbb R^d}\eta(x)\tau_x f_V\,dx
\qquad
\text{in } \Xi^k(\mathbb R^d;\mathbb C).
\]
Applying \(\overline L_Q\) and using again
\cite{Hsing2015}, Theorem~3.1.7, gives
\begin{equation}
\label{eq:master-pairing-Psi}
\int_{\mathbb R^d}\Psi(x)\eta(x)\,dx
=
\overline L_Q[f_V*\eta].
\end{equation}

Choose \((u_n)_{n\ge1}\subset\mathscr S(\mathbb R^d)\) such that
\[
u_n\to f_V
\qquad
\text{in } \Xi^k(\mathbb R^d;\mathbb C),
\]
which is possible by Lemma~\ref{L: SmoothModels}(c). We use the elementary
bound
\begin{equation}
\label{eq:master-convolution-estimate}
\|\zeta*\eta\|_{\Xi^k}
\le
\|\zeta\|_{\Xi^k}\|\eta\|_1,
\qquad
\zeta\in\Xi^k(\mathbb R^d;\mathbb C),\quad
\eta\in L^1(\mathbb R^d;\mathbb C).
\end{equation}
Indeed, for every \(|\alpha|\le k\),
\[
\partial^\alpha(\zeta*\eta)
=
(\partial^\alpha\zeta)*\eta,
\]
and therefore
\[
\|\partial^\alpha(\zeta*\eta)\|_\infty
\le
\|\partial^\alpha\zeta\|_\infty\|\eta\|_1.
\]
Also,
\[
\|\zeta*\eta\|_1
\le
\|\zeta\|_1\|\eta\|_1.
\]
Combining these inequalities proves \eqref{eq:master-convolution-estimate}.

Applying \eqref{eq:master-convolution-estimate} to
\(\zeta=u_n-f_V\) gives
\[
u_n*\eta\to f_V*\eta
\qquad
\text{in } \Xi^k(\mathbb R^d;\mathbb C).
\]
Since \(\overline L_Q\) is continuous,
\begin{equation}
\label{eq:master-limit-extension}
\overline L_Q[f_V*\eta]
=
\lim_{n\to\infty}\overline L_Q[u_n*\eta].
\end{equation}
Moreover, \(u_n*\eta\in\mathscr S(\mathbb R^d)\), because the Schwartz space is
closed under convolution; see \cite{stein2003fourier}, Chapter~5,
Proposition~1.11. Hence \(\overline L_Q[u_n*\eta]=L_Q[u_n*\eta]\). By
\eqref{eq:master-fourier-action},
\begin{equation}
\label{eq:master-convolution-fourier}
L_Q[u_n*\eta]
=
\frac{1}{(2\pi)^d}
\int_{\mathbb R^d}
Q(\omega)\widetilde u_n(-\omega)\widetilde\eta(-\omega)\,d\omega.
\end{equation}

Since \(u_n\to f_V\) in \(L^1(\mathbb R^d;\mathbb C)\),
\[
\sup_{\omega\in\mathbb R^d}
|\widetilde u_n(\omega)-\varphi_V(\omega)|
\le
\|u_n-f_V\|_1
\to0.
\]
Thus \(\widetilde u_n\to\varphi_V\) uniformly on \(\mathbb R^d\). Also,
\[
\sup_{n\ge1}\|\widetilde u_n\|_\infty
\le
\sup_{n\ge1}\|u_n\|_1
<\infty.
\]

Let \(N\) be the integer from assumption~\textup{(i)}. Since
\(\widetilde\eta\in\mathscr S(\mathbb R^d)\), there exists
\(C_{\eta,N}>0\) such that
\[
|\widetilde\eta(-\omega)|
\le
C_{\eta,N}(1+\|\omega\|_2)^{-N}.
\]
Hence
\[
|Q(\omega)\widetilde\eta(-\omega)|
\le
C_{\eta,N}
\frac{|Q(\omega)|}{(1+\|\omega\|_2)^N},
\]
and the right-hand side is integrable by assumption~\textup{(i)}. Therefore
\[
|Q(\omega)\widetilde u_n(-\omega)\widetilde\eta(-\omega)|
\le
\left(\sup_{m\ge1}\|\widetilde u_m\|_\infty\right)
|Q(\omega)\widetilde\eta(-\omega)|.
\]
Dominated convergence in \eqref{eq:master-convolution-fourier} yields
\begin{equation}
\label{eq:master-fVeta-fourier}
\overline L_Q[f_V*\eta]
=
\frac{1}{(2\pi)^d}
\int_{\mathbb R^d}
Q(\omega)\varphi_V(-\omega)\widetilde\eta(-\omega)\,d\omega.
\end{equation}

Since \(|\varphi_V|\le1\), assumption~\textup{(i)} and
Lemma~\ref{L:weighted-integrability-temperedness} imply that
\[
\mathcal I_{\omega\mapsto Q(\omega)\varphi_V(-\omega)}
\in
\mathscr S'(\mathbb R^d).
\]
By \eqref{eq:master-fVeta-fourier} and the definition of the inverse Fourier
transform on \(\mathscr S'(\mathbb R^d)\),
\[
\overline L_Q[f_V*\eta]
=
\mathcal F^{-1}
\left\{
\mathcal I_{\omega\mapsto Q(\omega)\varphi_V(-\omega)}
\right\}[\eta].
\]
Assumption~\textup{(iii)} gives
\[
\mathcal F^{-1}
\left\{
\mathcal I_{\omega\mapsto Q(\omega)\varphi_V(-\omega)}
\right\}
=
\mathcal I_\lambda
\qquad
\text{in } \mathscr S'(\mathbb R^d).
\]
Therefore
\[
\overline L_Q[f_V*\eta]
=
\mathcal I_\lambda[\eta]
=
\int_{\mathbb R^d}\lambda(x)\eta(x)\,dx,
\qquad
\eta\in\mathscr S(\mathbb R^d).
\]

Combining this identity with \eqref{eq:master-pairing-Psi}, we obtain
\[
\int_{\mathbb R^d}\Psi(x)\eta(x)\,dx
=
\int_{\mathbb R^d}\lambda(x)\eta(x)\,dx,
\qquad
\eta\in\mathscr S(\mathbb R^d).
\]
Thus \(\Psi-\lambda=0\) in \(\mathscr S'(\mathbb R^d)\). Since \(\Psi\) is
continuous and \(\lambda\in C_0(\mathbb R^d;\mathbb C)\), both functions are
locally integrable. The distributional equality implies
\(\Psi=\lambda\) Lebesgue-a.e. Since both functions are continuous,
\[
\Psi(x)=\lambda(x),
\qquad
x\in\mathbb R^d.
\]

We now identify the restriction of \(\overline L_Q\) to
\(\mathcal A_V(\mathbb R^d)\). Let \(f\in\mathcal A_V(\mathbb R^d)\), and set
\[
\mu_f\coloneq \mathcal K_V^{-1}[f].
\]
Then
\[
f=f_V*\mu_f.
\]
By Lemma~\ref{L:translation-xik}(b),
\[
f
=
\int_{\mathbb R^d}\tau_x f_V\,d\mu_f(x)
\qquad
\text{in } \Xi^k(\mathbb R^d;\mathbb C).
\]
Applying \(\overline L_Q\) and using Bochner integration once more,
\[
\begin{aligned}
\overline L_Q[f]
&=
\int_{\mathbb R^d}
\overline L_Q[\tau_x f_V]\,d\mu_f(x)  \\
&=
\int_{\mathbb R^d}\Psi(x)\,d\mu_f(x)  \\
&=
\int_{\mathbb R^d}\lambda(x)\,d\mu_f(x)  \\
&=
\mathcal T_{g,V,y}[f].
\end{aligned}
\]
Hence
\[
\overline L_Q|_{\mathcal A_V(\mathbb R^d)}
=
\mathcal T_{g,V,y}.
\]

Finally, let \(f\in\mathcal A_V(\mathbb R^d)\), and let
\((f_n)_{n\ge1}\subset\mathscr S(\mathbb R^d)\) satisfy
\[
\|f_n-f\|_{\Xi^k}\to0.
\]
Since \(\overline L_Q\) extends \(L_Q\) and is continuous on \(\Xi^k\),
\[
L_Q[f_n]
=
\overline L_Q[f_n]
\to
\overline L_Q[f]
=
\mathcal T_{g,V,y}[f].
\]
Using \eqref{eq:master-fourier-action} with \(\psi=f_n\), this is the same as
\[
\frac{1}{(2\pi)^d}
\int_{\mathbb R^d}
Q(\omega)\widetilde f_n(-\omega)\,d\omega
\to
\mathcal T_{g,V,y}[f].
\]
The limit is therefore independent of the approximating sequence.
\end{proof}

\subsubsection{Proof of Proposition~\ref{P:tweedie-approximation}}

\begin{proof}
Write
\[
\lambda_n\coloneq \lambda_{n,g,V,y},
\qquad
\lambda\coloneq \lambda_{g,V,y},
\qquad
\mathcal T_n\coloneq \mathcal T_{n,g,V,y}.
\]
For each \(n\), \(\lambda_n\in C_0(\mathbb R^d;\mathbb C)\). Hence
Theorem~\ref{T: Representation} gives a unique continuous linear functional
\[
\mathcal T_n\in\mathcal A_V'(\mathbb R^d)
\]
such that
\[
\mathcal T_n[f_V*\mu]
=
\int_{\mathbb R^d}\lambda_n(x)\,d\mu(x),
\qquad
\mu\in\mathcal M(\mathbb R^d).
\]

We now construct the limiting functional. By assumption,
\[
\lambda_n(x)\to\lambda(x),
\qquad
x\in\mathbb R^d,
\]
and
\[
M_\lambda
\coloneq
\sup_{n\ge1}\|\lambda_n\|_\infty
<\infty .
\]
Since \(\lambda\) is the pointwise limit of continuous functions, it is Borel measurable. Moreover,
\[
\|\lambda\|_\infty\le M_\lambda .
\]

For \(f\in\mathcal A_V(\mathbb R^d)\), let
\[
\mu_f\coloneq \mathcal K_V^{-1}[f].
\]
The measure \(\mu_f\) is well-defined and unique by
Proposition~\ref{P: bijection}. Define
\[
\mathcal T_{g,V,y}[f]
\coloneq
\int_{\mathbb R^d}\lambda(x)\,d\mu_f(x).
\]
This map is linear because \(f\mapsto\mu_f\) is linear. It is also continuous:
\[
|\mathcal T_{g,V,y}[f]|
\le
\|\lambda\|_\infty\|\mu_f\|_{TV}
=
\|\lambda\|_\infty\|f\|_{\mathcal A_V}.
\]
Therefore
\[
\mathcal T_{g,V,y}\in\mathcal A_V'(\mathbb R^d).
\]

This functional is the unique continuous linear functional with the displayed integral representation. Indeed, if
\(S\in\mathcal A_V'(\mathbb R^d)\) satisfies
\[
S[f_V*\mu]
=
\int_{\mathbb R^d}\lambda(x)\,d\mu(x),
\qquad
\mu\in\mathcal M(\mathbb R^d),
\]
then, for any \(f\in\mathcal A_V(\mathbb R^d)\),
\[
S[f]
=
S[f_V*\mu_f]
=
\int_{\mathbb R^d}\lambda(x)\,d\mu_f(x)
=
\mathcal T_{g,V,y}[f].
\]
Hence \(S=\mathcal T_{g,V,y}\).

It remains to prove convergence. Fix \(f\in\mathcal A_V(\mathbb R^d)\), and write
\(\mu_f=\mathcal K_V^{-1}[f]\). Then
\[
\mathcal T_n[f]
=
\int_{\mathbb R^d}\lambda_n(x)\,d\mu_f(x),
\qquad
\mathcal T_{g,V,y}[f]
=
\int_{\mathbb R^d}\lambda(x)\,d\mu_f(x).
\]
Since
\[
\lambda_n(x)-\lambda(x)\to0
\qquad
\text{for every }x\in\mathbb R^d,
\]
and
\[
|\lambda_n(x)-\lambda(x)|
\le
2M_\lambda,
\qquad
x\in\mathbb R^d,
\]
the dominated convergence theorem, applied with respect to the finite measure
\(|\mu_f|\), gives
\[
\int_{\mathbb R^d}
|\lambda_n(x)-\lambda(x)|\,d|\mu_f|(x)
\to0.
\]
Consequently,
\[
\begin{aligned}
|\mathcal T_n[f]-\mathcal T_{g,V,y}[f]|
&=
\left|
\int_{\mathbb R^d}
\{\lambda_n(x)-\lambda(x)\}\,d\mu_f(x)
\right|  \\
&\le
\int_{\mathbb R^d}
|\lambda_n(x)-\lambda(x)|\,d|\mu_f|(x)
\to0 .
\end{aligned}
\]
Thus
\[
\mathcal T_n[f]\to\mathcal T_{g,V,y}[f],
\qquad
f\in\mathcal A_V(\mathbb R^d).
\]

Now suppose that \((Y,X,V)\) also satisfies
Assumption~\ref{A: independence}. By Proposition~\ref{P: convolution},
\[
f_Y
=
f_V*P_X
=
\mathcal K_V[P_X],
\]
so \(f_Y\in\mathcal A_V(\mathbb R^d)\). Applying the preceding convergence result
with \(f=f_Y\) yields
\[
\mathcal T_{g,V,y}[f_Y]
=
\lim_{n\to\infty}\mathcal T_{n,g,V,y}[f_Y].
\]

Finally, by the definition of \(\lambda_{g,V,y}\),
\[
\mathcal T_{g,V,y}[f_Y]
=
\int_{\mathbb R^d}\lambda_{g,V,y}(x)\,dP_X(x)
=
\int_{\mathbb R^d}g(x)f_V(y-x)\,dP_X(x).
\]
Therefore, whenever \(f_Y(y)>0\), the ratio formula \eqref{Eq: ratio} gives
\[
\mathbb E[g(X)\mid Y=y]
=
\frac{
\int_{\mathbb R^d}g(x)f_V(y-x)\,dP_X(x)
}{
f_Y(y)
}
=
\frac{\mathcal T_{g,V,y}[f_Y]}{f_Y(y)}
=
\lim_{n\to\infty}
\frac{\mathcal T_{n,g,V,y}[f_Y]}{f_Y(y)}.
\]
This proves the approximation formula.
\end{proof}

\subsubsection{Proof of Proposition~\ref{P:measure-factorization-functional}}

\begin{proof}
Set
\[
Q(\omega)\coloneq \mathcal Q_{g,V,y}(\omega),
\qquad
R(\omega)\coloneq
\sum_{|\alpha|\le k}(i\omega)^\alpha
\varphi_{\mu_\alpha}(\omega).
\]
By assumption, \(Q=R\) Lebesgue-a.e. Hence
\[
\mathcal I_Q=\mathcal I_R
\qquad
\text{on } \mathscr S(\mathbb R^d).
\]

\smallskip
\noindent\emph{Step 1: temperedness.}
For each \(|\alpha|\le k\), the Fourier--Stieltjes transform of
\(\mu_\alpha\) is bounded:
\[
|\varphi_{\mu_\alpha}(\omega)|
=
\left|
\int_{\mathbb R^d}e^{i\omega^\top z}\,d\mu_\alpha(z)
\right|
\le
\|\mu_\alpha\|_{TV},
\qquad
\omega\in\mathbb R^d.
\]
Therefore there is a finite constant \(C>0\) such that
\[
|R(\omega)|
\le
\sum_{|\alpha|\le k}
|\omega^\alpha|\|\mu_\alpha\|_{TV}
\le
C(1+\|\omega\|_2)^k .
\]
Since \(Q=R\) Lebesgue-a.e., the same polynomial bound holds for \(Q\) after
changing \(Q\) on a null set. Thus, for any integer \(N>k+d\),
\[
\int_{\mathbb R^d}
\frac{|Q(\omega)|}{(1+\|\omega\|_2)^N}\,d\omega
\le
C
\int_{\mathbb R^d}
(1+\|\omega\|_2)^{k-N}\,d\omega
<\infty.
\]
Lemma~\ref{L:weighted-integrability-temperedness} gives
\[
\mathcal I_Q\in\mathscr S'(\mathbb R^d).
\]

\smallskip
\noindent\emph{Step 2: inverse Fourier transform on the Schwartz class.}
Let \(\psi\in\mathscr S(\mathbb R^d)\). Since
\[
\psi^\sharp(\omega)
=
(2\pi)^{-d}\widetilde\psi(-\omega)
\]
belongs to \(\mathscr S(\mathbb R^d)\), we have
\[
\int_{\mathbb R^d}
|\omega^\alpha\psi^\sharp(\omega)|\,d\omega
<\infty
\qquad
(|\alpha|\le k).
\]
Because each \(\mu_\alpha\) is finite, Fubini's theorem applies:
\[
\int_{\mathbb R^d}
\int_{\mathbb R^d}
|\omega^\alpha\psi^\sharp(\omega)|
\,d\omega\,d|\mu_\alpha|(z)
<\infty .
\]
Using \(Q=R\) a.e.,
\[
\begin{aligned}
\mathcal F^{-1}\{\mathcal I_Q\}[\psi]
&=
\mathcal I_Q[\psi^\sharp]  \\
&=
\int_{\mathbb R^d}Q(\omega)\psi^\sharp(\omega)\,d\omega  \\
&=
\sum_{|\alpha|\le k}
\int_{\mathbb R^d}
(i\omega)^\alpha
\varphi_{\mu_\alpha}(\omega)
\psi^\sharp(\omega)\,d\omega  \\
&=
\sum_{|\alpha|\le k}
\int_{\mathbb R^d}
\left[
\int_{\mathbb R^d}
(i\omega)^\alpha
e^{i\omega^\top z}
\psi^\sharp(\omega)\,d\omega
\right]d\mu_\alpha(z).
\end{aligned}
\]
By Fourier inversion under our convention,
\[
\psi(z)
=
\int_{\mathbb R^d}
e^{i\omega^\top z}\psi^\sharp(\omega)\,d\omega .
\]
Differentiating under the integral sign gives, for every \(|\alpha|\le k\),
\[
\partial^\alpha\psi(z)
=
\int_{\mathbb R^d}
(i\omega)^\alpha
e^{i\omega^\top z}
\psi^\sharp(\omega)\,d\omega .
\]
Therefore
\[
\mathcal F^{-1}\{\mathcal I_Q\}[\psi]
=
\sum_{|\alpha|\le k}
\int_{\mathbb R^d}
\partial^\alpha\psi(z)\,d\mu_\alpha(z),
\qquad
\psi\in\mathscr S(\mathbb R^d).
\]

\smallskip
\noindent\emph{Step 3: \(\Xi^k\)-continuity.}
Define, for \(\psi\in\Xi^k(\mathbb R^d;\mathbb C)\),
\[
L[\psi]
\coloneq
\sum_{|\alpha|\le k}
\int_{\mathbb R^d}
\partial^\alpha\psi(z)\,d\mu_\alpha(z).
\]
This is well-defined because
\(\partial^\alpha\psi\in C_0(\mathbb R^d;\mathbb C)\) for every
\(|\alpha|\le k\). Moreover,
\[
\begin{aligned}
|L[\psi]|
&\le
\sum_{|\alpha|\le k}
\left|
\int_{\mathbb R^d}
\partial^\alpha\psi(z)\,d\mu_\alpha(z)
\right|  \\
&\le
\sum_{|\alpha|\le k}
\|\partial^\alpha\psi\|_\infty
\|\mu_\alpha\|_{TV}  \\
&\le
\left(
\sum_{|\alpha|\le k}
\|\mu_\alpha\|_{TV}
\right)
\|\psi\|_{\Xi^k}.
\end{aligned}
\]
Thus \(L\) is a continuous linear functional on
\(\Xi^k(\mathbb R^d;\mathbb C)\). Step 2 shows that
\[
\mathcal F^{-1}\{\mathcal I_Q\}[\psi]
=
L[\psi],
\qquad
\psi\in\mathscr S(\mathbb R^d).
\]
Hence condition~\textup{(ii)} of
Theorem~\ref{T:master-fourier-identification} holds.

The assumptions of the proposition give the standing hypotheses of
Theorem~\ref{T:master-fourier-identification}. Moreover, Step 1 verifies
condition~\textup{(i)}, Step 3 verifies condition~\textup{(ii)}, and the
assumption that \(Q\) is a Fourier-domain representer for
\(\lambda_{g,V,y}\) gives condition~\textup{(iii)}. Applying
Theorem~\ref{T:master-fourier-identification}, we obtain, for every
\(f\in\mathcal A_V(\mathbb R^d)\),
\[
\mathcal T_{g,V,y}[f]
=
\lim_{n\to\infty}
\mathcal F^{-1}\{\mathcal I_Q\}[f_n],
\]
whenever \((f_n)_{n\ge1}\subset\mathscr S(\mathbb R^d)\) satisfies
\[
\|f_n-f\|_{\Xi^k}\to0.
\]
Since \(L\) is continuous on \(\Xi^k(\mathbb R^d;\mathbb C)\), Step 2 gives
\[
\lim_{n\to\infty}
\mathcal F^{-1}\{\mathcal I_Q\}[f_n]
=
\lim_{n\to\infty}L[f_n]
=
L[f].
\]
Therefore
\[
\mathcal T_{g,V,y}[f]
=
\sum_{|\alpha|\le k}
\int_{\mathbb R^d}
\partial^\alpha f(z)\,d\mu_\alpha(z),
\qquad
f\in\mathcal A_V(\mathbb R^d).
\]

Finally, under Assumption~\ref{A: independence},
Proposition~\ref{P: convolution} gives
\[
f_Y
=
f_V*P_X
\in
\mathcal A_V(\mathbb R^d).
\]
Taking \(f=f_Y\) in the preceding display yields
\[
\mathcal T_{g,V,y}[f_Y]
=
\sum_{|\alpha|\le k}
\int_{\mathbb R^d}
\partial^\alpha f_Y(z)\,d\mu_\alpha(z).
\]
By the ratio formula \eqref{Eq: ratio}, whenever \(f_Y(y)>0\),
\[
\mathbb E[g(X)\mid Y=y]
=
\frac{\mathcal T_{g,V,y}[f_Y]}{f_Y(y)}
=
\frac{
\displaystyle
\sum_{|\alpha|\le k}
\int_{\mathbb R^d}
\partial^\alpha f_Y(z)\,d\mu_\alpha(z)
}{
f_Y(y)
}.
\]
This is the stated posterior representation.
\end{proof}

\subsubsection{Proof of Proposition~\ref{P: conditional-heteroskedastic}}

\begin{proof}
We prove the three assertions separately. Throughout, \(\phi\) denotes the
standard \(d\)-dimensional Gaussian density.

\smallskip
\noindent\emph{Part (a): conditional Bayes formula.}
Since \(\mathbb R^d\) and \(\mathbb S_{++}^d\) are standard Borel spaces, there
exists a regular conditional law
\[
K(\Sigma_0,A)
\coloneq
P_{X\mid\Sigma}(A\mid \Sigma_0),
\qquad
A\in\mathcal B(\mathbb R^d).
\]
Thus \(K\) is a probability kernel from \(\mathbb S_{++}^d\) to
\(\mathbb R^d\), and
\[
\mathbb P(X\in A,\Sigma\in C)
=
\int_C K(\Sigma_0,A)\,dP_\Sigma(\Sigma_0)
\]
for all \(A\in\mathcal B(\mathbb R^d)\) and
\(C\in\mathcal B(\mathbb S_{++}^d)\).

Because \(V\indep (X,\Sigma)\), the product kernel
\[
K(\Sigma_0,\cdot)\otimes P_V
\]
is a version of \(P_{(X,V)\mid\Sigma}(\cdot\mid\Sigma_0)\). Indeed, for
\(A,B\in\mathcal B(\mathbb R^d)\) and
\(C\in\mathcal B(\mathbb S_{++}^d)\),
\[
\begin{aligned}
\mathbb P(X\in A,V\in B,\Sigma\in C)
&=
P_V(B)\mathbb P(X\in A,\Sigma\in C)  \\
&=
\int_C K(\Sigma_0,A)P_V(B)\,dP_\Sigma(\Sigma_0).
\end{aligned}
\]
The equality on rectangles determines the product kernel. For \(\Sigma_0\in\mathbb S_{++}^d\), define
\[
P^W_{\Sigma_0}(B)
\coloneq
P_{W\mid\Sigma}(B\mid\Sigma_0)
\coloneq
K(\Sigma_0,\Sigma_0^{1/2}B),
\qquad
B\in\mathcal B(\mathbb R^d),
\]
where
\[
\Sigma_0^{1/2}B
=
\{\Sigma_0^{1/2}w:w\in B\}.
\]
This is well-defined because \(w\mapsto\Sigma_0^{1/2}w\) is a homeomorphism of
\(\mathbb R^d\). It is also a probability kernel. For fixed \(B\),
\[
P^W_{\Sigma_0}(B)
=
\int_{\mathbb R^d}
\mathbf 1_B(\Sigma_0^{-1/2}x)\,K(\Sigma_0,dx),
\]
and the integrand is Borel in \((\Sigma_0,x)\).

Under the chosen conditional law given \(\Sigma=\Sigma_0\),
\[
W=\Sigma_0^{-1/2}X,
\qquad
Z=W+V,
\qquad
W\indep V,
\qquad
V\sim\mathcal N(0,I_d).
\]
Hence the conditional density of \(Z\) given \(\Sigma=\Sigma_0\) is
\[
f_{Z\mid\Sigma}(z\mid\Sigma_0)
=
\int_{\mathbb R^d}
\phi(z-w)\,dP^W_{\Sigma_0}(w).
\]
This density is strictly positive for every \(z\in\mathbb R^d\).

Now define the Gaussian density with covariance \(\Sigma_0\):
\[
\phi_{\Sigma_0}(u)
\coloneq
\det(\Sigma_0)^{-1/2}
\phi(\Sigma_0^{-1/2}u).
\]
Then the conditional density of \(Y\) given \(\Sigma=\Sigma_0\) is
\[
f_{Y\mid\Sigma}(y\mid\Sigma_0)
=
\int_{\mathbb R^d}
\phi_{\Sigma_0}(y-x)\,K(\Sigma_0,dx).
\]
Equivalently, with
\[
z_0\coloneq \Sigma_0^{-1/2}y,
\]
we have
\[
f_{Y\mid\Sigma}(y\mid\Sigma_0)
=
\det(\Sigma_0)^{-1/2}
f_{Z\mid\Sigma}(z_0\mid\Sigma_0).
\]
Therefore
\[
f_{Y\mid\Sigma}(y\mid\Sigma_0)>0
\qquad
\text{for all }(y,\Sigma_0)\in\mathbb R^d\times\mathbb S_{++}^d .
\]

For \(A\in\mathcal B(\mathbb R^d)\), define
\[
\Pi_{\Sigma_0,y}(A)
\coloneq
\frac{
\displaystyle
\int_A
\phi_{\Sigma_0}(y-x)\,K(\Sigma_0,dx)
}{
f_{Y\mid\Sigma}(y\mid\Sigma_0)
}.
\]
Then \(\Pi_{\Sigma_0,y}\) is a probability measure. For fixed \(A\), the map
\[
(\Sigma_0,y,x)
\mapsto
\mathbf 1_A(x)\phi_{\Sigma_0}(y-x)
\]
is nonnegative and Borel. Hence integration against the kernel
\(K(\Sigma_0,dx)\) shows that
\[
(\Sigma_0,y)\mapsto\Pi_{\Sigma_0,y}(A)
\]
is Borel measurable. Thus \(\Pi\) is a probability kernel from
\(\mathbb S_{++}^d\times\mathbb R^d\) to \(\mathbb R^d\).

We now verify that \(\Pi\) is a version of
\(P_{X\mid Y,\Sigma}\). Let
\(A,B\in\mathcal B(\mathbb R^d)\) and
\(C\in\mathcal B(\mathbb S_{++}^d)\). Then
\[
\begin{aligned}
&\int_C\int_B
\Pi_{\Sigma_0,y}(A)
f_{Y\mid\Sigma}(y\mid\Sigma_0)\,dy\,dP_\Sigma(\Sigma_0)  \\
&\qquad =
\int_C\int_B
\int_A
\phi_{\Sigma_0}(y-x)\,K(\Sigma_0,dx)\,dy\,dP_\Sigma(\Sigma_0)  \\
&\qquad =
\int_C
\int_{\mathbb R^d}
\int_{\mathbb R^d}
\mathbf 1_A(x)\mathbf 1_B(y)
\phi_{\Sigma_0}(y-x)\,dy\,
K(\Sigma_0,dx)\,dP_\Sigma(\Sigma_0)  \\
&\qquad =
\mathbb P(X\in A,Y\in B,\Sigma\in C).
\end{aligned}
\]
Therefore \(\Pi_{\Sigma_0,y}\) is a Bayes version of
\(P_{X\mid Y,\Sigma}(\cdot\mid y,\Sigma_0)\).

In the rest of the proof, the expression
\[
\mathbb E[g(X)\mid Y=y,\Sigma=\Sigma_0]
\]
is understood with respect to this version. If
\[
\int_{\mathbb R^d}
|g(x)|\phi_{\Sigma_0}(y-x)\,K(\Sigma_0,dx)<\infty,
\]
then
\[
\mathbb E[g(X)\mid Y=y,\Sigma=\Sigma_0]
=
\frac{
\displaystyle
\int_{\mathbb R^d}
g(x)\phi_{\Sigma_0}(y-x)\,K(\Sigma_0,dx)
}{
f_{Y\mid\Sigma}(y\mid\Sigma_0)
}.
\]
Using the definition of \(P^W_{\Sigma_0}\), the numerator becomes
\[
\int_{\mathbb R^d}
g(x)\phi_{\Sigma_0}(y-x)\,K(\Sigma_0,dx)
=
\det(\Sigma_0)^{-1/2}
\int_{\mathbb R^d}
g(\Sigma_0^{1/2}w)\phi(z_0-w)\,dP^W_{\Sigma_0}(w).
\]
The same determinant factor appears in
\[
f_{Y\mid\Sigma}(y\mid\Sigma_0)
=
\det(\Sigma_0)^{-1/2}
f_{Z\mid\Sigma}(z_0\mid\Sigma_0).
\]
After cancellation,
\[
\mathbb E[g(X)\mid Y=y,\Sigma=\Sigma_0]
=
\frac{
\displaystyle
\int_{\mathbb R^d}
g(\Sigma_0^{1/2}w)\phi(z_0-w)\,dP^W_{\Sigma_0}(w)
}{
f_{Z\mid\Sigma}(z_0\mid\Sigma_0)
}.
\]
This proves part~\textup{(a)}.

\smallskip
\noindent\emph{Part (b): standardized Tweedie functional.}
Fix \(\Sigma_0\in\mathbb S_{++}^d\) and \(y\in\mathbb R^d\), and set
\[
z_0=\Sigma_0^{-1/2}y.
\]
Define
\[
\lambda_{g,y,\Sigma_0}(w)
\coloneq
g(\Sigma_0^{1/2}w)\phi(z_0-w).
\]
Assume that
\[
\lambda_{g,y,\Sigma_0}\in C_0(\mathbb R^d;\mathbb C).
\]

For the standardized Gaussian noise, \(f_V=\phi\) and
\[
\varphi_V(\omega)=e^{-\|\omega\|_2^2/2}.
\]
The characteristic function is nowhere zero, so the convolution operator
\(\mu\mapsto \phi*\mu\) is injective on
\(\mathcal M(\mathbb R^d)\). Therefore, for
\(f=\phi*\mu\in\mathcal A_V(\mathbb R^d)\), the measure \(\mu\) is unique.

Define
\[
\mathcal T_{g,y,\Sigma_0}[f]
\coloneq
\int_{\mathbb R^d}
\lambda_{g,y,\Sigma_0}(w)\,d\mu(w),
\qquad
f=\phi*\mu.
\]
This is well-defined and linear. Moreover,
\[
\begin{aligned}
|\mathcal T_{g,y,\Sigma_0}[\phi*\mu]|
&\le
\|\lambda_{g,y,\Sigma_0}\|_\infty
\|\mu\|_{TV}  \\
&=
\|\lambda_{g,y,\Sigma_0}\|_\infty
\|\phi*\mu\|_{\mathcal A_V}.
\end{aligned}
\]
Thus
\[
\mathcal T_{g,y,\Sigma_0}\in\mathcal A_V'(\mathbb R^d).
\]
Uniqueness follows from the same injectivity argument: any continuous linear
functional with this integral representation must agree with
\(\mathcal T_{g,y,\Sigma_0}\) on every element of
\(\mathcal A_V(\mathbb R^d)\).

Since
\[
f_{Z\mid\Sigma}(\cdot\mid\Sigma_0)
=
\phi*P^W_{\Sigma_0},
\]
we have
\[
\mathcal T_{g,y,\Sigma_0}
\bigl[f_{Z\mid\Sigma}(\cdot\mid\Sigma_0)\bigr]
=
\int_{\mathbb R^d}
g(\Sigma_0^{1/2}w)\phi(z_0-w)\,dP^W_{\Sigma_0}(w).
\]
Combining this identity with part~\textup{(a)} gives
\[
\mathbb E[g(X)\mid Y=y,\Sigma=\Sigma_0]
=
\frac{
\mathcal T_{g,y,\Sigma_0}
\bigl[f_{Z\mid\Sigma}(\cdot\mid\Sigma_0)\bigr]
}{
f_{Z\mid\Sigma}(z_0\mid\Sigma_0)
}.
\]
This proves part~\textup{(b)}.

\smallskip
\noindent\emph{Part (c): Fourier representation.}
Fix \(\Sigma_0\in\mathbb S_{++}^d\) and \(y\in\mathbb R^d\), and keep
\[
z_0=\Sigma_0^{-1/2}y.
\]
Consider the standardized model
\[
W\sim P^W_{\Sigma_0},
\qquad
V\sim\mathcal N(0,I_d),
\qquad
W\indep V,
\qquad
Z=W+V.
\]
In this model, the noise density is \(\phi\), the characteristic function
\[
\varphi_V(\omega)=e^{-\|\omega\|_2^2/2}
\]
is nowhere zero, and \(\phi\in C_0^k(\mathbb R^d)\) for every \(k\ge0\).

Assume the hypotheses in part~\textup{(c)}. In particular,
\[
\lambda_{g,y,\Sigma_0}\in C_0(\mathbb R^d;\mathbb C),
\]
conditions~\textup{(i)} and~\textup{(ii)} of
Theorem~\ref{T:master-fourier-identification} hold for
\(\mathcal Q_{g,y,\Sigma_0}\), and \(\mathcal Q_{g,y,\Sigma_0}\) is a
Fourier-domain representer for \(\lambda_{g,y,\Sigma_0}\). Therefore
Theorem~\ref{T:master-fourier-identification} applies to the standardized
Gaussian model.

Set
\[
Q_{\Sigma_0,y}
\coloneq
\mathcal Q_{g,y,\Sigma_0}.
\]
The theorem implies that
\[
\mathcal F^{-1}\{\mathcal I_{Q_{\Sigma_0,y}}\}
\]
extends continuously to \(\Xi^k(\mathbb R^d;\mathbb C)\), and that its
restriction to \(\mathcal A_V(\mathbb R^d)\) equals
\(\mathcal T_{g,y,\Sigma_0}\). Hence, for every
\(f\in\mathcal A_V(\mathbb R^d)\) and every sequence
\((f_n)_{n\ge1}\subset\mathscr S(\mathbb R^d)\) such that
\[
\|f_n-f\|_{\Xi^k}\to0,
\]
we have
\[
\mathcal T_{g,y,\Sigma_0}[f]
=
\lim_{n\to\infty}
\mathcal F^{-1}\{\mathcal I_{Q_{\Sigma_0,y}}\}[f_n].
\]
Equivalently, using the Fourier convention fixed above,
\[
\mathcal T_{g,y,\Sigma_0}[f]
=
\lim_{n\to\infty}
\frac{1}{(2\pi)^d}
\int_{\mathbb R^d}
Q_{\Sigma_0,y}(\omega)\widetilde f_n(-\omega)\,d\omega .
\]
The limit is independent of the approximating sequence.

Taking
\[
f=f_{Z\mid\Sigma}(\cdot\mid\Sigma_0)
=
\phi*P^W_{\Sigma_0}
\]
gives
\[
\mathcal T_{g,y,\Sigma_0}
\bigl[f_{Z\mid\Sigma}(\cdot\mid\Sigma_0)\bigr]
=
\lim_{n\to\infty}
\frac{1}{(2\pi)^d}
\int_{\mathbb R^d}
Q_{\Sigma_0,y}(\omega)\widetilde f_n(-\omega)\,d\omega ,
\]
for any \((f_n)_{n\ge1}\subset\mathscr S(\mathbb R^d)\) satisfying
\[
\|f_n-f_{Z\mid\Sigma}(\cdot\mid\Sigma_0)\|_{\Xi^k}\to0.
\]
Combining this with the ratio from part~\textup{(b)} yields
\[
\mathbb E[g(X)\mid Y=y,\Sigma=\Sigma_0]
=
\frac{1}{f_{Z\mid\Sigma}(z_0\mid\Sigma_0)}
\lim_{n\to\infty}
\frac{1}{(2\pi)^d}
\int_{\mathbb R^d}
Q_{\Sigma_0,y}(\omega)\widetilde f_n(-\omega)\,d\omega .
\]
This proves part~\textup{(c)}.
\end{proof}

\subsection{Integrable Fourier representers}
\label{A:integrable-fourier-representers}

This section records a simple class of Tweedie functionals for which the
Fourier-domain representer is integrable. In this case, the inverse Fourier
transform is an ordinary bounded continuous function, and the Tweedie functional
can be written as an expectation with respect to the observed density.

\begin{proposition}[Expectation representation]
\label{P:UnbiasedExpectation}
Assume Assumptions~\ref{A: Identification} and~\ref{A:SmoothBoundedNoise} hold
with smoothness order \(k=0\). Fix \(y\in\mathbb R^d\), and let
\(g:\mathbb R^d\to\mathbb R\) be measurable. Suppose that
\[
\lambda_{g,V,y}\in C_0(\mathbb R^d;\mathbb C),
\]
and let $Q\coloneq \mathcal Q_{g,V,y}$ be a Fourier-domain representer for \(\lambda_{g,V,y}\). If
\[
Q\in L^1(\mathbb R^d;\mathbb C),
\]
then, for every \(f\in\mathcal A_V(\mathbb R^d)\),
\[
\mathcal T_{g,V,y}[f]
=
\int_{\mathbb R^d} Q^\sharp(z)f(z)\,dz .
\]

If, in addition, Assumptions~\ref{A: independence} and~\ref{A: Density} hold,
then, whenever \(f_Y(y)>0\),
\[
\mathbb E[g(X)\mid Y=y]
=
\frac{1}{f_Y(y)}
\int_{\mathbb R^d} Q^\sharp(z)f_Y(z)\,dz
=
\frac{\mathbb E[Q^\sharp(Y)]}{f_Y(y)} .
\]
\end{proposition}

The proposition says that an \(L^1\) Fourier-domain representer turns the
Tweedie numerator into an ordinary expectation under the observed law. Thus a
sample average of \(Q^\sharp(Y)\) estimates the numerator directly.

\begin{example}[A Gaussian expectation representation]
\label{E:GaussianExpectationRepresentation}
Suppose \(d=1\),
\[
Y=X+V,
\qquad
X\indep V,
\qquad
V\sim\mathcal N(0,1).
\]
Fix \(a>1\), and define
\[
g(x)
=
\frac{1}{a}
\exp\!\left(
\frac{a^2-1}{2a^2}x^2
\right).
\]
For fixed \(y\in\mathbb R\),
\[
\lambda_{g,V,y}(x)
=
g(x)f_V(y-x).
\]
Since
\[
f_V(u)
=
\frac{1}{\sqrt{2\pi}}e^{-u^2/2},
\]
completion of the square gives
\[
\lambda_{g,V,y}(x)
=
e^{(a^2-1)y^2/2}
\frac{1}{\sqrt{2\pi}\,a}
\exp\!\left(
-\frac{(x-a^2y)^2}{2a^2}
\right).
\]
Therefore
\[
\widetilde\lambda_{g,V,y}(\omega)
=
\exp\!\left(
\frac{(a^2-1)y^2}{2}
+i\omega a^2y
-\frac{a^2\omega^2}{2}
\right).
\]
Since
\[
\varphi_V(\omega)=e^{-\omega^2/2},
\]
a Fourier-domain representer is
\[
Q(\omega)
=
\frac{\widetilde\lambda_{g,V,y}(\omega)}{\varphi_V(-\omega)}
=
\exp\!\left(
\frac{(a^2-1)y^2}{2}
+i\omega a^2y
-\frac{(a^2-1)\omega^2}{2}
\right).
\]
Because \(a>1\), we have \(Q\in L^1(\mathbb R)\). Its inverse Fourier transform is
\[
Q^\sharp(z)
=
e^{(a^2-1)y^2/2}
\frac{1}{\sqrt{2\pi(a^2-1)}}
\exp\!\left(
-\frac{(z-a^2y)^2}{2(a^2-1)}
\right).
\]
Proposition~\ref{P:UnbiasedExpectation} gives
\[
\mathcal T_{g,V,y}[f_Y]
=
\mathbb E\!\left[Q^\sharp(Y)\right].
\]
Consequently, whenever \(f_Y(y)>0\),
\[
\mathbb E[g(X)\mid Y=y]
=
\frac{
\mathbb E\!\left[Q^\sharp(Y)\right]
}{
f_Y(y)
}.
\]
\end{example}

\subsubsection{Proof of Proposition~\ref{P:UnbiasedExpectation}}

\begin{proof}
Write
\[
Q\coloneq \mathcal Q_{g,V,y}.
\]

\smallskip
\noindent\emph{Step 1: temperedness.}
Since \(Q\in L^1(\mathbb R^d;\mathbb C)\), the functional
\[
\mathcal I_Q[\psi]
=
\int_{\mathbb R^d}Q(\omega)\psi(\omega)\,d\omega,
\qquad
\psi\in\mathscr S(\mathbb R^d),
\]
is well-defined. Moreover,
\[
|\mathcal I_Q[\psi]|
\le
\|Q\|_1\|\psi\|_\infty .
\]
Thus
\[
\mathcal I_Q\in\mathscr S'(\mathbb R^d).
\]

\smallskip
\noindent\emph{Step 2: inverse Fourier transform on the Schwartz class.}
Define
\[
Q^\sharp(z)
\coloneq
\frac{1}{(2\pi)^d}
\int_{\mathbb R^d}e^{-i\omega^\top z}Q(\omega)\,d\omega .
\]
By the Riemann--Lebesgue lemma,
\[
Q^\sharp\in C_0(\mathbb R^d;\mathbb C),
\]
and
\[
\|Q^\sharp\|_\infty
\le
(2\pi)^{-d}\|Q\|_1 .
\]

Let \(\psi\in\mathscr S(\mathbb R^d)\). Since
\[
\int_{\mathbb R^d}\int_{\mathbb R^d}
|Q(\omega)\psi(z)|\,d\omega\,dz
=
\|Q\|_1\|\psi\|_1
<\infty,
\]
Fubini's theorem applies. Therefore
\[
\begin{aligned}
\int_{\mathbb R^d}Q^\sharp(z)\psi(z)\,dz
&=
\frac{1}{(2\pi)^d}
\int_{\mathbb R^d}
Q(\omega)
\left[
\int_{\mathbb R^d}e^{-i\omega^\top z}\psi(z)\,dz
\right]d\omega  \\
&=
\frac{1}{(2\pi)^d}
\int_{\mathbb R^d}
Q(\omega)\widetilde\psi(-\omega)\,d\omega  \\
&=
\int_{\mathbb R^d}
Q(\omega)\psi^\sharp(\omega)\,d\omega  \\
&=
\mathcal I_Q[\psi^\sharp]  \\
&=
\mathcal F^{-1}\{\mathcal I_Q\}[\psi].
\end{aligned}
\]
Hence
\[
\mathcal F^{-1}\{\mathcal I_Q\}[\psi]
=
\int_{\mathbb R^d}Q^\sharp(z)\psi(z)\,dz,
\qquad
\psi\in\mathscr S(\mathbb R^d).
\]

\smallskip
\noindent\emph{Step 3: extension to \(\Xi^0\) and application of the master theorem.}
Define
\[
L[f]
\coloneq
\int_{\mathbb R^d}Q^\sharp(z)f(z)\,dz,
\qquad
f\in L^1(\mathbb R^d;\mathbb C).
\]
Since \(Q^\sharp\in L^\infty(\mathbb R^d;\mathbb C)\),
\[
|L[f]|
\le
\|Q^\sharp\|_\infty\|f\|_1.
\]
In particular, for \(f\in\Xi^0(\mathbb R^d;\mathbb C)\),
\[
|L[f]|
\le
\|Q^\sharp\|_\infty\|f\|_{\Xi^0}.
\]
Thus \(L\) is continuous on \(\Xi^0(\mathbb R^d;\mathbb C)\). Step 2 shows that
\(L\) extends \(\mathcal F^{-1}\{\mathcal I_Q\}\) from
\(\mathscr S(\mathbb R^d)\) to \(\Xi^0(\mathbb R^d;\mathbb C)\).

The assumptions of the proposition give the standing hypotheses of
Theorem~\ref{T:master-fourier-identification}. Moreover, Step 1 verifies
condition~\textup{(i)}, the preceding bound verifies condition~\textup{(ii)}
with \(k=0\), and the assumption that \(Q\) is a Fourier-domain representer for
\(\lambda_{g,V,y}\) gives condition~\textup{(iii)}. Therefore
Theorem~\ref{T:master-fourier-identification} gives, for every
\(f\in\mathcal A_V(\mathbb R^d)\),

\[
\mathcal T_{g,V,y}[f]
=
L[f]
=
\int_{\mathbb R^d}Q^\sharp(z)f(z)\,dz .
\]

Now assume Assumptions~\ref{A: independence} and~\ref{A: Density}. By
Proposition~\ref{P: convolution},
\[
f_Y
=
f_V*P_X
\in
\mathcal A_V(\mathbb R^d).
\]
Taking \(f=f_Y\) in the preceding display gives
\[
\mathcal T_{g,V,y}[f_Y]
=
\int_{\mathbb R^d}Q^\sharp(z)f_Y(z)\,dz.
\]
Since \(Q^\sharp\) is bounded, the right-hand side equals
\[
\mathbb E[Q^\sharp(Y)].
\]
Finally, the ratio formula \eqref{Eq: ratio} yields, whenever \(f_Y(y)>0\),
\[
\mathbb E[g(X)\mid Y=y]
=
\frac{\mathcal T_{g,V,y}[f_Y]}{f_Y(y)}
=
\frac{\mathbb E[Q^\sharp(Y)]}{f_Y(y)}.
\]
\end{proof}

\subsection*{Additional supporting material}
\label{A:supplement-introduction}

This supplement collects material that supports the main results but is not
needed in the main line of the paper. The first part extends the pointwise
Tweedie functional to a global operator acting on the class of mixtures for
which \(g(X)\) is integrable. This gives an operator-level version of the
posterior numerator and yields a smoothness statement for the posterior mean.

The second part records additional families of Fourier-domain representations.
These include  Hilbert-transform representers,
and one-sided compensated kernels. These results are useful for identifying
closed-form Tweedie functionals beyond the cases stated in the main text.

The final part gives the detailed computations behind the examples and tables.
It derives the posterior mean formulas for several one-dimensional noise laws,
works out functionals under the conventional Laplace mechanism, and records
formulas for the heteroskedastic Gaussian sequence model. These computations are
included to make the table entries and applications reproducible.

\subsection{The Tweedie operator}
\label{A:TweedieOperator}

This section passes from the pointwise identity at a fixed \(y\) to the
posterior mean as an integrable \(\sigma(Y)\)-measurable random variable. The
main step is to restrict attention to priors for which \(g(X)\) is integrable.

For a measurable function \(g:\mathbb R^d\to\mathbb R\), define
\[
\mathcal M_g(\mathbb R^d)
\coloneq
\left\{
\mu\in\mathcal M(\mathbb R^d):
\int_{\mathbb R^d}|g(x)|\,d|\mu|(x)<\infty
\right\},
\]
and
\[
\mathcal A_{V,g}(\mathbb R^d)
\coloneq
\mathcal K_V\bigl(\mathcal M_g(\mathbb R^d)\bigr).
\]
Thus \(\mathcal A_{V,g}(\mathbb R^d)\) is the class of \(V\)-mixtures whose
mixing measures integrate \(|g|\). The next lemma identifies this class as the
global domain on which the posterior mean exists in the usual \(L^1\) sense.

\begin{lemma}[Global domain]
\label{L:GlobalCEDomain}
Let \(X\sim P_X\), and let \(g:\mathbb R^d\to\mathbb R\) be measurable. Then
\[
\mathbb E[g(X)\mid Y]
\]
exists as an integrable \(\sigma(Y)\)-measurable random variable if and only if
\[
P_X\in\mathcal M_g(\mathbb R^d).
\]
\end{lemma}

\begin{proof}
Set
\[
Z\coloneq g(X).
\]
The conditional expectation \(\mathbb E[Z\mid Y]\) exists as an integrable
\(\sigma(Y)\)-measurable random variable if and only if \(Z\in L^1\). Since the
marginal law of \(X\) is \(P_X\),
\[
\mathbb E|Z|
=
\int_{\mathbb R^d}|g(x)|\,dP_X(x).
\]
Thus \(Z\in L^1\) if and only if
\[
\int_{\mathbb R^d}|g(x)|\,dP_X(x)<\infty .
\]
Because \(P_X\) is a probability measure, \(|P_X|=P_X\), so this condition is
equivalent to
\[
P_X\in\mathcal M_g(\mathbb R^d).
\]
\end{proof}

The pointwise functionals
\(\{\mathcal T_{g,V,y}:y\in\mathbb R^d\}\) can now be assembled into a single
operator. Instead of fixing \(y\), the operator returns the numerator function
\[
y\longmapsto
\int_{\mathbb R^d}g(x)f_V(y-x)\,d\mu(x).
\]
Pointwise evaluation then recovers the scalar Tweedie functional:
\[
(\mathscr T_{g,V}[f])(y)
=
\mathcal T_{g,V,y}[f],
\]
whenever the latter is defined.

\begin{proposition}[The Tweedie operator]
\label{P:operator-Tweedie}
Assume Assumptions~\ref{A: Identification} and~\ref{A:SmoothBoundedNoise} hold
for some \(k\ge0\). Let \(g:\mathbb R^d\to\mathbb R\) be measurable. Then there
exists a unique linear operator
\[
\mathscr T_{g,V}:
\mathcal A_{V,g}(\mathbb R^d)
\to
\Xi^k(\mathbb R^d;\mathbb C)
\]
such that, for every \(\mu\in\mathcal M_g(\mathbb R^d)\),
\[
\mathscr T_{g,V}\bigl[\mathcal K_V[\mu]\bigr]
=
\int_{\mathbb R^d}g(x)\tau_x f_V\,d\mu(x)
\qquad
\text{in } \Xi^k(\mathbb R^d;\mathbb C),
\]
where the integral is understood in the Bochner sense with respect to the
finite complex measure \(\mu\).

Moreover, if \(f=\mathcal K_V[\mu]\in\mathcal A_{V,g}(\mathbb R^d)\), then
\[
(\mathscr T_{g,V}[f])(y)
=
\int_{\mathbb R^d}g(x)f_V(y-x)\,d\mu(x),
\qquad
y\in\mathbb R^d.
\]
If, in addition,
\[
\lambda_{g,V,y}\in C_0(\mathbb R^d;\mathbb C),
\]
then
\[
(\mathscr T_{g,V}[f])(y)
=
\mathcal T_{g,V,y}[f].
\]

Finally, suppose that \(P_X\in\mathcal M_g(\mathbb R^d)\) and that
Assumptions~\ref{A: independence} and~\ref{A: Density} hold. Let
\[
f_Y
=
\mathcal K_V[P_X]
\in
\mathcal A_{V,g}(\mathbb R^d).
\]
Then
\[
\mathbb E[g(X)\mid Y]
=
\frac{\mathscr T_{g,V}[f_Y](Y)}{f_Y(Y)}
\qquad
\mathbb P\text{-almost surely}.
\]
\end{proposition}

\begin{proof}
Let
\[
f=\mathcal K_V[\mu]\in\mathcal A_{V,g}(\mathbb R^d).
\]
By Assumption~\ref{A: Identification}, the representing measure \(\mu\) is
unique.

We first show that the Bochner integral is well-defined. By
Assumption~\ref{A:SmoothBoundedNoise},
\[
f_V\in\Xi^k(\mathbb R^d;\mathbb C).
\]
By Lemma~\ref{L:translation-xik}, the map
\[
x\longmapsto \tau_x f_V
\]
is continuous from \(\mathbb R^d\) into \(\Xi^k(\mathbb R^d;\mathbb C)\).
Since \(g\) is measurable and scalar multiplication is continuous, the map
\[
x\longmapsto g(x)\tau_x f_V
\]
is strongly measurable. Translations are isometries on \(\Xi^k\), so
\[
\|g(x)\tau_x f_V\|_{\Xi^k}
=
|g(x)|\|f_V\|_{\Xi^k}.
\]
Therefore
\[
\int_{\mathbb R^d}
\|g(x)\tau_x f_V\|_{\Xi^k}\,d|\mu|(x)
=
\|f_V\|_{\Xi^k}
\int_{\mathbb R^d}|g(x)|\,d|\mu|(x)
<\infty,
\]
because \(\mu\in\mathcal M_g(\mathbb R^d)\). Hence
\(x\mapsto g(x)\tau_x f_V\) is Bochner integrable with respect to \(\mu\).

Define
\[
\mathscr T_{g,V}[f]
\coloneq
\int_{\mathbb R^d}g(x)\tau_x f_V\,d\mu(x).
\]
The definition is unambiguous because \(\mu=\mathcal K_V^{-1}[f]\) is unique.
Linearity follows from the linearity of \(\mathcal K_V^{-1}\) and of the
Bochner integral. This proves existence and uniqueness of the operator with the
displayed representation.

We now identify its pointwise values. For each \(y\in\mathbb R^d\), evaluation
at \(y\) is a continuous linear functional on \(\Xi^k(\mathbb R^d;\mathbb C)\),
because
\[
|h(y)|\le \|h\|_\infty\le \|h\|_{\Xi^k}.
\]
Thus evaluation commutes with the Bochner integral:
\[
\begin{aligned}
(\mathscr T_{g,V}[f])(y)
&=
\int_{\mathbb R^d}
g(x)(\tau_x f_V)(y)\,d\mu(x)  \\
&=
\int_{\mathbb R^d}
g(x)f_V(y-x)\,d\mu(x).
\end{aligned}
\]
If \(\lambda_{g,V,y}\in C_0(\mathbb R^d;\mathbb C)\), then
Theorem~\ref{T: Representation} gives
\[
\mathcal T_{g,V,y}[f]
=
\int_{\mathbb R^d}
g(x)f_V(y-x)\,d\mu(x).
\]
Hence
\[
(\mathscr T_{g,V}[f])(y)
=
\mathcal T_{g,V,y}[f].
\]

It remains to prove the almost-sure posterior representation. Suppose that
\(P_X\in\mathcal M_g(\mathbb R^d)\) and that
Assumptions~\ref{A: independence} and~\ref{A: Density} hold. By
Proposition~\ref{P: convolution},
\[
f_Y
=
f_V*P_X
=
\mathcal K_V[P_X],
\]
so
\[
f_Y\in\mathcal A_{V,g}(\mathbb R^d).
\]
Moreover, since \(f_V\) is bounded under Assumption~\ref{A:SmoothBoundedNoise},
\[
\int_{\mathbb R^d}
|g(x)|f_V(y-x)\,dP_X(x)
\le
\|f_V\|_\infty
\int_{\mathbb R^d}|g(x)|\,dP_X(x)
<\infty.
\]
Thus the pointwise ratio formula applies for every \(y\) such that
\(f_Y(y)>0\):
\[
\mathbb E[g(X)\mid Y=y]
=
\frac{
\displaystyle
\int_{\mathbb R^d}g(x)f_V(y-x)\,dP_X(x)
}{
f_Y(y)
}.
\]
Using the pointwise identity above with \(\mu=P_X\), we obtain
\[
\mathbb E[g(X)\mid Y=y]
=
\frac{(\mathscr T_{g,V}[f_Y])(y)}{f_Y(y)},
\qquad
f_Y(y)>0.
\]

Finally,
\[
\mathbb P\{f_Y(Y)=0\}
=
P_Y\{y:f_Y(y)=0\}
=
\int_{\{y:f_Y(y)=0\}}f_Y(y)\,dy
=
0.
\]
Hence the ratio is defined \(\mathbb P\)-almost surely, with arbitrary values on
\(\{f_Y(Y)=0\}\). Therefore
\[
\mathbb E[g(X)\mid Y]
=
\frac{\mathscr T_{g,V}[f_Y](Y)}{f_Y(Y)}
\qquad
\mathbb P\text{-almost surely}.
\]
\end{proof}

Proposition~\ref{P:operator-Tweedie} factors the pointwise Tweedie functional as
\[
\mathcal T_{g,V,y}
=
\operatorname{ev}_y\circ \mathscr T_{g,V}.
\]
Thus the dependence on \(y\) enters only through evaluation after the global
operator \(\mathscr T_{g,V}\) has acted on \(f_Y\). This factorization gives a
direct regularity statement for the posterior mean.

\begin{corollary}[Smoothness of the posterior mean]
\label{C:PosteriorMeanSmoothness}
Assume the hypotheses of the final assertion of
Proposition~\ref{P:operator-Tweedie}. Then \(\mathbb E[g(X)\mid Y]\) admits the
version
\[
\theta_V^g(y)
=
\frac{(\mathscr T_{g,V}[f_Y])(y)}{f_Y(y)},
\qquad
f_Y(y)>0.
\]
This version is \(k\)-times continuously differentiable on the open set
\[
\mathcal O_Y
\coloneq
\{y\in\mathbb R^d:f_Y(y)>0\}.
\]
\end{corollary}

\begin{proof}
By Proposition~\ref{P:operator-Tweedie},
\[
N_g
\coloneq
\mathscr T_{g,V}[f_Y]
\in
\Xi^k(\mathbb R^d;\mathbb C).
\]
By Lemma~\ref{L: SmoothModels}(b),
\[
f_Y=\mathcal K_V[P_X]\in\Xi^k(\mathbb R^d;\mathbb C).
\]
Hence \(N_g\) and \(f_Y\) both belong to \(C^k(\mathbb R^d;\mathbb C)\).

Since \(f_Y\) is continuous,
\[
\mathcal O_Y
=
\{y\in\mathbb R^d:f_Y(y)>0\}
\]
is open. On \(\mathcal O_Y\), the denominator does not vanish. Therefore the
quotient rule gives
\[
\frac{N_g}{f_Y}
\in
C^k(\mathcal O_Y;\mathbb C).
\]
By Proposition~\ref{P:operator-Tweedie}, this quotient is a version of
\(\mathbb E[g(X)\mid Y]\) on \(\mathcal O_Y\).
\end{proof}

\subsection{Detailed computations for the table formulas} \label{A: DetailedComp}

\subsubsection{Posterior means under one-dimensional noise laws}
\label{A:Calculations}

This section records the Fourier calculations underlying the entries in
Table~\ref{tab:noise-families}. We specialize to the posterior mean case
\(g(x)=x\) in dimension \(d=1\). Thus
\[
\lambda_{g,V,y}(x)
=
x f_V(y-x).
\]
The calculation has two steps. First, we find a Fourier-domain representer
\(\mathcal Q_{g,V,y}\) using
Theorem~\ref{T:master-fourier-identification}. Second, we identify the inverse
Fourier transform
\[
\mathcal F^{-1}\{\mathcal I_{\mathcal Q_{g,V,y}}\}
\]
using the distributional identities developed in Section~\ref{S: Computation}.

\paragraph{A universal one-dimensional symbol}
\label{A:AuxiliaryResults}

The next lemma isolates the calculation used in the examples below. The same
identity appears in \cite{raphan2011least}, Equation~6.4.

\begin{lemma}[Universal Fourier-domain representer for posterior means]
\label{L:posterior-mean-symbol-1d}
Assume \(d=1\), let \(g(x)=x\), and fix \(y\in\mathbb R\). Define
\[
\lambda_{g,V,y}(x)
\coloneq
x f_V(y-x).
\]
Suppose that
\[
\lambda_{g,V,y}\in \Xi^0(\mathbb R;\mathbb C).
\]
Let
\[
Z_V
\coloneq
\{\omega\in\mathbb R:\varphi_V(-\omega)=0\}.
\]
If \(Z_V\) has Lebesgue measure zero, then the function
\[
\mathcal Q_{g,V,y}(\omega)
\coloneq
\begin{cases}
\displaystyle
e^{i\omega y}
\left[
y
-
i
\frac{\frac{d}{d\omega}\varphi_V(-\omega)}
{\varphi_V(-\omega)}
\right],
&
\omega\notin Z_V, \\[2.0ex]
0,
&
\omega\in Z_V,
\end{cases}
\]
is a Fourier-domain representer for \(\lambda_{g,V,y}\).
\end{lemma}

\begin{proof}
By definition,
\[
\widetilde\lambda_{g,V,y}(\omega)
=
\int_{\mathbb R}
e^{i\omega x}x f_V(y-x)\,dx .
\]
Make the change of variables \(u=y-x\). Then \(x=y-u\), and
\[
\widetilde\lambda_{g,V,y}(\omega)
=
\int_{\mathbb R}
e^{i\omega(y-u)}(y-u)f_V(u)\,du .
\]
Therefore
\[
\widetilde\lambda_{g,V,y}(\omega)
=
e^{i\omega y}
\left[
y
\int_{\mathbb R}e^{-i\omega u}f_V(u)\,du
-
\int_{\mathbb R}u e^{-i\omega u}f_V(u)\,du
\right].
\]
The first integral is \(\varphi_V(-\omega)\).

It remains to identify the second integral. Since
\(\lambda_{g,V,y}\in L^1(\mathbb R)\), the same change of variables gives
\[
\int_{\mathbb R}|y-u|f_V(u)\,du<\infty .
\]
Because \(f_V\) is a probability density,
\[
\int_{\mathbb R}|u|f_V(u)\,du
\le
\int_{\mathbb R}|y-u|f_V(u)\,du
+
|y|\int_{\mathbb R}f_V(u)\,du
<\infty .
\]
Thus \(u\mapsto u f_V(u)\) is integrable. By dominated convergence,
\[
\frac{d}{d\omega}\varphi_V(-\omega)
=
\frac{d}{d\omega}
\int_{\mathbb R}e^{-i\omega u}f_V(u)\,du
=
\int_{\mathbb R}(-iu)e^{-i\omega u}f_V(u)\,du .
\]
Hence
\[
\int_{\mathbb R}u e^{-i\omega u}f_V(u)\,du
=
i\frac{d}{d\omega}\varphi_V(-\omega).
\]
Substituting this identity into the expression for
\(\widetilde\lambda_{g,V,y}\) gives
\[
\widetilde\lambda_{g,V,y}(\omega)
=
e^{i\omega y}
\left[
y\varphi_V(-\omega)
-
i\frac{d}{d\omega}\varphi_V(-\omega)
\right].
\]

On \(Z_V^c\), division by \(\varphi_V(-\omega)\) gives
\[
\frac{\widetilde\lambda_{g,V,y}(\omega)}
{\varphi_V(-\omega)}
=
e^{i\omega y}
\left[
y
-
i
\frac{\frac{d}{d\omega}\varphi_V(-\omega)}
{\varphi_V(-\omega)}
\right].
\]
On \(Z_V\), we set \(\mathcal Q_{g,V,y}=0\). Since \(Z_V\) is Lebesgue null,
\[
\mathcal Q_{g,V,y}(\omega)\varphi_V(-\omega)
=
\widetilde\lambda_{g,V,y}(\omega)
\qquad
\text{for Lebesgue-a.e. }\omega .
\]
Thus \(\mathcal Q_{g,V,y}\) is a Fourier-domain representer.
\end{proof}

\paragraph{Hilbert-transform representers}
\label{A:hilbert-transform-representers}

This subsection covers Fourier-domain representers with a singular
Hilbert-transform component. We use the convention
\[
\mathcal H[h](y)
\coloneq
\frac{1}{\pi}
\operatorname{p.v.}
\int_{\mathbb R}
\frac{h(z)}{y-z}\,dz
=
\frac{1}{\pi}
\lim_{\varepsilon\downarrow0}
\int_{|t|>\varepsilon}
\frac{h(y-t)}{t}\,dt .
\]
For functions in \(\Xi^1(\mathbb R;\mathbb C)\), the principal value below is
well-defined by the estimates in the proof.

\begin{proposition}[Tweedie representation via the Hilbert transform]
\label{P:pv-kernel-functional}
Assume \(d=1\), Assumption~\ref{A: Identification} holds, and
Assumption~\ref{A:SmoothBoundedNoise} holds with smoothness order \(1\). Fix
\(y\in\mathbb R\), and let \(g:\mathbb R\to\mathbb R\) be measurable. Assume
\[
\lambda_{g,V,y}\in C_0(\mathbb R;\mathbb C),
\]
and let \(\mathcal Q_{g,V,y}\) be a Fourier-domain representer for
\(\lambda_{g,V,y}\).

Suppose there exist \(c_{g,V,y}\in\mathbb C\) and a measurable function
\(k_{g,V,y}:(0,\infty)\to\mathbb C\) such that
\[
\int_0^1 t\,|k_{g,V,y}(t)|\,dt<\infty,
\qquad
\int_1^\infty |k_{g,V,y}(t)|\,dt<\infty,
\]
and
\[
\mathcal Q_{g,V,y}(\omega)
=
e^{i\omega y}
\left[
-i\pi c_{g,V,y}\operatorname{sgn}(\omega)
+
\int_0^\infty
\bigl(e^{-i\omega t}-e^{i\omega t}\bigr)
k_{g,V,y}(t)\,dt
\right]
\]
for Lebesgue-a.e. \(\omega\in\mathbb R\). Then, for every
\(f\in\mathcal A_V(\mathbb R)\),
\[
\mathcal T_{g,V,y}[f]
=
\pi c_{g,V,y}\mathcal H[f](y)
+
\int_0^\infty
\bigl(f(y-t)-f(y+t)\bigr)k_{g,V,y}(t)\,dt .
\]

If, in addition, Assumptions~\ref{A: independence} and~\ref{A: Density} hold,
then, whenever \(f_Y(y)>0\),
\[
\mathbb E[g(X)\mid Y=y]
=
\frac{1}{f_Y(y)}
\left\{
\pi c_{g,V,y}\mathcal H[f_Y](y)
+
\int_0^\infty
\bigl(f_Y(y-t)-f_Y(y+t)\bigr)k_{g,V,y}(t)\,dt
\right\}.
\]
\end{proposition}

\begin{proof}
Write
\[
c\coloneq c_{g,V,y},
\qquad
\kappa\coloneq k_{g,V,y},
\qquad
Q\coloneq \mathcal Q_{g,V,y}.
\]

\smallskip
\noindent\emph{Step 1: temperedness.}
By assumption, \(Q\) agrees Lebesgue-a.e. with
\[
Q_*(\omega)
\coloneq
e^{i\omega y}
\left[
-i\pi c\,\operatorname{sgn}(\omega)
+
q_\kappa(\omega)
\right],
\]
where
\[
q_\kappa(\omega)
\coloneq
\int_0^\infty
\bigl(e^{-i\omega t}-e^{i\omega t}\bigr)\kappa(t)\,dt .
\]
The integrability assumptions on \(\kappa\) imply that \(q_\kappa\) is
well-defined. Indeed,
\[
|e^{-i\omega t}-e^{i\omega t}|
=
2|\sin(\omega t)|
\le
2\min\{|\omega|t,1\}.
\]
Therefore
\[
|q_\kappa(\omega)|
\le
2|\omega|\int_0^1 t|\kappa(t)|\,dt
+
2\int_1^\infty |\kappa(t)|\,dt .
\]
Hence there is a finite constant \(C>0\) such that
\[
|Q_*(\omega)|
\le
C(1+|\omega|),
\qquad
\omega\in\mathbb R .
\]
Since \(Q=Q_*\) Lebesgue-a.e., the function \(Q\) may be replaced by
\(Q_*\) for distributional purposes. For any integer \(N>2\),
\[
\int_{\mathbb R}
\frac{|Q(\omega)|}{(1+|\omega|)^N}\,d\omega
<\infty .
\]
Lemma~\ref{L:weighted-integrability-temperedness} gives
\[
\mathcal I_Q\in\mathscr S'(\mathbb R).
\]

\smallskip
\noindent\emph{Step 2: inverse Fourier transform on the Schwartz class.}
Let \(\psi\in\mathscr S(\mathbb R)\), and write
\[
\psi^\sharp(\omega)
=
(2\pi)^{-1}\widetilde\psi(-\omega).
\]
Since \(\psi^\sharp\in\mathscr S(\mathbb R)\),
\[
\int_{\mathbb R}|\psi^\sharp(\omega)|\,d\omega<\infty,
\qquad
\int_{\mathbb R}|\omega|\,|\psi^\sharp(\omega)|\,d\omega<\infty .
\]
Moreover,
\[
\begin{aligned}
&\int_0^\infty
\int_{\mathbb R}
\left|
e^{i\omega y}
\bigl(e^{-i\omega t}-e^{i\omega t}\bigr)
\kappa(t)\psi^\sharp(\omega)
\right|
\,d\omega\,dt  \\
&\qquad\le
2\int_0^1 t|\kappa(t)|\,dt
\int_{\mathbb R}|\omega|\,|\psi^\sharp(\omega)|\,d\omega
+
2\int_1^\infty |\kappa(t)|\,dt
\int_{\mathbb R}|\psi^\sharp(\omega)|\,d\omega
<\infty .
\end{aligned}
\]
Thus Fubini's theorem applies to the \(q_\kappa\)-term. Fourier inversion gives
\[
\int_{\mathbb R}
e^{i\omega(y-t)}\psi^\sharp(\omega)\,d\omega
=
\psi(y-t),
\qquad
\int_{\mathbb R}
e^{i\omega(y+t)}\psi^\sharp(\omega)\,d\omega
=
\psi(y+t).
\]
Hence
\[
\int_{\mathbb R}
e^{i\omega y}q_\kappa(\omega)\psi^\sharp(\omega)\,d\omega
=
\int_0^\infty
\bigl(\psi(y-t)-\psi(y+t)\bigr)\kappa(t)\,dt .
\]

It remains to handle the singular term. Under the Fourier convention used in
this paper, the identity in \cite{grafakos2008classical}, Chapter~5,
Equation~5.1.12, gives
\[
\mathcal F^{-1}
\left\{
\mathcal I_{-i\pi\operatorname{sgn}}
\right\}
=
-\operatorname{p.v.}\!\left(\frac{1}{\,\cdot\,}\right).
\]
Multiplication by \(e^{i\omega y}\) shifts the inverse transform. Therefore
\[
\begin{aligned}
\mathcal F^{-1}
\left\{
\mathcal I_{e^{i\omega y}(-i\pi c\operatorname{sgn}(\omega))}
\right\}[\psi]
&=
c\operatorname{p.v.}
\int_{\mathbb R}
\frac{\psi(z)}{y-z}\,dz  \\
&=
\pi c\,\mathcal H[\psi](y).
\end{aligned}
\]
Combining the regular and singular parts gives, for every
\(\psi\in\mathscr S(\mathbb R)\),
\[
\mathcal F^{-1}\{\mathcal I_Q\}[\psi]
=
\pi c\,\mathcal H[\psi](y)
+
\int_0^\infty
\bigl(\psi(y-t)-\psi(y+t)\bigr)\kappa(t)\,dt .
\]

\smallskip
\noindent\emph{Step 3: \(\Xi^1\)-continuity.}
For \(\psi\in\Xi^1(\mathbb R;\mathbb C)\), define
\[
L[\psi]
\coloneq
\pi c\,\mathcal H[\psi](y)
+
\int_0^\infty
\bigl(\psi(y-t)-\psi(y+t)\bigr)\kappa(t)\,dt .
\]
We show that \(L\) is well-defined and continuous on
\(\Xi^1(\mathbb R;\mathbb C)\).

First consider the regular integral. For \(0<t\le1\),
\[
\psi(y-t)-\psi(y+t)
=
-\int_{-t}^t \psi'(y+s)\,ds,
\]
so
\[
|\psi(y-t)-\psi(y+t)|
\le
2t\|\psi'\|_\infty .
\]
For \(t\ge1\),
\[
|\psi(y-t)-\psi(y+t)|
\le
2\|\psi\|_\infty .
\]
Therefore
\[
\int_0^\infty
|\psi(y-t)-\psi(y+t)|\,|\kappa(t)|\,dt
\le
2\|\psi'\|_\infty\int_0^1 t|\kappa(t)|\,dt
+
2\|\psi\|_\infty\int_1^\infty|\kappa(t)|\,dt .
\]

Now consider the Hilbert-transform term. For \(\varepsilon>0\),
\[
\int_{|t|>\varepsilon}
\frac{\psi(y-t)}{t}\,dt
=
\int_\varepsilon^\infty
\frac{\psi(y-t)-\psi(y+t)}{t}\,dt .
\]
On \(0<t\le1\),
\[
\frac{|\psi(y-t)-\psi(y+t)|}{t}
\le
2\|\psi'\|_\infty .
\]
On \(t\ge1\),
\[
\begin{aligned}
\int_1^\infty \frac{|\psi(y-t)|}{t}\,dt
&=
\int_{-\infty}^{y-1}
\frac{|\psi(u)|}{y-u}\,du
\le
\|\psi\|_1, \\
\int_1^\infty \frac{|\psi(y+t)|}{t}\,dt
&=
\int_{y+1}^{\infty}
\frac{|\psi(u)|}{u-y}\,du
\le
\|\psi\|_1 .
\end{aligned}
\]
Thus the principal value defining \(\mathcal H[\psi](y)\) exists, and
\[
|\mathcal H[\psi](y)|
\le
\frac{1}{\pi}
\left(
2\|\psi'\|_\infty+2\|\psi\|_1
\right).
\]
Combining the bounds gives a finite constant \(C_{g,V,y}\) such that
\[
|L[\psi]|
\le
C_{g,V,y}\|\psi\|_{\Xi^1},
\qquad
\psi\in\Xi^1(\mathbb R;\mathbb C).
\]
Hence
\[
L\in\bigl(\Xi^1(\mathbb R;\mathbb C)\bigr)'.
\]
Step 2 shows that
\[
L[\psi]
=
\mathcal F^{-1}\{\mathcal I_Q\}[\psi],
\qquad
\psi\in\mathscr S(\mathbb R).
\]

\smallskip
\noindent\emph{Step 4: passage to mixtures.}
The assumptions of the proposition give the standing hypotheses of
Theorem~\ref{T:master-fourier-identification}. Moreover, Step 1 verifies
condition~\textup{(i)}, Step 3 verifies condition~\textup{(ii)} with smoothness
order \(1\), and the assumption that \(Q\) is a Fourier-domain representer for
\(\lambda_{g,V,y}\) gives condition~\textup{(iii)}. Therefore
Theorem~\ref{T:master-fourier-identification} applies.

Let \(f\in\mathcal A_V(\mathbb R)\), and choose
\((f_n)_{n\ge1}\subset\mathscr S(\mathbb R)\) such that
\[
\|f_n-f\|_{\Xi^1}\to0 .
\]
Then
\[
\mathcal T_{g,V,y}[f]
=
\lim_{n\to\infty}
\mathcal F^{-1}\{\mathcal I_Q\}[f_n].
\]
Using Step 2 and the continuity of \(L\),
\[
\lim_{n\to\infty}
\mathcal F^{-1}\{\mathcal I_Q\}[f_n]
=
\lim_{n\to\infty}L[f_n]
=
L[f].
\]
Thus
\[
\mathcal T_{g,V,y}[f]
=
\pi c\,\mathcal H[f](y)
+
\int_0^\infty
\bigl(f(y-t)-f(y+t)\bigr)\kappa(t)\,dt .
\]
Restoring \(c=c_{g,V,y}\) and \(\kappa=k_{g,V,y}\) proves the functional
identity.

Finally, assume Assumptions~\ref{A: independence} and~\ref{A: Density}. By
Proposition~\ref{P: convolution},
\[
f_Y
=
f_V*P_X
\in
\mathcal A_V(\mathbb R).
\]
Taking \(f=f_Y\) in the functional identity and applying the ratio formula
\eqref{Eq: ratio} gives, whenever \(f_Y(y)>0\),
\[
\mathbb E[g(X)\mid Y=y]
=
\frac{1}{f_Y(y)}
\left\{
\pi c_{g,V,y}\mathcal H[f_Y](y)
+
\int_0^\infty
\bigl(f_Y(y-t)-f_Y(y+t)\bigr)k_{g,V,y}(t)\,dt
\right\}.
\]
\end{proof}

\paragraph{One-sided compensated kernels}
\label{A:one-sided-compensated-kernels}

The next family covers Fourier-domain representers whose inverse transform is a
one-sided compensated kernel.

\begin{proposition}[Tweedie representation via one-sided compensated kernels]
\label{P:onesided-kernel-functional}
Assume \(d=1\), Assumption~\ref{A: Identification} holds, and
Assumption~\ref{A:SmoothBoundedNoise} holds with smoothness order \(1\). Fix
\(y\in\mathbb R\), and let \(g:\mathbb R\to\mathbb R\) be measurable. Assume
\[
\lambda_{g,V,y}\in C_0(\mathbb R;\mathbb C),
\]
and let \(\mathcal Q_{g,V,y}\) be a Fourier-domain representer for
\(\lambda_{g,V,y}\).

Suppose there exist \(c_{g,V,y}\in\mathbb C\) and a measurable function
\(k_{g,V,y}:(0,\infty)\to\mathbb C\) such that
\[
\int_0^1 t\,|k_{g,V,y}(t)|\,dt<\infty,
\qquad
\int_1^\infty |k_{g,V,y}(t)|\,dt<\infty,
\]
and
\[
\mathcal Q_{g,V,y}(\omega)
=
e^{i\omega y}
\left[
c_{g,V,y}
+
\int_0^\infty
\bigl(1-e^{-i\omega t}\bigr)k_{g,V,y}(t)\,dt
\right]
\]
for Lebesgue-a.e. \(\omega\in\mathbb R\). Then, for every
\(f\in\mathcal A_V(\mathbb R)\),
\[
\mathcal T_{g,V,y}[f]
=
c_{g,V,y}f(y)
+
\int_0^\infty
\bigl(f(y)-f(y-t)\bigr)k_{g,V,y}(t)\,dt .
\]

If, in addition, Assumptions~\ref{A: independence} and~\ref{A: Density} hold,
then, whenever \(f_Y(y)>0\),
\[
\mathbb E[g(X)\mid Y=y]
=
\frac{1}{f_Y(y)}
\left\{
c_{g,V,y}f_Y(y)
+
\int_0^\infty
\bigl(f_Y(y)-f_Y(y-t)\bigr)k_{g,V,y}(t)\,dt
\right\}.
\]
\end{proposition}

\begin{proof}
Write
\[
c\coloneq c_{g,V,y},
\qquad
\kappa\coloneq k_{g,V,y},
\qquad
Q\coloneq \mathcal Q_{g,V,y}.
\]

\smallskip
\noindent\emph{Step 1: temperedness.}
By assumption, \(Q\) agrees Lebesgue-a.e. with
\[
Q_*(\omega)
\coloneq
e^{i\omega y}
\bigl(c+q_\kappa(\omega)\bigr),
\]
where
\[
q_\kappa(\omega)
\coloneq
\int_0^\infty
\bigl(1-e^{-i\omega t}\bigr)\kappa(t)\,dt .
\]
The integrability assumptions on \(\kappa\) imply that \(q_\kappa\) is
well-defined. Indeed,
\[
|1-e^{-i\omega t}|
\le
\min\{2,|\omega|t\}.
\]
Therefore
\[
|q_\kappa(\omega)|
\le
|\omega|\int_0^1 t|\kappa(t)|\,dt
+
2\int_1^\infty |\kappa(t)|\,dt .
\]
Hence there is a finite constant \(C>0\) such that
\[
|Q_*(\omega)|
\le
C(1+|\omega|),
\qquad
\omega\in\mathbb R .
\]
Since \(Q=Q_*\) Lebesgue-a.e., the two functions induce the same distribution
on \(\mathscr S(\mathbb R)\). For any integer \(N>2\),
\[
\int_{\mathbb R}
\frac{|Q(\omega)|}{(1+|\omega|)^N}\,d\omega
<\infty,
\]
after replacing \(Q\) by \(Q_*\) on a null set. Lemma~\ref{L:weighted-integrability-temperedness} gives
\[
\mathcal I_Q\in\mathscr S'(\mathbb R).
\]

\smallskip
\noindent\emph{Step 2: inverse Fourier transform on the Schwartz class.}
Let \(\psi\in\mathscr S(\mathbb R)\), and write
\[
\psi^\sharp(\omega)
=
(2\pi)^{-1}\widetilde\psi(-\omega).
\]
Since \(\psi^\sharp\in\mathscr S(\mathbb R)\),
\[
\int_{\mathbb R}
(1+|\omega|)|\psi^\sharp(\omega)|\,d\omega
<\infty .
\]
Moreover,
\[
\begin{aligned}
&\int_0^\infty
\int_{\mathbb R}
\left|
e^{i\omega y}
\bigl(1-e^{-i\omega t}\bigr)
\kappa(t)\psi^\sharp(\omega)
\right|
\,d\omega\,dt  \\
&\qquad\le
\int_0^1 t|\kappa(t)|\,dt
\int_{\mathbb R}|\omega|\,|\psi^\sharp(\omega)|\,d\omega
+
2\int_1^\infty |\kappa(t)|\,dt
\int_{\mathbb R}|\psi^\sharp(\omega)|\,d\omega
<\infty .
\end{aligned}
\]
Thus Fubini's theorem applies. Using \(Q=Q_*\) a.e.,
\[
\begin{aligned}
\mathcal F^{-1}\{\mathcal I_Q\}[\psi]
&=
\mathcal I_Q[\psi^\sharp] \\
&=
\int_{\mathbb R}Q(\omega)\psi^\sharp(\omega)\,d\omega  \\
&=
c
\int_{\mathbb R}e^{i\omega y}\psi^\sharp(\omega)\,d\omega  \\
&\qquad
+
\int_0^\infty
\left[
\int_{\mathbb R}
e^{i\omega y}
\bigl(1-e^{-i\omega t}\bigr)
\psi^\sharp(\omega)\,d\omega
\right]\kappa(t)\,dt .
\end{aligned}
\]
Fourier inversion on \(\mathscr S(\mathbb R)\) gives
\[
\int_{\mathbb R}e^{i\omega y}\psi^\sharp(\omega)\,d\omega
=
\psi(y),
\qquad
\int_{\mathbb R}e^{i\omega(y-t)}\psi^\sharp(\omega)\,d\omega
=
\psi(y-t).
\]
Therefore
\[
\mathcal F^{-1}\{\mathcal I_Q\}[\psi]
=
c\psi(y)
+
\int_0^\infty
\bigl(\psi(y)-\psi(y-t)\bigr)\kappa(t)\,dt .
\]

\smallskip
\noindent\emph{Step 3: \(\Xi^1\)-continuity.}
For \(\psi\in\Xi^1(\mathbb R;\mathbb C)\), define
\[
L[\psi]
\coloneq
c\psi(y)
+
\int_0^\infty
\bigl(\psi(y)-\psi(y-t)\bigr)\kappa(t)\,dt .
\]
This is well-defined. For \(0<t\le1\),
\[
|\psi(y)-\psi(y-t)|
\le
t\|\psi'\|_\infty,
\]
while for \(t\ge1\),
\[
|\psi(y)-\psi(y-t)|
\le
2\|\psi\|_\infty .
\]
Hence
\[
\begin{aligned}
|L[\psi]|
&\le
|c|\,\|\psi\|_\infty
+
\|\psi'\|_\infty\int_0^1 t|\kappa(t)|\,dt
+
2\|\psi\|_\infty\int_1^\infty |\kappa(t)|\,dt  \\
&\le
C_{g,V,y}\|\psi\|_{\Xi^1}
\end{aligned}
\]
for some finite constant \(C_{g,V,y}\). Thus
\[
L\in \bigl(\Xi^1(\mathbb R;\mathbb C)\bigr)'.
\]
Step 2 shows that
\[
L[\psi]
=
\mathcal F^{-1}\{\mathcal I_Q\}[\psi],
\qquad
\psi\in\mathscr S(\mathbb R).
\]
Therefore condition~\textup{(ii)} of
Theorem~\ref{T:master-fourier-identification} holds with smoothness order \(1\).

\smallskip
\noindent\emph{Step 4: passage to mixtures.}
The assumptions of the proposition give the standing hypotheses of
Theorem~\ref{T:master-fourier-identification}. Moreover, Step 1 verifies
condition~\textup{(i)}, Step 3 verifies condition~\textup{(ii)} with smoothness
order \(1\), and the assumption that \(Q\) is a Fourier-domain representer for
\(\lambda_{g,V,y}\) gives condition~\textup{(iii)}. Therefore
Theorem~\ref{T:master-fourier-identification} applies.

Let \(f\in\mathcal A_V(\mathbb R)\), and choose
\((f_n)_{n\ge1}\subset\mathscr S(\mathbb R)\) such that
\[
\|f_n-f\|_{\Xi^1}\to0 .
\]
Then
\[
\mathcal T_{g,V,y}[f]
=
\lim_{n\to\infty}
\mathcal F^{-1}\{\mathcal I_Q\}[f_n].
\]
Using Step 2 and the continuity of \(L\),
\[
\lim_{n\to\infty}
\mathcal F^{-1}\{\mathcal I_Q\}[f_n]
=
\lim_{n\to\infty}L[f_n]
=
L[f].
\]
Therefore
\[
\mathcal T_{g,V,y}[f]
=
c f(y)
+
\int_0^\infty
\bigl(f(y)-f(y-t)\bigr)\kappa(t)\,dt .
\]
Restoring \(c=c_{g,V,y}\) and \(\kappa=k_{g,V,y}\) proves the functional
identity.

Finally, assume Assumptions~\ref{A: independence} and~\ref{A: Density}. By
Proposition~\ref{P: convolution},
\[
f_Y
=
f_V*P_X
\in
\mathcal A_V(\mathbb R).
\]
Taking \(f=f_Y\) in the functional identity and applying the ratio formula
\eqref{Eq: ratio} gives, whenever \(f_Y(y)>0\),
\[
\mathbb E[g(X)\mid Y=y]
=
\frac{1}{f_Y(y)}
\left\{
c_{g,V,y}f_Y(y)
+
\int_0^\infty
\bigl(f_Y(y)-f_Y(y-t)\bigr)k_{g,V,y}(t)\,dt
\right\}.
\]
\end{proof}

\subsubsection{Computation of Table~\ref{tab:noise-families} entries}
\label{A:TableEntries}

This subsection derives the formulas reported in
Table~\ref{tab:noise-families}. Throughout, we work in dimension \(d=1\) and
take
\[
g(x)=x,
\qquad
\lambda_{g,V,y}(x)=x f_V(y-x).
\]
All posterior-mean formulas below are stated for points \(y\) such that
\(f_Y(y)>0\).

The calculations use Lemma~\ref{L:posterior-mean-symbol-1d}. After computing
the Fourier-domain representer \(\mathcal Q_{g,V,y}\), we identify its
real-space action using one of the representation results from the preceding
sections.

\begin{example}[Normal noise]
\label{E:TableNormal}
Let
\[
V\sim \mathrm{Normal}(\mu,\sigma^2).
\]
Then
\[
\varphi_V(-\omega)
=
e^{-i\mu\omega-\sigma^2\omega^2/2}.
\]
Lemma~\ref{L:posterior-mean-symbol-1d} gives
\[
\mathcal Q_{g,V,y}(\omega)
=
e^{i\omega y}
\left[
y-\mu+i\sigma^2\omega
\right].
\]
Equivalently,
\[
\mathcal Q_{g,V,y}(\omega)
=
(y-\mu)\varphi_{\delta_y}(\omega)
+
(i\omega)\varphi_{\sigma^2\delta_y}(\omega).
\]
Thus Proposition~\ref{P:measure-factorization-functional} applies with
smoothness order \(k=1\),
\[
\mu_0=(y-\mu)\delta_y,
\qquad
\mu_1=\sigma^2\delta_y.
\]
Hence
\[
\mathcal T_{g,V,y}[f_Y]
=
(y-\mu)f_Y(y)+\sigma^2 f_Y'(y).
\]
Therefore
\[
\boxed{
\mathbb E[X\mid Y=y]
=
y-\mu+\sigma^2\frac{f_Y'(y)}{f_Y(y)}.
}
\]
\end{example}

\begin{example}[Generalized Laplace noise]
\label{E:TableGeneralizedLaplace}
Let
\[
V\sim \mathrm{Generalized\ Laplace}(\mu,b,\lambda),
\qquad
\lambda>\frac12.
\]
The condition \(\lambda>1/2\) ensures \(f_V\in C_0(\mathbb R)\). The
characteristic function satisfies
\[
\varphi_V(-\omega)
=
e^{-i\mu\omega}(1+b^2\omega^2)^{-\lambda}.
\]
Lemma~\ref{L:posterior-mean-symbol-1d} gives
\[
\mathcal Q_{g,V,y}(\omega)
=
e^{i\omega y}
\left[
y-\mu+
\frac{2i\lambda b^2\omega}{1+b^2\omega^2}
\right].
\]

Let \(\nu_y\) be the finite signed measure with density
\[
d\nu_y(z)
=
\operatorname{sgn}(z-y)e^{-|z-y|/b}\,dz.
\]
A direct computation gives
\[
\varphi_{\nu_y}(\omega)
=
\frac{2ib^2\omega}{1+b^2\omega^2}e^{i\omega y}.
\]
Therefore
\[
\mathcal Q_{g,V,y}(\omega)
=
(y-\mu)\varphi_{\delta_y}(\omega)
+
\lambda\varphi_{\nu_y}(\omega).
\]
Proposition~\ref{P:measure-factorization-functional} applies with
smoothness order \(k=0\) and
\[
\mu_0=(y-\mu)\delta_y+\lambda\nu_y.
\]
Thus
\[
\mathcal T_{g,V,y}[f_Y]
=
(y-\mu)f_Y(y)
+
\lambda
\int_{\mathbb R}
\operatorname{sgn}(z-y)e^{-|z-y|/b}f_Y(z)\,dz.
\]
Therefore
\[
\boxed{
\mathbb E[X\mid Y=y]
=
y-\mu+
\frac{\lambda}{f_Y(y)}
\int_{\mathbb R}
\operatorname{sgn}(z-y)e^{-|z-y|/b}f_Y(z)\,dz.
}
\]
\end{example}

\begin{example}[Laplace noise]
\label{E:TableLaplace}
Let
\[
V\sim \mathrm{Laplace}(\mu,b).
\]
This is the case \(\lambda=1\) of
Example~\ref{E:TableGeneralizedLaplace}. Hence
\[
\mathcal T_{g,V,y}[f_Y]
=
(y-\mu)f_Y(y)
+
\int_{\mathbb R}
\operatorname{sgn}(z-y)e^{-|z-y|/b}f_Y(z)\,dz,
\]
and therefore
\[
\boxed{
\mathbb E[X\mid Y=y]
=
y-\mu+
\frac{1}{f_Y(y)}
\int_{\mathbb R}
\operatorname{sgn}(z-y)e^{-|z-y|/b}f_Y(z)\,dz.
}
\]
\end{example}

\begin{example}[Asymmetric Laplace noise]
\label{E:TableAsymmetricLaplace}
Let
\[
V\sim \mathrm{Asymmetric\ Laplace}(\mu,b_-,b_+).
\]
Then
\[
\varphi_V(-\omega)
=
\frac{e^{-i\mu\omega}}
{(1-i b_-\omega)(1+i b_+\omega)}.
\]
Lemma~\ref{L:posterior-mean-symbol-1d} gives
\[
\mathcal Q_{g,V,y}(\omega)
=
e^{i\omega y}
\left[
y-\mu+
\frac{b_-}{1-i b_-\omega}
-
\frac{b_+}{1+i b_+\omega}
\right].
\]

Let \(\nu_y\) be the finite signed measure with density
\[
d\nu_y(z)
=
e^{-(z-y)/b_-}\mathbf 1_{\{z\ge y\}}\,dz
-
e^{(z-y)/b_+}\mathbf 1_{\{z<y\}}\,dz.
\]
Then
\[
\varphi_{\nu_y}(\omega)
=
e^{i\omega y}
\left[
\frac{b_-}{1-i b_-\omega}
-
\frac{b_+}{1+i b_+\omega}
\right].
\]
Hence
\[
\mathcal Q_{g,V,y}(\omega)
=
(y-\mu)\varphi_{\delta_y}(\omega)
+
\varphi_{\nu_y}(\omega).
\]
Proposition~\ref{P:measure-factorization-functional} applies with
smoothness order \(k=0\) and
\[
\mu_0=(y-\mu)\delta_y+\nu_y.
\]
Thus
\[
\mathcal T_{g,V,y}[f_Y]
=
(y-\mu)f_Y(y)
+
\int_y^\infty e^{-(z-y)/b_-}f_Y(z)\,dz
-
\int_{-\infty}^y e^{(z-y)/b_+}f_Y(z)\,dz.
\]
Therefore
\[
\boxed{
\mathbb E[X\mid Y=y]
=
y-\mu+
\frac{1}{f_Y(y)}
\left[
\int_y^\infty e^{-(z-y)/b_-}f_Y(z)\,dz
-
\int_{-\infty}^y e^{(z-y)/b_+}f_Y(z)\,dz
\right].
}
\]
\end{example}

\begin{example}[Logistic noise]
\label{E:TableLogistic}
Let
\[
V\sim \mathrm{Logistic}(\mu,s).
\]
Then
\[
\varphi_V(-\omega)
=
e^{-i\mu\omega}\frac{\pi s\omega}{\sinh(\pi s\omega)}.
\]
Lemma~\ref{L:posterior-mean-symbol-1d} gives
\[
\mathcal Q_{g,V,y}(\omega)
=
e^{i\omega y}
\left[
y-\mu
+
i\left\{
\pi s\,\coth(\pi s\omega)-\frac1\omega
\right\}
\right].
\]

Using \cite{gradshteyn2014table}, Formula~3.912.2,
\[
\pi s\,\coth(\pi s\omega)-\frac1\omega
=
2\int_0^\infty
\frac{\sin(\omega t)}{e^{t/s}-1}\,dt.
\]
Therefore
\[
i\left\{
\pi s\,\coth(\pi s\omega)-\frac1\omega
\right\}
=
-
\int_0^\infty
\bigl(e^{-i\omega t}-e^{i\omega t}\bigr)
\frac{dt}{e^{t/s}-1}.
\]
Hence
\[
\mathcal Q_{g,V,y}(\omega)
=
(y-\mu)\varphi_{\delta_y}(\omega)
-
e^{i\omega y}
\int_0^\infty
\bigl(e^{-i\omega t}-e^{i\omega t}\bigr)
\frac{dt}{e^{t/s}-1}.
\]

The first term is handled by
Proposition~\ref{P:measure-factorization-functional}. The second term is
handled by Proposition~\ref{P:pv-kernel-functional} with
\[
c_{g,V,y}=0,
\qquad
k_{g,V,y}(t)
=
-\frac{1}{e^{t/s}-1}.
\]
Thus
\[
\mathcal T_{g,V,y}[f_Y]
=
(y-\mu)f_Y(y)
-
\int_0^\infty
\frac{f_Y(y-t)-f_Y(y+t)}{e^{t/s}-1}\,dt.
\]
Equivalently,
\[
\boxed{
\mathbb E[X\mid Y=y]
=
y-\mu+
\frac{1}{f_Y(y)}
\int_0^\infty
\frac{f_Y(y+t)-f_Y(y-t)}{e^{t/s}-1}\,dt.
}
\]
\end{example}

\begin{example}[Gumbel noise]
\label{E:TableGumbel}
Let
\[
V\sim \mathrm{Gumbel}(\mu,\beta).
\]
Then
\[
\varphi_V(-\omega)
=
e^{-i\mu\omega}\Gamma(1+i\beta\omega).
\]
Lemma~\ref{L:posterior-mean-symbol-1d} gives
\[
\mathcal Q_{g,V,y}(\omega)
=
e^{i\omega y}
\left[
y-\mu+\beta\psi_{\mathrm{DG}}(1+i\beta\omega)
\right],
\]
where
\[
\psi_{\mathrm{DG}}=\frac{\Gamma'}{\Gamma}
\]
is the digamma function. Let \(\gamma_E\) denote the Euler--Mascheroni constant.
Using \cite{abramowitz1972handbook}, Formula~6.3.22,
\[
\psi_{\mathrm{DG}}(1+z)
=
-\gamma_E
+
\int_0^\infty
\frac{1-e^{-zt}}{e^t-1}\,dt,
\qquad
\Re(z)>-1.
\]
With \(z=i\beta\omega\), and after the change of variables \(u=\beta t\),
\[
\beta\psi_{\mathrm{DG}}(1+i\beta\omega)
=
-\beta\gamma_E
+
\int_0^\infty
\frac{1-e^{-i\omega u}}{e^{u/\beta}-1}\,du.
\]
Therefore
\[
\mathcal Q_{g,V,y}(\omega)
=
e^{i\omega y}
\left[
y-\mu-\beta\gamma_E
+
\int_0^\infty
\bigl(1-e^{-i\omega u}\bigr)
\frac{du}{e^{u/\beta}-1}
\right].
\]

Proposition~\ref{P:onesided-kernel-functional} applies with
\[
c_{g,V,y}=y-\mu-\beta\gamma_E,
\qquad
k_{g,V,y}(u)=\frac{1}{e^{u/\beta}-1}.
\]
It follows that
\[
\mathcal T_{g,V,y}[f_Y]
=
(y-\mu-\beta\gamma_E)f_Y(y)
+
\int_0^\infty
\frac{f_Y(y)-f_Y(y-u)}{e^{u/\beta}-1}\,du.
\]
Therefore
\[
\boxed{
\mathbb E[X\mid Y=y]
=
y-\mu-\beta\gamma_E
+
\frac{1}{f_Y(y)}
\int_0^\infty
\frac{f_Y(y)-f_Y(y-u)}{e^{u/\beta}-1}\,du.
}
\]
\end{example}

\begin{example}[Cauchy noise]
\label{E:TableCauchy}
Let
\[
V\sim \mathrm{Cauchy}(\mu,\gamma),
\qquad
\gamma>0.
\]
We again take \(g(x)=x\). In this case the universal symbol lemma does not
apply. Indeed,
\[
\lambda_{g,V,y}(x)
=
x f_V(y-x)
=
\frac{\gamma x}
{\pi\{(y-\mu-x)^2+\gamma^2\}}
\]
belongs to \(C_0(\mathbb R;\mathbb C)\), but
\[
x f_V(y-x)
\sim
\frac{\gamma}{\pi x}
\qquad
(|x|\to\infty).
\]
Thus
\[
\lambda_{g,V,y}\notin L^1(\mathbb R),
\qquad
\lambda_{g,V,y}\notin \Xi^0(\mathbb R;\mathbb C),
\]
so Lemma~\ref{L:posterior-mean-symbol-1d} cannot be invoked.

We verify the Fourier-domain representer identity directly. Set
\[
a\coloneq y-\mu,
\qquad
h_y(x)\coloneq f_V(y-x).
\]
Then \(h_y\in L^1(\mathbb R)\), and
\[
\widetilde h_y(\omega)
=
e^{i\omega y}\varphi_V(-\omega)
=
e^{i\omega(y-\mu)-\gamma|\omega|}
=
e^{ia\omega-\gamma|\omega|}.
\]
Although \(\lambda_{g,V,y}(x)=x h_y(x)\) is not integrable, it is bounded and
belongs to \(C_0(\mathbb R;\mathbb C)\). Hence it defines a tempered
distribution by integration against Schwartz functions.

Let \(\psi\in\mathscr S(\mathbb R)\). Using the distributional identity
\[
\mathcal F\{x h_y\}
=
-i\frac{d}{d\omega}\widetilde h_y,
\]
we obtain
\[
\mathcal F\{\mathcal I_{\lambda_{g,V,y}}\}[\psi]
=
i\int_{\mathbb R}
e^{ia\omega-\gamma|\omega|}
\psi'(\omega)\,d\omega .
\]
The function \(\omega\mapsto e^{ia\omega-\gamma|\omega|}\) is absolutely
continuous on \((-\infty,0)\) and \((0,\infty)\), and is continuous at \(0\).
Thus integration by parts gives no boundary term at the origin, and
\[
\begin{aligned}
i\int_{\mathbb R}
e^{ia\omega-\gamma|\omega|}
\psi'(\omega)\,d\omega
&=
-i\int_{\mathbb R}
\frac{d}{d\omega}
\left(e^{ia\omega-\gamma|\omega|}\right)
\psi(\omega)\,d\omega  \\
&=
\int_{\mathbb R}
e^{ia\omega-\gamma|\omega|}
\left(a+i\gamma\operatorname{sgn}(\omega)\right)
\psi(\omega)\,d\omega .
\end{aligned}
\]
Therefore
\[
\mathcal F\{\mathcal I_{\lambda_{g,V,y}}\}
=
\mathcal I_{\omega\mapsto
e^{ia\omega-\gamma|\omega|}
\left(a+i\gamma\operatorname{sgn}(\omega)\right)}
\qquad
\text{in } \mathscr S'(\mathbb R).
\]

Since
\[
\varphi_V(-\omega)
=
e^{-i\mu\omega-\gamma|\omega|},
\]
the function
\[
\mathcal Q_{g,V,y}(\omega)
\coloneq
e^{i\omega y}
\left[
y-\mu+i\gamma\operatorname{sgn}(\omega)
\right]
\]
satisfies
\[
\mathcal Q_{g,V,y}(\omega)\varphi_V(-\omega)
=
e^{ia\omega-\gamma|\omega|}
\left(a+i\gamma\operatorname{sgn}(\omega)\right)
\]
for Lebesgue-a.e. \(\omega\). Hence
\(\mathcal Q_{g,V,y}\) is a Fourier-domain representer for
\(\lambda_{g,V,y}\).

We now decompose the representer as
\[
\mathcal Q_{g,V,y}(\omega)
=
(y-\mu)\varphi_{\delta_y}(\omega)
-
i\pi\left(-\frac{\gamma}{\pi}\right)
e^{i\omega y}\operatorname{sgn}(\omega).
\]
The first term is a measure-factorization term with
\[
\mu_0=(y-\mu)\delta_y.
\]
The second term is the Hilbert-transform term in
Proposition~\ref{P:pv-kernel-functional}, with
\[
c_{g,V,y}=-\frac{\gamma}{\pi},
\qquad
k_{g,V,y}\equiv0.
\]
Combining Proposition~\ref{P:measure-factorization-functional} with
Proposition~\ref{P:pv-kernel-functional} gives
\[
\mathcal T_{g,V,y}[f_Y]
=
(y-\mu)f_Y(y)
-
\gamma\,\mathcal H[f_Y](y).
\]
Therefore, whenever \(f_Y(y)>0\),
\[
\boxed{
\mathbb E[X\mid Y=y]
=
y-\mu
-
\gamma\frac{\mathcal H[f_Y](y)}{f_Y(y)}.
}
\]
\end{example}

\begin{example}[Hyperbolic secant noise]
\label{E:TableHyperbolicSecant}
Let
\[
V\sim \mathrm{Hyperbolic\ Secant}(\mu,s),
\qquad
s>0.
\]
Then
\[
\varphi_V(-\omega)
=
e^{-i\mu\omega}\operatorname{sech}(s\omega).
\]
Lemma~\ref{L:posterior-mean-symbol-1d} gives
\[
\mathcal Q_{g,V,y}(\omega)
=
e^{i\omega y}
\left[
y-\mu+i s\,\tanh(s\omega)
\right].
\]

Using \cite{gradshteyn2014table}, Formula~3.981.1,
\[
\tanh(s\omega)
=
\frac{1}{s}
\int_0^\infty
\frac{\sin(\omega t)}
{\sinh\!\left(\frac{\pi t}{2s}\right)}\,dt.
\]
Hence
\[
is\,\tanh(s\omega)
=
-
\int_0^\infty
\frac{e^{-i\omega t}-e^{i\omega t}}
{2\sinh\!\left(\frac{\pi t}{2s}\right)}\,dt.
\]
Therefore
\[
\mathcal Q_{g,V,y}(\omega)
=
(y-\mu)\varphi_{\delta_y}(\omega)
-
e^{i\omega y}
\int_0^\infty
\frac{e^{-i\omega t}-e^{i\omega t}}
{2\sinh\!\left(\frac{\pi t}{2s}\right)}\,dt.
\]

The first term is handled by
Proposition~\ref{P:measure-factorization-functional}. The second term is
handled by Proposition~\ref{P:pv-kernel-functional} with
\[
c_{g,V,y}=0,
\qquad
k_{g,V,y}(t)
=
-\frac{1}{2\sinh\!\left(\frac{\pi t}{2s}\right)}.
\]
Thus
\[
\mathcal T_{g,V,y}[f_Y]
=
(y-\mu)f_Y(y)
-
\int_0^\infty
\frac{f_Y(y-t)-f_Y(y+t)}
{2\sinh\!\left(\frac{\pi t}{2s}\right)}\,dt.
\]
Therefore
\[
\boxed{
\mathbb E[X\mid Y=y]
=
y-\mu
+
\frac{1}{f_Y(y)}
\int_0^\infty
\frac{f_Y(y+t)-f_Y(y-t)}
{2\sinh\!\left(\frac{\pi t}{2s}\right)}\,dt .
}
\]
\end{example}

\begin{example}[Gamma noise]
\label{E:TableGamma}
Let
\[
V\sim \mathrm{Gamma}(\alpha,\theta),
\qquad
\alpha>1,\quad \theta>0,
\]
and extend its density by zero off \((0,\infty)\). The condition
\(\alpha>1\) ensures that the extended density belongs to
\(C_0(\mathbb R)\). Since
\[
\varphi_V(-\omega)
=
(1+i\theta\omega)^{-\alpha},
\]
Lemma~\ref{L:posterior-mean-symbol-1d} gives
\[
\mathcal Q_{g,V,y}(\omega)
=
e^{i\omega y}
\left[
y-\frac{\alpha\theta}{1+i\theta\omega}
\right].
\]

Let \(\nu_y\) be the finite measure with density
\[
d\nu_y(z)
=
\frac{1}{\theta}
e^{(z-y)/\theta}
\mathbf 1_{\{z<y\}}\,dz.
\]
A direct computation gives
\[
\varphi_{\nu_y}(\omega)
=
\frac{e^{i\omega y}}{1+i\theta\omega}.
\]
Therefore
\[
\mathcal Q_{g,V,y}(\omega)
=
y\,\varphi_{\delta_y}(\omega)
-
\alpha\theta\,\varphi_{\nu_y}(\omega).
\]
Proposition~\ref{P:measure-factorization-functional} applies with
smoothness order \(k=0\) and
\[
\mu_0
=
y\delta_y-\alpha\theta\nu_y.
\]
Hence
\[
\mathcal T_{g,V,y}[f_Y]
=
y f_Y(y)
-
\alpha
\int_{-\infty}^y
e^{(z-y)/\theta}f_Y(z)\,dz.
\]
Therefore
\[
\boxed{
\mathbb E[X\mid Y=y]
=
y
-
\frac{\alpha}{f_Y(y)}
\int_{-\infty}^y
e^{(z-y)/\theta}f_Y(z)\,dz .
}
\]
\end{example}

\begin{example}[Noncentral chi-square noise]
\label{E:TableNoncentralChiSquare}
Let
\[
V\sim \mathrm{Noncentral}\ \chi^2_\nu(\delta),
\qquad
\nu>2,\quad \delta\ge0,
\]
and extend its density by zero off \((0,\infty)\). The condition
\(\nu>2\) ensures that the extended density belongs to \(C_0(\mathbb R)\).
Since
\[
\varphi_V(-\omega)
=
(1+2i\omega)^{-\nu/2}
\exp\!\left(
-\frac{i\delta\omega}{1+2i\omega}
\right),
\]
we have
\[
\frac{\frac{d}{d\omega}\varphi_V(-\omega)}
{\varphi_V(-\omega)}
=
-\frac{i\nu}{1+2i\omega}
-
\frac{i\delta}{(1+2i\omega)^2}.
\]
Lemma~\ref{L:posterior-mean-symbol-1d} therefore gives
\[
\mathcal Q_{g,V,y}(\omega)
=
e^{i\omega y}
\left[
y
-
\frac{\nu}{1+2i\omega}
-
\frac{\delta}{(1+2i\omega)^2}
\right].
\]

Let \(\eta_{1,y}\) and \(\eta_{2,y}\) be the finite measures with densities
\[
d\eta_{1,y}(z)
=
\frac12 e^{(z-y)/2}\mathbf 1_{\{z<y\}}\,dz,
\]
and
\[
d\eta_{2,y}(z)
=
\frac{y-z}{4}e^{(z-y)/2}\mathbf 1_{\{z<y\}}\,dz.
\]
A direct computation gives
\[
\varphi_{\eta_{1,y}}(\omega)
=
\frac{e^{i\omega y}}{1+2i\omega},
\qquad
\varphi_{\eta_{2,y}}(\omega)
=
\frac{e^{i\omega y}}{(1+2i\omega)^2}.
\]
Hence
\[
\mathcal Q_{g,V,y}(\omega)
=
y\,\varphi_{\delta_y}(\omega)
-
\nu\,\varphi_{\eta_{1,y}}(\omega)
-
\delta\,\varphi_{\eta_{2,y}}(\omega).
\]
Proposition~\ref{P:measure-factorization-functional} applies with
smoothness order \(k=0\) and
\[
\mu_0
=
y\delta_y-\nu\eta_{1,y}-\delta\eta_{2,y}.
\]
Thus
\[
\mathcal T_{g,V,y}[f_Y]
=
y f_Y(y)
-
\frac{\nu}{2}
\int_{-\infty}^y
e^{(z-y)/2}f_Y(z)\,dz
-
\frac{\delta}{4}
\int_{-\infty}^y
(y-z)e^{(z-y)/2}f_Y(z)\,dz.
\]
Therefore
\[
\boxed{
\mathbb E[X\mid Y=y]
=
y
-
\frac{1}{f_Y(y)}
\left[
\frac{\nu}{2}
\int_{-\infty}^y
e^{(z-y)/2}f_Y(z)\,dz
+
\frac{\delta}{4}
\int_{-\infty}^y
(y-z)e^{(z-y)/2}f_Y(z)\,dz
\right].
}
\]
\end{example}

\begin{example}[Inverse Gaussian noise]
\label{E:TableInverseGaussian}
Let
\[
V\sim \mathrm{Inverse\ Gaussian}(\mu,\lambda),
\qquad
\mu>0,\quad \lambda>0.
\]
Then
\[
\varphi_V(-\omega)
=
\exp\!\left[
\frac{\lambda}{\mu}
\left(
1-\sqrt{1+\frac{2i\mu^2\omega}{\lambda}}
\right)
\right].
\]
Moreover,
\[
\frac{\frac{d}{d\omega}\varphi_V(-\omega)}
{\varphi_V(-\omega)}
=
-\frac{i\mu}
{\sqrt{1+\frac{2i\mu^2\omega}{\lambda}}}.
\]
Lemma~\ref{L:posterior-mean-symbol-1d} gives
\[
\mathcal Q_{g,V,y}(\omega)
=
e^{i\omega y}
\left[
y
-
\frac{\mu}
{\sqrt{1+\frac{2i\mu^2\omega}{\lambda}}}
\right].
\]

Let \(\nu_y\) be the finite measure with density
\[
d\nu_y(z)
=
\sqrt{\frac{\lambda}{2\pi\mu^2}}\,
(y-z)^{-1/2}
\exp\!\left(
-\frac{\lambda(y-z)}{2\mu^2}
\right)
\mathbf 1_{\{z<y\}}\,dz.
\]
A direct computation gives
\[
\varphi_{\nu_y}(\omega)
=
\frac{e^{i\omega y}}
{\sqrt{1+\frac{2i\mu^2\omega}{\lambda}}}.
\]
Therefore
\[
\mathcal Q_{g,V,y}(\omega)
=
y\,\varphi_{\delta_y}(\omega)
-
\mu\,\varphi_{\nu_y}(\omega).
\]
Proposition~\ref{P:measure-factorization-functional} applies with
smoothness order \(k=0\) and
\[
\mu_0
=
y\delta_y-\mu\nu_y.
\]
Hence
\[
\mathcal T_{g,V,y}[f_Y]
=
y f_Y(y)
-
\sqrt{\frac{\lambda}{2\pi}}
\int_{-\infty}^y
(y-z)^{-1/2}
\exp\!\left(
-\frac{\lambda(y-z)}{2\mu^2}
\right)
f_Y(z)\,dz.
\]
Therefore
\[
\boxed{
\mathbb E[X\mid Y=y]
=
y
-
\frac{1}{f_Y(y)}
\sqrt{\frac{\lambda}{2\pi}}
\int_{-\infty}^y
(y-z)^{-1/2}
\exp\!\left(
-\frac{\lambda(y-z)}{2\mu^2}
\right)
f_Y(z)\,dz .
}
\]
\end{example}

\subsubsection{Functionals under the conventional Laplace mechanism}
\label{A:LaplaceMechanism}

This section records the formulas used for the conventional Laplace mechanism
from Section~\ref{SS: Laplace}. The noise density is the product Laplace density
\[
f_V(v)
=
\prod_{j=1}^d
\frac{1}{2b}\exp\!\left(-\frac{|v_j|}{b}\right)
=
\frac{1}{(2b)^d}
\exp\!\left(-\frac{\|v\|_1}{b}\right),
\qquad
v\in\mathbb R^d.
\]
Its characteristic function is
\[
\varphi_V(\omega)
=
\prod_{j=1}^d
\frac{1}{1+b^2\omega_j^2}.
\]

We first define the kernels that appear in the formulas. For a measurable
function \(f:\mathbb R^d\to\mathbb C\), set
\[
(\mathsf D_{b,j}f)(y)
\coloneq
\int_{\mathbb R}
\operatorname{sgn}(u)e^{-|u|/b}
f(y+u e_j)\,du,
\qquad
j=1,\ldots,d.
\]
Let
\[
(\mathsf D_b f)(y)
\coloneq
\bigl(
(\mathsf D_{b,1}f)(y),\ldots,
(\mathsf D_{b,d}f)(y)
\bigr)^\top .
\]
Next define the matrix-valued operator \(\mathsf G_b\) entrywise by
\[
(\mathsf G_{b,jj}f)(y)
\coloneq
\int_{\mathbb R}
(2|u|-b)e^{-|u|/b}
f(y+u e_j)\,du,
\]
and, for \(j\neq k\),
\[
(\mathsf G_{b,jk}f)(y)
\coloneq
\int_{\mathbb R}\int_{\mathbb R}
\operatorname{sgn}(u)\operatorname{sgn}(v)
e^{-(|u|+|v|)/b}
f(y+u e_j+v e_k)\,du\,dv .
\]
Thus
\[
(\mathsf G_b f)(y)
=
\bigl[
(\mathsf G_{b,jk}f)(y)
\bigr]_{j,k=1}^d .
\]

\paragraph{Loss functionals in the multi-dimensional Laplace mechanism}
\label{A:LaplaceMechanismLosses}

\begin{example}[Posterior mean under product Laplace noise]
\label{E:LaplaceMechanismMean}
Let
\[
g(x)=x.
\]
For the product Laplace law, the Fourier-domain representer factorizes
coordinatewise. For the \(j\)-th coordinate,
\[
\mathcal Q_{j,V,y}(\omega)
=
e^{i\omega^\top y}
\left[
y_j+
\frac{2ib^2\omega_j}{1+b^2\omega_j^2}
\right].
\]
Let \(\nu_{b,y}^{(j)}\) be the finite signed measure defined by
\[
\int_{\mathbb R^d} f(z)\,d\nu_{b,y}^{(j)}(z)
=
\int_{\mathbb R}
\operatorname{sgn}(u)e^{-|u|/b}
f(y+u e_j)\,du .
\]
A direct computation gives
\[
\varphi_{\nu_{b,y}^{(j)}}(\omega)
=
e^{i\omega^\top y}
\frac{2ib^2\omega_j}{1+b^2\omega_j^2}.
\]
Therefore
\[
\mathcal Q_{j,V,y}(\omega)
=
y_j\varphi_{\delta_y}(\omega)
+
\varphi_{\nu_{b,y}^{(j)}}(\omega).
\]
Applying Proposition~\ref{P:measure-factorization-functional} entrywise gives
\[
\mathcal T_{j,V,y}[f]
=
y_j f(y)+(\mathsf D_{b,j}f)(y),
\qquad
f\in\mathcal A_V(\mathbb R^d).
\]
Hence, whenever \(f_Y(y)>0\),
\[
\boxed{
\mathbb E[X\mid Y=y]
=
y+
\frac{(\mathsf D_b f_Y)(y)}{f_Y(y)}.
}
\]
Equivalently, for \(j=1,\ldots,d\),
\[
\mathbb E[X_j\mid Y=y]
=
y_j+
\frac{
\displaystyle
\int_{\mathbb R}
\operatorname{sgn}(u)e^{-|u|/b}
f_Y(y+u e_j)\,du
}{
f_Y(y)
}.
\]
\end{example}

\begin{example}[Posterior second moment and covariance under product Laplace noise]
\label{E:LaplaceMechanismCovariance}
Let
\[
g(x)=xx^\top .
\]
We apply the scalar representation entrywise to
\(g_{jk}(x)=x_jx_k\). Define
\[
a_j(\omega)
\coloneq
\frac{2ib^2\omega_j}{1+b^2\omega_j^2},
\qquad
a(\omega)
\coloneq
\bigl(a_1(\omega),\ldots,a_d(\omega)\bigr)^\top .
\]
The Fourier-domain representer for the second moment is
\[
\mathcal Q_{2,V,y}(\omega)
=
e^{i\omega^\top y}
\left[
yy^\top
+
y a(\omega)^\top
+
a(\omega)y^\top
+
H(\omega)
\right],
\]
where
\[
H_{jj}(\omega)
=
\frac{2b^2(1-3b^2\omega_j^2)}
{(1+b^2\omega_j^2)^2},
\qquad
H_{jk}(\omega)
=
a_j(\omega)a_k(\omega)
\quad (j\neq k).
\]

The first-order terms are represented by the signed measures
\(\nu_{b,y}^{(j)}\) from Example~\ref{E:LaplaceMechanismMean}. For the
second-order terms, define signed measures \(\eta_{b,y}^{jk}\) by
\[
\int_{\mathbb R^d} f(z)\,d\eta_{b,y}^{jj}(z)
=
\int_{\mathbb R}
(2|u|-b)e^{-|u|/b}
f(y+u e_j)\,du,
\]
and, for \(j\neq k\),
\[
\int_{\mathbb R^d} f(z)\,d\eta_{b,y}^{jk}(z)
=
\int_{\mathbb R}\int_{\mathbb R}
\operatorname{sgn}(u)\operatorname{sgn}(v)
e^{-(|u|+|v|)/b}
f(y+u e_j+v e_k)\,du\,dv .
\]
These measures satisfy
\[
\varphi_{\eta_{b,y}^{jk}}(\omega)
=
e^{i\omega^\top y}H_{jk}(\omega),
\qquad
1\le j,k\le d.
\]
Thus Proposition~\ref{P:measure-factorization-functional} applies entrywise.
For every \(f\in\mathcal A_V(\mathbb R^d)\),
\[
\mathcal T_{2,V,y}[f]
=
yy^\top f(y)
+
y(\mathsf D_b f)(y)^\top
+
(\mathsf D_b f)(y)y^\top
+
(\mathsf G_b f)(y).
\]
Taking \(f=f_Y\) and dividing by \(f_Y(y)\) gives the posterior second moment:
\[
\mathbb E[XX^\top\mid Y=y]
=
yy^\top
+
y\left(\frac{(\mathsf D_b f_Y)(y)}{f_Y(y)}\right)^\top
+
\left(\frac{(\mathsf D_b f_Y)(y)}{f_Y(y)}\right)y^\top
+
\frac{(\mathsf G_b f_Y)(y)}{f_Y(y)}.
\]
Combining this formula with
Example~\ref{E:LaplaceMechanismMean}, we obtain the covariance formula
\[
\boxed{
\operatorname{Cov}(X\mid Y=y)
=
\frac{(\mathsf G_b f_Y)(y)}{f_Y(y)}
-
\frac{
(\mathsf D_b f_Y)(y)(\mathsf D_b f_Y)(y)^\top
}{
f_Y(y)^2
}.
}
\]
Equivalently, coordinatewise,
\[
\operatorname{Cov}(X_j,X_k\mid Y=y)
=
\frac{(\mathsf G_{b,jk}f_Y)(y)}{f_Y(y)}
-
\frac{
(\mathsf D_{b,j}f_Y)(y)(\mathsf D_{b,k}f_Y)(y)
}{
f_Y(y)^2
}.
\]
\end{example}

\begin{example}[Posterior moment generating function under product Laplace noise]
\label{E:LaplaceMechanismMGF}
Fix \(t\in\mathbb R^d\) with
\[
\|t\|_\infty<\frac1b,
\]
and let
\[
g_t(x)=e^{t^\top x}.
\]
The restriction on \(t\) ensures that all kernels below are finite signed
measures. Since the coordinates of \(V\) are independent Laplace random
variables, the Fourier-domain representer factorizes coordinatewise:
\[
\mathcal Q_{g_t,V,y}(\omega)
=
e^{t^\top y}e^{i\omega^\top y}
\prod_{j=1}^d
\frac{1+b^2\omega_j^2}
{1-b^2(t_j+i\omega_j)^2}.
\]

For \(r\in(-1/b,1/b)\), define the finite signed measure \(\rho_{b,r}\) on
\(\mathbb R\) by
\[
d\rho_{b,r}(u)
=
e^{ru}
\left[
r\operatorname{sgn}(u)-\frac{b r^2}{2}
\right]
e^{-|u|/b}\,du .
\]
A direct calculation gives
\[
\int_{\mathbb R}e^{i\omega u}\,d\rho_{b,r}(u)
=
\frac{1+b^2\omega^2}{1-b^2(r+i\omega)^2}-1.
\]
Therefore, for each coordinate \(j\),
\[
\frac{1+b^2\omega_j^2}
{1-b^2(t_j+i\omega_j)^2}
=
1+
\int_{\mathbb R}e^{i\omega_j u}\,d\rho_{b,t_j}(u).
\]

For a measurable function \(f:\mathbb R^d\to\mathbb C\), define
\[
(\mathsf M_{b,t}f)(y)
\coloneq
\sum_{S\subseteq\{1,\ldots,d\}}
\int_{\mathbb R^{|S|}}
f\!\left(y+\sum_{j\in S}u_j e_j\right)
\prod_{j\in S}d\rho_{b,t_j}(u_j),
\]
where the term for \(S=\varnothing\) is interpreted as \(f(y)\). Expanding the
coordinatewise product and applying
Proposition~\ref{P:measure-factorization-functional} gives
\[
\mathcal T_{g_t,V,y}[f]
=
e^{t^\top y}(\mathsf M_{b,t}f)(y),
\qquad
f\in\mathcal A_V(\mathbb R^d).
\]
Hence, whenever \(f_Y(y)>0\),
\[
\boxed{
\mathbb E[e^{t^\top X}\mid Y=y]
=
e^{t^\top y}
\frac{(\mathsf M_{b,t}f_Y)(y)}{f_Y(y)} .
}
\]

Equivalently, in expanded form,
\[
\mathbb E[e^{t^\top X}\mid Y=y]
=
\frac{e^{t^\top y}}{f_Y(y)}
\sum_{S\subseteq\{1,\ldots,d\}}
\int_{\mathbb R^{|S|}}
f_Y\!\left(y+\sum_{j\in S}u_j e_j\right)
\prod_{j\in S}d\rho_{b,t_j}(u_j).
\]
\end{example}

\paragraph{Loss functionals in the one-dimensional Laplace mechanism}
\label{A:LaplaceMechanismLossesOneDim}

We now specialize to \(d=1\) and
\[
V\sim \mathrm{Laplace}(0,b).
\]

For the  loss formulas below, it is useful to introduce
\[
(\mathsf K_b f)(y)
\coloneq
\int_{\mathbb R}e^{-|u|/b}f(y+u)\,du,
\]
\[
(\mathsf K_b^+ f)(y)
\coloneq
\int_0^\infty e^{-u/b}f(y+u)\,du,
\qquad
(\mathsf K_b^- f)(y)
\coloneq
\int_{-\infty}^0 e^{u/b}f(y+u)\,du.
\]
For \(a,y\in\mathbb R\), define the interval correction
\[
(\mathsf J_{b,a}f)(y)
\coloneq
\int_{\min(0,a-y)}^{\max(0,a-y)}
e^{-|u|/b}f(y+u)\,du,
\]
and the boundary correction
\[
(\mathsf R_{b,a}f)(y)
\coloneq
b e^{-|y-a|/b}f(a).
\]

\begin{example}[Posterior distribution function under one-dimensional Laplace noise]
\label{E:LaplaceMechanismCDF}
Assume \(d=1\), and let
\[
g_a(x)=\mathbf 1_{\{x\le a\}},
\qquad
a\in\mathbb R .
\]
This functional is easier to compute directly than through Fourier inversion.
Define
\[
A(a)
\coloneq
\int_{x\le a} f_V(a-x)\,dP_X(x),
\qquad
B(a)
\coloneq
\int_{x>a} f_V(a-x)\,dP_X(x).
\]
Then
\[
f_Y(a)=A(a)+B(a).
\]

The Laplace density is not differentiable at the origin, so we use the right
derivative
\[
D_+f_Y(a)
\coloneq
\lim_{h\downarrow0}
\frac{f_Y(a+h)-f_Y(a)}{h}.
\]
This derivative exists. Indeed, the Laplace density is Lipschitz, and its right
derivative at \(a-x\) is
\[
D_+ f_V(a-x)
=
\begin{cases}
-\dfrac1b f_V(a-x), & x\le a, \\[0.8em]
\phantom{-}\dfrac1b f_V(a-x), & x>a .
\end{cases}
\]
Dominated convergence therefore gives
\[
D_+f_Y(a)
=
-\frac1b A(a)+\frac1b B(a).
\]
Solving the two equations
\[
f_Y(a)=A(a)+B(a),
\qquad
D_+f_Y(a)=-\frac1b A(a)+\frac1b B(a),
\]
yields
\[
A(a)
=
\frac12\{f_Y(a)-bD_+f_Y(a)\},
\qquad
B(a)
=
\frac12\{f_Y(a)+bD_+f_Y(a)\}.
\]

Define the lower and upper Laplace tail functionals
\[
L_b^-(a)
\coloneq
\frac12\{f_Y(a)-bD_+f_Y(a)\},
\qquad
L_b^+(a)
\coloneq
\frac12\{f_Y(a)+bD_+f_Y(a)\}.
\]
Thus \(L_b^-(a)=A(a)\) and \(L_b^+(a)=B(a)\).

Now fix \(y\) with \(f_Y(y)>0\). If \(a\le y\), then for \(x\le a\),
\[
f_V(y-x)
=
e^{-(y-a)/b}f_V(a-x).
\]
Hence
\[
\int_{x\le a} f_V(y-x)\,dP_X(x)
=
e^{-(y-a)/b}L_b^-(a).
\]
If \(a>y\), it is shorter to compute the upper tail. For \(x>a\),
\[
f_V(y-x)
=
e^{-(a-y)/b}f_V(a-x),
\]
so
\[
\int_{x>a} f_V(y-x)\,dP_X(x)
=
e^{-(a-y)/b}L_b^+(a).
\]
Therefore
\[
\boxed{
\mathbb P(X\le a\mid Y=y)
=
\begin{cases}
\displaystyle
\dfrac{e^{-(y-a)/b}L_b^-(a)}{f_Y(y)},
& a\le y, \\[1.2em]
\displaystyle
1-
\dfrac{e^{-(a-y)/b}L_b^+(a)}{f_Y(y)},
& a>y .
\end{cases}
}
\]
Equivalently,
\[
L_b^-(a)=\frac12\{f_Y(a)-bD_+f_Y(a)\},
\qquad
L_b^+(a)=\frac12\{f_Y(a)+bD_+f_Y(a)\}.
\]
\end{example}

\begin{example}[Posterior squared risk under Laplace noise]
\label{E:LaplaceMechanismSquaredRisk}
Fix \(a\in\mathbb R\). Expanding the square gives
\[
\mathbb E[(X-a)^2\mid Y=y]
=
\mathbb E[X^2\mid Y=y]
-
2a\,\mathbb E[X\mid Y=y]
+
a^2.
\]
Substituting the posterior first and second moment formulas yields
\[
\boxed{
\mathbb E[(X-a)^2\mid Y=y]
=
(y-a)^2
+
2(y-a)\frac{(\mathsf D_b f_Y)(y)}{f_Y(y)}
+
\frac{(\mathsf G_b f_Y)(y)}{f_Y(y)}.
}
\]
Equivalently,
\[
\begin{aligned}
\mathbb E[(X-a)^2\mid Y=y]
&=
(y-a)^2  \\
&\quad+
\frac{2(y-a)}{f_Y(y)}
\int_{\mathbb R}
\operatorname{sgn}(u)e^{-|u|/b}f_Y(y+u)\,du  \\
&\quad+
\frac{1}{f_Y(y)}
\int_{\mathbb R}
(2|u|-b)e^{-|u|/b}f_Y(y+u)\,du .
\end{aligned}
\]
\end{example}

\begin{example}[Posterior absolute risk under Laplace noise]
\label{E:LaplaceMechanismAbsoluteRisk}
Fix \(a\in\mathbb R\), and let
\[
F_y(t)
\coloneq
\mathbb P(X\le t\mid Y=y).
\]
The identity
\[
\mathbb E[|X-a|\mid Y=y]
=
\int_{-\infty}^{a}F_y(t)\,dt
+
\int_a^\infty \{1-F_y(t)\}\,dt
\]
combined with the posterior distribution formula from
Example~\ref{E:LaplaceMechanismCDF} gives
\[
\mathbb E[|X-a|\mid Y=y]
=
|y-a|
+
\frac{1}{f_Y(y)}
\left[
\int_{\mathbb R}
\operatorname{sgn}(u)\operatorname{sgn}(y+u-a)
e^{-|u|/b}f_Y(y+u)\,du
-
(\mathsf R_{b,a}f_Y)(y)
\right].
\]
Since
\[
\operatorname{sgn}(u)\operatorname{sgn}(u-(a-y))
=
1
\]
outside the interval between \(0\) and \(a-y\), and equals \(-1\) inside that
interval, this becomes
\[
\boxed{
\mathbb E[|X-a|\mid Y=y]
=
|y-a|
+
\frac{
(\mathsf K_b f_Y)(y)
-
2(\mathsf J_{b,a}f_Y)(y)
-
(\mathsf R_{b,a}f_Y)(y)
}{
f_Y(y)
}.
}
\]
In expanded form,
\[
\begin{aligned}
\mathbb E[|X-a|\mid Y=y]
&=
|y-a|
+
\frac{1}{f_Y(y)}
\bigg[
\int_{\mathbb R} e^{-|u|/b}f_Y(y+u)\,du  \\
&\qquad
-
2
\int_{\min(0,a-y)}^{\max(0,a-y)}
e^{-|u|/b}f_Y(y+u)\,du
-
b e^{-|y-a|/b}f_Y(a)
\bigg].
\end{aligned}
\]
\end{example}

\begin{example}[Posterior hinge loss under Laplace noise]
\label{E:LaplaceMechanismHingeLoss}
Fix \(a\in\mathbb R\). Since
\[
(x-a)_+
=
\frac12|x-a|+\frac12(x-a),
\]
we have
\[
\mathbb E[(X-a)_+\mid Y=y]
=
\frac12\mathbb E[|X-a|\mid Y=y]
+
\frac12\{\mathbb E[X\mid Y=y]-a\}.
\]
Substituting the posterior mean and absolute-risk formulas gives
\[
\boxed{
\mathbb E[(X-a)_+\mid Y=y]
=
(y-a)_+
+
\frac{
(\mathsf K_b^+ f_Y)(y)
-
(\mathsf J_{b,a}f_Y)(y)
-
\frac12(\mathsf R_{b,a}f_Y)(y)
}{
f_Y(y)
}.
}
\]
Equivalently,
\[
\begin{aligned}
\mathbb E[(X-a)_+\mid Y=y]
&=
(y-a)_+
+
\frac{1}{f_Y(y)}
\int_0^\infty e^{-u/b}f_Y(y+u)\,du  \\
&\quad
-
\frac{1}{f_Y(y)}
\int_{\min(0,a-y)}^{\max(0,a-y)}
e^{-|u|/b}f_Y(y+u)\,du  \\
&\quad
-
\frac{b e^{-|y-a|/b}f_Y(a)}{2f_Y(y)}.
\end{aligned}
\]
\end{example}

\begin{example}[Posterior pinball loss under Laplace noise]
\label{E:LaplaceMechanismPinballLoss}
Fix \(a\in\mathbb R\) and \(\tau\in(0,1)\). Let
\[
\rho_\tau(z)
=
\tau z_+
+
(1-\tau)(-z)_+ .
\]
Then
\[
\rho_\tau(X-a)
=
\tau(X-a)_+
+
(1-\tau)(a-X)_+.
\]
Using the hinge formula for both \((X-a)_+\) and \((a-X)_+\) gives
\[
\boxed{
\begin{aligned}
\mathbb E[\rho_\tau(X-a)\mid Y=y]
&=
\tau(y-a)_+
+
(1-\tau)(a-y)_+  \\
&\quad+
\frac{
\tau(\mathsf K_b^+ f_Y)(y)
+
(1-\tau)(\mathsf K_b^- f_Y)(y)
-
(\mathsf J_{b,a}f_Y)(y)
-
\frac12(\mathsf R_{b,a}f_Y)(y)
}{
f_Y(y)
}.
\end{aligned}
}
\]
In expanded form,
\[
\begin{aligned}
\mathbb E[\rho_\tau(X-a)\mid Y=y]
&=
\tau(y-a)_+
+
(1-\tau)(a-y)_+  \\
&\quad+
\frac{\tau}{f_Y(y)}
\int_0^\infty e^{-u/b}f_Y(y+u)\,du  \\
&\quad+
\frac{1-\tau}{f_Y(y)}
\int_{-\infty}^0 e^{u/b}f_Y(y+u)\,du  \\
&\quad-
\frac{1}{f_Y(y)}
\int_{\min(0,a-y)}^{\max(0,a-y)}
e^{-|u|/b}f_Y(y+u)\,du  \\
&\quad-
\frac{b e^{-|y-a|/b}f_Y(a)}{2f_Y(y)} .
\end{aligned}
\]
\end{example}

\subsubsection{Gaussian approximation to posterior distribution functions}
\label{A:GaussianCdfApproximation}

The next lemma gives an explicit derivative expansion for a smooth
approximation to the posterior distribution function.

\begin{lemma}[Explicit form of the approximating functional for the posterior cdf]
\label{L:SeriesRepresentation}
Fix \(n\ge1\), \(x,y\in\mathbb R\), and \(\sigma>0\). Suppose \(d=1\) and
\[
V\sim \mathcal N(0,\sigma^2).
\]
Define
\[
g_{n,x}(u)
\coloneq
\Phi\left(
\frac{x+n^{-1/2}-u}{n^{-1}}
\right),
\qquad
u\in\mathbb R.
\]
Set
\[
s_n\coloneq \sqrt{\sigma^2+n^{-2}},
\qquad
a_n\coloneq \frac{x+n^{-1/2}-y}{s_n}.
\]
Then, for every \(f\in\mathcal A_V(\mathbb R)\),
\[
\mathcal T_{g_{n,x},V,y}[f]
=
\sum_{k=0}^{\infty}
\frac{(-1)^k}{k!}
\left(\frac{\sigma^2}{s_n}\right)^k
\Phi^{(k)}(a_n) f^{(k)}(y).
\]
For each fixed \(n\), the series converges absolutely.
\end{lemma}

\begin{proof}
Let
\[
\rho_n\coloneq \frac{\sigma}{s_n}.
\]
Since \(n\ge1\), we have \(0<\rho_n<1\). We use the probabilists' Hermite
polynomials
\[
\operatorname{He}_k(z)
=
(-1)^k e^{z^2/2}
\frac{d^k}{dz^k}e^{-z^2/2},
\qquad z\in\mathbb R.
\]

\smallskip
\noindent\emph{Step 1: Hermite expansion of the smooth cdf.}
For \(u\in\mathbb R\), set
\[
z=\frac{y-u}{\sigma}.
\]
Then
\[
\frac{x+n^{-1/2}-u}{n^{-1}}
=
\frac{a_n+\rho_n z}{\sqrt{1-\rho_n^2}}.
\]
Mehler's formula gives, for \(|\rho|<1\),
\[
\frac{\phi_{-\rho}(w,z)}{\phi(w)\phi(z)}
=
\sum_{k=0}^{\infty}
\frac{(-\rho)^k}{k!}
\operatorname{He}_k(w)\operatorname{He}_k(z),
\]
where \(\phi_{-\rho}\) is the standard bivariate normal density with correlation
\(-\rho\); see \cite{ismail2009classical}, Theorem~4.7.2. Therefore
\[
\int_{-\infty}^{a}
\frac{\phi_{-\rho}(w,z)}{\phi(z)}\,dw
=
\Phi\left(\frac{a+\rho z}{\sqrt{1-\rho^2}}\right).
\]

We justify termwise integration. By Cramer's bound for Hermite polynomials,
there is a constant \(C>0\) such that
\[
|\operatorname{He}_k(r)|
\le
C e^{r^2/4}\sqrt{k!},
\qquad r\in\mathbb R,\quad k\ge0;
\]
see \cite{abramowitz1972handbook}, Formula~22.14.17. Hence, for fixed
\(z\in\mathbb R\),
\[
\phi(w)
\left|
\frac{(-\rho)^k}{k!}
\operatorname{He}_k(w)\operatorname{He}_k(z)
\right|
\le
C^2 e^{z^2/4}\rho^k e^{-w^2/4}.
\]
Since \(\sum_{k\ge0}\rho^k<\infty\) and \(e^{-w^2/4}\) is integrable on
\((-\infty,a]\), Fubini's theorem gives
\[
\int_{-\infty}^{a}
\frac{\phi_{-\rho}(w,z)}{\phi(z)}\,dw
=
\sum_{k=0}^{\infty}
\frac{(-\rho)^k}{k!}
\operatorname{He}_k(z)
\int_{-\infty}^{a}
\operatorname{He}_k(w)\phi(w)\,dw .
\]

For \(k\ge1\), Rodrigues' formula implies
\[
\operatorname{He}_k(w)\phi(w)
=
(-1)^k\phi^{(k)}(w).
\]
Thus
\[
\int_{-\infty}^{a}
\operatorname{He}_k(w)\phi(w)\,dw
=
(-1)^k\phi^{(k-1)}(a)
=
(-1)^k\Phi^{(k)}(a).
\]
The same identity holds for \(k=0\), with \(\Phi^{(0)}=\Phi\). Therefore
\[
\Phi\left(\frac{a+\rho z}{\sqrt{1-\rho^2}}\right)
=
\sum_{k=0}^{\infty}
\frac{\rho^k}{k!}
\Phi^{(k)}(a)\operatorname{He}_k(z).
\]
Taking \(a=a_n\), \(\rho=\rho_n\), and \(z=(y-u)/\sigma\), we obtain
\[
g_{n,x}(u)
=
\sum_{k=0}^{\infty}
\frac{\rho_n^k}{k!}
\Phi^{(k)}(a_n)
\operatorname{He}_k\!\left(\frac{y-u}{\sigma}\right).
\]

\smallskip
\noindent\emph{Step 2: finite approximants.}
For \(M\ge0\), define
\[
g_{n,M,x,y}(u)
\coloneq
\sum_{k=0}^{M}
\frac{\rho_n^k}{k!}
\Phi^{(k)}(a_n)
\operatorname{He}_k\!\left(\frac{y-u}{\sigma}\right).
\]
Let
\[
\lambda_{n,M,x,y}(u)
\coloneq
g_{n,M,x,y}(u)f_V(y-u),
\qquad
\lambda_{n,x,y}(u)
\coloneq
g_{n,x}(u)f_V(y-u).
\]
We now show that
\[
\lambda_{n,M,x,y}\to \lambda_{n,x,y}
\qquad
\text{uniformly on }\mathbb R .
\]

Since
\[
f_V(y-u)
=
\frac1{\sigma}\phi\!\left(\frac{y-u}{\sigma}\right),
\]
it is enough to bound the terms
\[
\frac1{\sigma}\phi(z)
\frac{\rho_n^k}{k!}
\Phi^{(k)}(a_n)\operatorname{He}_k(z).
\]
For \(k\ge1\),
\[
\Phi^{(k)}(a_n)
=
\phi^{(k-1)}(a_n)
=
(-1)^{k-1}
\operatorname{He}_{k-1}(a_n)\phi(a_n).
\]
Cramer's bound gives
\[
|\Phi^{(k)}(a_n)|
\le
\frac{C}{\sqrt{2\pi}}
e^{-a_n^2/4}\sqrt{(k-1)!}.
\]
Hence, uniformly in \(z\in\mathbb R\),
\[
\frac1{\sigma}\phi(z)
\frac{\rho_n^k}{k!}
|\Phi^{(k)}(a_n)|
|\operatorname{He}_k(z)|
\le
\frac{C^2}{2\pi\sigma}
e^{-a_n^2/4}
\frac{\rho_n^k}{\sqrt{k}}.
\]
Since
\[
\sum_{k=1}^{\infty}\frac{\rho_n^k}{\sqrt{k}}<\infty,
\]
the Weierstrass \(M\)-test gives uniform absolute convergence; see
\cite{rudin1976principles}, Theorem~7.10. Thus
\[
\lambda_{n,M,x,y}\to \lambda_{n,x,y}
\quad
\text{uniformly,}
\]
and
\[
\sup_{M\ge0}
\|\lambda_{n,M,x,y}\|_\infty
<\infty .
\]
Moreover, each \(\lambda_{n,M,x,y}\) is a finite sum of polynomials times a
Gaussian density, and therefore belongs to \(C_0(\mathbb R;\mathbb C)\).

By Proposition~\ref{P:tweedie-approximation}, applied with \(M\) as the
approximation index,
\[
\mathcal T_{g_{n,x},V,y}[f]
=
\lim_{M\to\infty}
\mathcal T_{n,M,x,y}[f],
\qquad
f\in\mathcal A_V(\mathbb R),
\]
where \(\mathcal T_{n,M,x,y}\) denotes the Tweedie functional generated by
\(g_{n,M,x,y}\) at the point \(y\).

\smallskip
\noindent\emph{Step 3: explicit calculation for the finite approximants.}
For the Gaussian density,
\[
\frac{d^k}{du^k}f_V(y-u)
=
\frac{1}{\sigma^k}
\operatorname{He}_k\!\left(\frac{y-u}{\sigma}\right)
f_V(y-u).
\]
Therefore
\[
\operatorname{He}_k\!\left(\frac{y-u}{\sigma}\right)f_V(y-u)
=
\sigma^k\frac{d^k}{du^k}f_V(y-u).
\]
Using the Fourier convention of the paper,
\[
\widetilde{\frac{d^k}{du^k}f_V(y-\cdot)}(\omega)
=
(-i\omega)^k e^{i\omega y}e^{-\sigma^2\omega^2/2}.
\]
Since \(\varphi_V(-\omega)=e^{-\sigma^2\omega^2/2}\), the Fourier-domain
representer for \(\lambda_{n,M,x,y}\) is
\[
\mathcal Q_{n,M,x,y}(\omega)
=
\sum_{k=0}^{M}
\frac{1}{k!}
\left(
-\frac{\sigma^2}{s_n}
\right)^k
\Phi^{(k)}(a_n)
(i\omega)^k e^{i\omega y}.
\]
Thus Proposition~\ref{P:measure-factorization-functional}, applied with
smoothness order \(M\), gives
\[
\mathcal T_{n,M,x,y}[f]
=
\sum_{k=0}^{M}
\frac{1}{k!}
\left(
-\frac{\sigma^2}{s_n}
\right)^k
\Phi^{(k)}(a_n)
f^{(k)}(y).
\]

Letting \(M\to\infty\) yields
\[
\mathcal T_{g_{n,x},V,y}[f]
=
\sum_{k=0}^{\infty}
\frac{(-1)^k}{k!}
\left(\frac{\sigma^2}{s_n}\right)^k
\Phi^{(k)}(a_n)
f^{(k)}(y),
\]
provided the series is absolutely convergent. We prove this next.

\smallskip
\noindent\emph{Step 4: absolute convergence.}
Write
\[
f=f_V*\mu,
\qquad
\mu\in\mathcal M(\mathbb R).
\]
Since all derivatives of the Gaussian density are bounded and continuous,
differentiation under the integral gives
\[
f^{(k)}(y)
=
\int_{\mathbb R}
f_V^{(k)}(y-u)\,d\mu(u).
\]
By Rodrigues' formula,
\[
f_V^{(k)}(y-u)
=
\frac{(-1)^k}{\sigma^k}
\operatorname{He}_k\!\left(\frac{y-u}{\sigma}\right)
f_V(y-u).
\]
Thus
\[
f^{(k)}(y)
=
\frac{(-1)^k}{\sigma^k}
\int_{\mathbb R}
\operatorname{He}_k\!\left(\frac{y-u}{\sigma}\right)
f_V(y-u)\,d\mu(u).
\]
Consequently,
\[
\begin{aligned}
&\left|
\frac{(-1)^k}{k!}
\left(\frac{\sigma^2}{s_n}\right)^k
\Phi^{(k)}(a_n)f^{(k)}(y)
\right|  \\
&\qquad\le
\int_{\mathbb R}
\frac1{\sigma}
\phi\!\left(\frac{y-u}{\sigma}\right)
\frac{\rho_n^k}{k!}
|\Phi^{(k)}(a_n)|
\left|
\operatorname{He}_k\!\left(\frac{y-u}{\sigma}\right)
\right|
\,d|\mu|(u).
\end{aligned}
\]
The uniform bound from Step 2 gives, for \(k\ge1\),
\[
\left|
\frac{(-1)^k}{k!}
\left(\frac{\sigma^2}{s_n}\right)^k
\Phi^{(k)}(a_n)f^{(k)}(y)
\right|
\le
C_n\frac{\rho_n^k}{\sqrt{k}}\|\mu\|_{TV},
\]
where \(C_n<\infty\) does not depend on \(k\). Since
\[
\sum_{k=1}^{\infty}\frac{\rho_n^k}{\sqrt{k}}<\infty,
\]
the derivative series converges absolutely for each fixed \(n\). The \(k=0\)
term is finite, so the proof is complete.
\end{proof}

\subsubsection{Functionals under the heteroskedastic Gaussian sequence model}
\label{A:HeteroskedasticGaussianFunctionals}
\label{A:Examples}

This section records several Tweedie formulas for the heteroskedastic Gaussian
sequence model. 

\paragraph{Functionals in the multi-dimensional Gaussian model}
\label{A:HeteroGaussianNonsmooth}

Fix
\[
\Sigma_0\in\mathbb S_{++}^d,
\qquad
y\in\mathbb R^d,
\]
and write
\[
B\coloneq \Sigma_0^{1/2},
\qquad
z_0\coloneq B^{-1}y.
\]
We work in the standardized model
\[
Z=W+V,
\qquad
V\sim\mathcal N(0,I_d),
\]
from Proposition~\ref{P: conditional-heteroskedastic}. Let
\[
f_Z(z)
\coloneq
f_{Z\mid\Sigma}(z\mid\Sigma_0),
\qquad
f_{\Sigma_0}(u)
\coloneq
f_{Y\mid\Sigma}(u\mid\Sigma_0).
\]
Then
\[
f_Z(z)
=
\det(B)\,f_{\Sigma_0}(Bz).
\]

We use the following elementary Gaussian Fourier identities. Let
\[
h_{z_0}(w)\coloneq \phi(z_0-w).
\]
Since
\[
\widetilde h_{z_0}(\omega)
=
e^{i\omega^\top z_0}e^{-\|\omega\|_2^2/2},
\]
and
\[
\varphi_V(-\omega)=e^{-\|\omega\|_2^2/2},
\]
we have
\[
\frac{\widetilde{B w h_{z_0}}(\omega)}
{\varphi_V(-\omega)}
=
e^{i\omega^\top z_0}(y+iB\omega),
\]
and
\[
\frac{\widetilde{Bww^\top B h_{z_0}}(\omega)}
{\varphi_V(-\omega)}
=
e^{i\omega^\top z_0}
B\{(z_0+i\omega)(z_0+i\omega)^\top+I_d\}B .
\]
Equivalently,
\[
e^{i\omega^\top z_0}
B\{(z_0+i\omega)(z_0+i\omega)^\top+I_d\}B
=
e^{i\omega^\top z_0}
\left[
yy^\top+\Sigma_0
+i\,y(B\omega)^\top
+i\,(B\omega)y^\top
-B\omega\omega^\top B
\right].
\]

\begin{example}[Posterior mean]
\label{E:HeteroGaussianMean}
Let
\[
g(x)=x.
\]
The Fourier-domain representer is
\[
\mathcal Q_{g,y,\Sigma_0}(\omega)
=
e^{i\omega^\top z_0}(y+iB\omega).
\]
This is a finite linear combination of
\[
\varphi_{\delta_{z_0}}(\omega)
\quad\text{and}\quad
(i\omega_\ell)\varphi_{\delta_{z_0}}(\omega),
\qquad
\ell=1,\ldots,d.
\]
Therefore Proposition~\ref{P:measure-factorization-functional} applies
coordinatewise with smoothness order \(1\), and gives
\[
\mathcal T_{g,y,\Sigma_0}[f]
=
y f(z_0)+B\nabla f(z_0),
\qquad
f\in\mathcal A_V(\mathbb R^d).
\]
Taking \(f=f_Z\) yields
\[
\mathbb E[X\mid Y=y,\Sigma=\Sigma_0]
=
y+B\nabla_z\log f_Z(z_0).
\]
Since
\[
\nabla_z\log f_Z(z_0)
=
B\nabla_y\log f_{\Sigma_0}(y),
\]
we obtain the usual heteroskedastic Tweedie formula:
\[
\boxed{
\mathbb E[X\mid Y=y,\Sigma=\Sigma_0]
=
y+\Sigma_0\nabla_y\log f_{\Sigma_0}(y).
}
\]
\end{example}

\begin{example}[Posterior covariance]
\label{E:HeteroGaussianCovariance}
Let
\[
g(x)=xx^\top.
\]
The Fourier-domain representer is
\[
\mathcal Q_{g,y,\Sigma_0}(\omega)
=
e^{i\omega^\top z_0}
\left[
yy^\top+\Sigma_0
+i\,y(B\omega)^\top
+i\,(B\omega)y^\top
-B\omega\omega^\top B
\right].
\]
This symbol is a finite linear combination of
\[
\varphi_{\delta_{z_0}}(\omega),
\qquad
(i\omega_\ell)\varphi_{\delta_{z_0}}(\omega),
\qquad
(i\omega_\ell)(i\omega_m)\varphi_{\delta_{z_0}}(\omega).
\]
Thus Proposition~\ref{P:measure-factorization-functional} applies entrywise
with smoothness order \(2\), and gives
\[
\begin{aligned}
\mathcal T_{g,y,\Sigma_0}[f]
&=
(yy^\top+\Sigma_0)f(z_0)  \\
&\quad
+
y\{B\nabla f(z_0)\}^\top
+
\{B\nabla f(z_0)\}y^\top
+
B\nabla^2 f(z_0)B .
\end{aligned}
\]
Taking \(f=f_Z\), dividing by \(f_Z(z_0)\), and subtracting the outer product
of the posterior mean from Example~\ref{E:HeteroGaussianMean} gives
\[
\operatorname{Var}(X\mid Y=y,\Sigma=\Sigma_0)
=
\Sigma_0
+
B\nabla_z^2\log f_Z(z_0)B.
\]
Since
\[
\nabla_z^2\log f_Z(z_0)
=
B\nabla_y^2\log f_{\Sigma_0}(y)B,
\]
we obtain
\[
\boxed{
\operatorname{Var}(X\mid Y=y,\Sigma=\Sigma_0)
=
\Sigma_0
+
\Sigma_0\nabla_y^2\log f_{\Sigma_0}(y)\Sigma_0 .
}
\]
\end{example}

\begin{example}[Posterior moment generating function]
\label{E:HeteroGaussianMGF}
Fix \(t\in\mathbb R^d\), and let
\[
g_t(x)=e^{t^\top x}.
\]
Then
\[
\lambda_{g_t,y,\Sigma_0}(w)
=
e^{t^\top Bw}\phi(z_0-w).
\]
Completing the square gives
\[
e^{t^\top Bw}\phi(z_0-w)
=
\exp\!\left(t^\top y+\frac12 t^\top\Sigma_0 t\right)
\phi(z_0+Bt-w).
\]
Therefore
\[
\mathcal Q_{g_t,y,\Sigma_0}(\omega)
=
\exp\!\left(t^\top y+\frac12 t^\top\Sigma_0 t\right)
e^{i\omega^\top(z_0+Bt)}.
\]
Proposition~\ref{P:measure-factorization-functional} applies with smoothness
order \(0\), represented by the measure
\[
\mu_0
=
\exp\!\left(t^\top y+\frac12 t^\top\Sigma_0 t\right)
\delta_{z_0+Bt}.
\]
Thus
\[
\mathcal T_{g_t,y,\Sigma_0}[f]
=
\exp\!\left(t^\top y+\frac12 t^\top\Sigma_0 t\right)
f(z_0+Bt).
\]
Taking \(f=f_Z\) gives
\[
\mathbb E[e^{t^\top X}\mid Y=y,\Sigma=\Sigma_0]
=
\exp\!\left(t^\top y+\frac12 t^\top\Sigma_0 t\right)
\frac{f_Z(z_0+Bt)}{f_Z(z_0)}.
\]
Changing back to the original \(Y\)-scale,
\[
\boxed{
\mathbb E[e^{t^\top X}\mid Y=y,\Sigma=\Sigma_0]
=
\exp\!\left(t^\top y+\frac12 t^\top\Sigma_0 t\right)
\frac{f_{\Sigma_0}(y+\Sigma_0 t)}{f_{\Sigma_0}(y)}.
}
\]
\end{example}

\paragraph{Functionals in the one-dimensional Gaussian model}
\label{A:HeteroGaussianNonsmoothOneDim}

Assume \(d=1\), write
\[
\Sigma_0=\sigma^2,
\qquad
\sigma>0,
\]
and set
\[
f_\sigma(u)
\coloneq
f_{Y\mid\Sigma}(u\mid\sigma^2),
\qquad
f_\sigma^{(k)}(u)
\coloneq
\partial_u^k f_{Y\mid\Sigma}(u\mid\sigma^2).
\]
All formulas below are stated for points \(y\) such that
\[
f_\sigma(y)>0.
\]

For \(n\ge1\) and \(c\in\mathbb R\), define
\[
s_n
\coloneq
\sqrt{\sigma^2+n^{-2}},
\qquad
q_{n,c}
\coloneq
\frac{c+n^{-1/2}-y}{s_n},
\qquad
q_c
\coloneq
\frac{c-y}{\sigma}.
\]
We use the convention \(\Phi^{(0)}=\Phi\). For
\(r\in\{0,1\}\) and \(m\ge r\), define the Gaussian series operator
\[
\mathsf S_{n,c}^{(r,m)}[f](y)
\coloneq
\sum_{k=m}^{\infty}
\frac{(-1)^k}{k!}
\left(\frac{\sigma^2}{s_n}\right)^k
s_n^r
\Phi^{(k-r)}(q_{n,c})
f^{(k)}(y).
\]
For each fixed \(n\), the series is absolutely convergent in the cases used
below by Lemma~\ref{L:SeriesRepresentation}. The parameter \(r\) records how
many times the smoothed cdf has been integrated: \(r=0\) gives the cdf series,
while \(r=1\) gives the integrated series used for hinge losses.

We also define
\[
\mathsf A_{n,c}(y)
\coloneq
s_n\{q_{n,c}\Phi(q_{n,c})+\phi(q_{n,c})\}.
\]
This is the zeroth-order term that appears when the smoothed cdf is integrated
over \((-\infty,c]\).

\begin{example}[Posterior moments in dimension one]
\label{E:HeteroGaussianMoments}
Assume \(d=1\), and write
\[
\Sigma_0=\sigma^2,
\qquad
\sigma>0.
\]
Let
\[
f_\sigma(u)
\coloneq
f_{Y\mid\Sigma}(u\mid\sigma^2).
\]
By Example~\ref{E:HeteroGaussianMGF},
\[
\mathbb E[e^{tX}\mid Y=y,\Sigma=\sigma^2]
=
\exp\!\left(ty+\frac12\sigma^2t^2\right)
\frac{f_\sigma(y+\sigma^2t)}{f_\sigma(y)}.
\]

For \(m\ge0\), define
\[
\mathsf H_m(a,\sigma^2)
\coloneq
\left.
\frac{d^m}{dt^m}
\exp\!\left(at+\frac12\sigma^2t^2\right)
\right|_{t=0}.
\]
Equivalently,
\[
\mathsf H_m(a,\sigma^2)
=
\sum_{j=0}^{\lfloor m/2\rfloor}
\frac{m!}{2^j j!(m-2j)!}
\sigma^{2j}a^{m-2j}.
\]
Differentiating the posterior moment generating function at \(t=0\) gives
\[
\boxed{
\mathbb E[X^m\mid Y=y,\Sigma=\sigma^2]
=
\sum_{r=0}^{m}
\binom{m}{r}
\sigma^{2r}
\mathsf H_{m-r}(y,\sigma^2)
\frac{f_\sigma^{(r)}(y)}{f_\sigma(y)}.
}
\]

Now let
\[
\kappa(y,\sigma^2)
\coloneq
\mathbb E[X\mid Y=y,\Sigma=\sigma^2]
=
y+\sigma^2\frac{f_\sigma'(y)}{f_\sigma(y)}.
\]
The centered posterior moment follows by replacing \(y\) with
\(y-\kappa(y,\sigma^2)\) in the exponential factor:
\[
\boxed{
\mathbb E\!\left[
\{X-\kappa(y,\sigma^2)\}^m
\mid Y=y,\Sigma=\sigma^2
\right]
=
\sum_{r=0}^{m}
\binom{m}{r}
\sigma^{2r}
\mathsf H_{m-r}(y-\kappa(y,\sigma^2),\sigma^2)
\frac{f_\sigma^{(r)}(y)}{f_\sigma(y)}.
}
\]
\end{example}

\begin{example}[Posterior even-power risk in dimension one]
\label{E:HeteroGaussianEvenRisk}
Assume \(d=1\), write
\[
\Sigma_0=\sigma^2,
\qquad
\sigma>0,
\]
and fix \(a\in\mathbb R\). Let \(q=2m\), where \(m\in\mathbb N\). Since
\[
|X-a|^{2m}=(X-a)^{2m},
\]
we differentiate the shifted posterior moment generating function
\[
\mathbb E[e^{t(X-a)}\mid Y=y,\Sigma=\sigma^2]
=
\exp\!\left((y-a)t+\frac12\sigma^2t^2\right)
\frac{f_\sigma(y+\sigma^2t)}{f_\sigma(y)}.
\]
Using the polynomial \(\mathsf H_m\) from
Example~\ref{E:HeteroGaussianMoments}, we obtain
\[
\boxed{
\mathbb E[|X-a|^{2m}\mid Y=y,\Sigma=\sigma^2]
=
\sum_{r=0}^{2m}
\binom{2m}{r}
\sigma^{2r}
\mathsf H_{2m-r}(y-a,\sigma^2)
\frac{f_\sigma^{(r)}(y)}{f_\sigma(y)}.
}
\]
\end{example}

\begin{example}[Posterior distribution function]
\label{E:HeteroGaussianCDF}
Fix \(x\in\mathbb R\). For \(n\ge1\), define
\[
g_{n,x}(u)
\coloneq
\Phi\left(
\frac{x+n^{-1/2}-u}{n^{-1}}
\right).
\]
Then
\[
g_{n,x}(u)\to \mathbf 1_{\{u\le x\}}
\qquad
\text{for every }u\in\mathbb R.
\]
By Lemma~\ref{L:SeriesRepresentation},
\[
\mathcal T_{g_{n,x},V,y}[f_\sigma]
=
\mathsf S_{n,x}^{(0,0)}[f_\sigma](y).
\]
Separating the zeroth term gives
\[
\mathsf S_{n,x}^{(0,0)}[f_\sigma](y)
=
\Phi(q_{n,x})f_\sigma(y)
+
\mathsf S_{n,x}^{(0,1)}[f_\sigma](y).
\]
Since \(0\le g_{n,x}\le1\), dominated convergence gives
\[
\boxed{
\mathbb P(X\le x\mid Y=y,\Sigma=\sigma^2)
=
\Phi(q_x)
+
\frac{1}{f_\sigma(y)}
\lim_{n\to\infty}
\mathsf S_{n,x}^{(0,1)}[f_\sigma](y).
}
\]
Equivalently,
\[
\mathbb P(X\le x\mid Y=y,\Sigma=\sigma^2)
=
\Phi\left(\frac{x-y}{\sigma}\right)
+
\frac{1}{f_\sigma(y)}
\lim_{n\to\infty}
\sum_{k=1}^{\infty}
\frac{(-1)^k}{k!}
\left(\frac{\sigma^2}{s_n}\right)^k
\Phi^{(k)}(q_{n,x})
f_\sigma^{(k)}(y).
\]
\end{example}

\begin{example}[Posterior left hinge]
\label{E:HeteroGaussianLeftHinge}
Fix \(a\in\mathbb R\). We first compute the left hinge
\[
(a-X)_+.
\]
For \(n\ge1\), define
\[
\ell_{n,a}(u)
\coloneq
\int_{-\infty}^{a}
\Phi\left(
\frac{t+n^{-1/2}-u}{n^{-1}}
\right)dt .
\]
Then
\[
\ell_{n,a}(u)\to(a-u)_+
\qquad
\text{for every }u\in\mathbb R.
\]
Moreover, \(\ell_{n,a}(u)\phi_\sigma(y-u)\) is continuous, vanishes at
infinity, and is uniformly bounded in \(n\). Hence
Proposition~\ref{P:tweedie-approximation} applies.

Integrating the expansion in Lemma~\ref{L:SeriesRepresentation} over
\(t\le a\) gives
\[
\mathcal T_{\ell_{n,a},V,y}[f_\sigma]
=
\mathsf A_{n,a}(y)f_\sigma(y)
+
\mathsf S_{n,a}^{(1,1)}[f_\sigma](y).
\]
Since
\[
\mathsf S_{n,a}^{(1,1)}[f_\sigma](y)
=
-\sigma^2\Phi(q_{n,a})f_\sigma'(y)
+
\mathsf S_{n,a}^{(1,2)}[f_\sigma](y),
\]
we obtain
\[
\mathcal T_{\ell_{n,a},V,y}[f_\sigma]
=
\mathsf A_{n,a}(y)f_\sigma(y)
-
\sigma^2\Phi(q_{n,a})f_\sigma'(y)
+
\mathsf S_{n,a}^{(1,2)}[f_\sigma](y).
\]
Letting \(n\to\infty\) yields
\[
\boxed{
\mathbb E[(a-X)_+\mid Y=y,\Sigma=\sigma^2]
=
\sigma\{q_a\Phi(q_a)+\phi(q_a)\}
-
\sigma^2\Phi(q_a)
\frac{f_\sigma'(y)}{f_\sigma(y)}
+
\frac{1}{f_\sigma(y)}
\lim_{n\to\infty}
\mathsf S_{n,a}^{(1,2)}[f_\sigma](y).
}
\]
\end{example}

\begin{example}[Posterior right hinge]
\label{E:HeteroGaussianRightHinge}
Fix \(a\in\mathbb R\). Since
\[
(X-a)_+
=
X-a+(a-X)_+,
\]
we combine Example~\ref{E:HeteroGaussianLeftHinge} with the Gaussian Tweedie
formula
\[
\mathbb E[X\mid Y=y,\Sigma=\sigma^2]
=
y+\sigma^2\frac{f_\sigma'(y)}{f_\sigma(y)}.
\]
This gives
\[
\boxed{
\begin{aligned}
\mathbb E[(X-a)_+\mid Y=y,\Sigma=\sigma^2]
&=
\sigma\phi(q_a)
+
(y-a)\{1-\Phi(q_a)\}  \\
&\quad
+
\sigma^2\{1-\Phi(q_a)\}
\frac{f_\sigma'(y)}{f_\sigma(y)}
+
\frac{1}{f_\sigma(y)}
\lim_{n\to\infty}
\mathsf S_{n,a}^{(1,2)}[f_\sigma](y).
\end{aligned}
}
\]
\end{example}

\begin{example}[Posterior absolute risk]
\label{E:HeteroGaussianAbsoluteRisk}
Fix \(a\in\mathbb R\). Since
\[
|X-a|
=
2(X-a)_+-(X-a),
\]
Example~\ref{E:HeteroGaussianRightHinge} gives
\[
\boxed{
\begin{aligned}
\mathbb E[|X-a|\mid Y=y,\Sigma=\sigma^2]
&=
2\sigma\phi(q_a)
+
(a-y)\{2\Phi(q_a)-1\}  \\
&\quad
+
\sigma^2\{1-2\Phi(q_a)\}
\frac{f_\sigma'(y)}{f_\sigma(y)}
+
\frac{2}{f_\sigma(y)}
\lim_{n\to\infty}
\mathsf S_{n,a}^{(1,2)}[f_\sigma](y).
\end{aligned}
}
\]
\end{example}

\begin{example}[Posterior pinball loss]
\label{E:HeteroGaussianPinballLoss}
Fix \(a\in\mathbb R\) and \(\tau\in(0,1)\). Let
\[
\rho_\tau(z)
=
\tau z_+
+
(1-\tau)(-z)_+.
\]
Then
\[
\rho_\tau(X-a)
=
\tau(X-a)_+
+
(1-\tau)(a-X)_+.
\]
Combining Examples~\ref{E:HeteroGaussianLeftHinge} and
\ref{E:HeteroGaussianRightHinge} yields
\[
\boxed{
\begin{aligned}
\mathbb E[\rho_\tau(X-a)\mid Y=y,\Sigma=\sigma^2]
&=
\sigma\phi(q_a)
+
\tau(y-a)\{1-\Phi(q_a)\}  \\
&\quad
+
(1-\tau)(a-y)\Phi(q_a)  \\
&\quad
+
\sigma^2\{\tau-\Phi(q_a)\}
\frac{f_\sigma'(y)}{f_\sigma(y)}
+
\frac{1}{f_\sigma(y)}
\lim_{n\to\infty}
\mathsf S_{n,a}^{(1,2)}[f_\sigma](y).
\end{aligned}
}
\]
\end{example}

\end{document}